\documentclass[11pt,reqno]{article}
\usepackage{babel}
\usepackage[en-GB]{datetime2}
\usepackage[utf8]{inputenc}
\usepackage[hidelinks]{hyperref}
\usepackage{enumitem}
\usepackage{xcolor}
\usepackage{caption}
\usepackage{comment}
\usepackage{graphicx}
\usepackage{bbm}
\usepackage{mathtools}
\usepackage[right=1in, left=1in, top=1in, bottom=1in]{geometry}
\usepackage{soul,xcolor}
\setstcolor{red}

\usepackage{amsmath}
\usepackage{amsthm}
\usepackage{amssymb}
\usepackage{thmtools}
\usepackage{thm-restate}

\makeatletter
\renewcommand*{\eqref}[1]{%
  \hyperref[{#1}]{\textup{\tagform@{\ref*{#1}}}}%
}
\makeatother

\newtheorem{theorem}{Theorem}[section]
\newtheorem{lemma}[theorem]{Lemma}

\newtheorem{conj}[theorem]{Conjecture}

\newtheorem{question}{Question}[theorem]
\newtheorem{fact}[theorem]{Fact}

\theoremstyle{definition}
\newtheorem{claim}{Claim}[theorem]
\newtheorem{defn}[theorem]{Definition}
\theoremstyle{remark}

\newcommand{\eps}{\varepsilon}
\renewcommand{\c}{\mathcal}
\renewcommand{\P}{\mathbb{P}}
\newcommand{\E}{\mathbb{E}}

\newcommand{\sm}{\setminus}
\newcommand{\ora}{\overrightarrow}
\newcommand{\ola}{\overleftarrow}

  \title{Hamilton decompositions of regular tripartite tournaments}
\author{Francesco Di Braccio\thanks{Department of Mathematics, London School of Economics and Political Science, London WC2A 2AE, UK.} \and Joanna Lada$^*$ \and Viresh Patel\thanks{School of Mathematical Sciences, Queen Mary University of London, Mile
End Road, London E1 4NS, UK.} \and Yani Pehova$^*$\thanks{Supported by EPSRC Fellowship EP/V038168/1.} \and Jozef Skokan$^*$}
\date{16th July 2025}

\begin{document}
\hyphenation{counterclock-wise}

\maketitle
\makeatletter{\renewcommand*\@makefnmark{}\footnotetext{Contact: \texttt{\{f.di-braccio|j.m.lada|j.skokan\}@lse.ac.uk}, \texttt{viresh.patel@qmul.ac.uk},~\texttt{yani.pehova@gmail.com.}}\makeatother}

\abstract{

K\"uhn and Osthus conjectured in 2013 that regular tripartite tournaments are decompo\-sa\-ble into Hamilton cycles. Somewhat surprisingly, Granet gave a simple counterexample to this conjecture almost 10 years later. In this paper, we show that the conjecture of K\"uhn and Osthus nevertheless holds in an approximate sense, by proving that every regular tripartite tournament admits an approximate decomposition into Hamilton cycles. We also study Hamilton cycle packings of directed graphs in the same regime, and show that for large $n$, every balanced tripartite digraph on $3n$ vertices which is $d$-regular for $d\ge (1+o(1))n$ admits a Hamilton decomposition.}

\section{Introduction}\label{sec:intro}

Graph decompositions and packings form a central area of study in extremal combinatorics (see \cite{glock2021extremal, kuhn_survey} for some surveys on the topic). A fundamental question in this area concerns whether a graph can be decomposed into Hamilton cycles. A graph has a \emph{Hamilton decomposition} if it has a collection of edge-disjoint Hamilton cycles that together cover all the edges of the graph. An obvious necessary condition for a graph to admit a Hamilton decomposition is that it is regular. Among regular graphs, however, determining Hamilton decomposability is rarely straightforward, even in the simplest, most symmetric graphs. In 1892, Walecki \cite{walecki} showed that the complete graph $K_n$ can be decomposed into Hamilton cycles when $n$ is odd. 
More recently, Csaba, K\"uhn, Lo, and Osthus ~\cite{csaba2016} proved a much stronger result which says that, for large $n$, every $d$-regular $n$-vertex graph with $d\ge \lfloor n/2\rfloor$ can be decomposed into Hamilton cycles and at most one perfect matching (resolving the long-standing Hamilton decomposition conjecture of Nash-Williams).

Recall that in a \emph{directed graph} (or \emph{digraph}), edges have a direction and we allow up to two edges between any pair of vertices (one in each direction) and no loops, whereas an \emph{oriented graph} is a digraph where we allow at most one edge between each pair of vertices. An \emph{orientation} of a (undirected) graph $G$ is an oriented graph obtained by giving a direction to each of $G$'s edges. A digraph is \emph{$d$-regular} if every vertex has in and outdegree $d$. In the oriented setting, Kelly's conjecture was a longstanding open problem stating that every regular tournament  (i.e. regular orientation of $K_n$) has a Hamilton decomposition. This was proved for large $n$ by K\"uhn and Osthus in \cite{kelly-published}; in fact they proved the stronger result that, for large $n$, any $d$-regular, $n$-vertex oriented graph with $d\ge (3/8+o(1))n$ has a Hamilton decomposition. 

In this paper, we are interested in variants of Kelly's conjecture for partite tournaments. Let $K_r(n)$ be the (undirected) complete $r$-partite graph with $n$ vertices in each class (also called the $n$-blow-up of $K_r$). It is known that $K_r(n)$ is Hamilton decomposable when its degree is even~\cite{hetyei,laskar}. Also, the directed version of $K_r(n)$ (where one replaces each edge $xy$ of $K_r(n)$ with the two corresponding directed edges $(x,y)$ and $(y,x)$) is Hamilton decomposable unless $(r,n)\in \{(4,1),(6,1)\}$~\cite{Ng}. As with Kelly's conjecture, the oriented setting here is considerably more challenging and exhibits interesting behaviour. A \emph{regular $r$-partite tournament} is a regular orientation of $K_r(n)$ for some $n \geq 1$. Large regular $r$-partite tournaments for $r \geq 5$ admit a Hamilton decomposition by a simple application of the main result from \cite{kelly-published} mentioned above (noting that such graphs $G$ are $d$-regular for $d=\frac{(r-1)n}{2}\ge \frac{2}{5}|G| \geq \frac{3}{8}|G|$). With slightly more effort, one can show using the same techniques that large regular $4$-partite tournaments are Hamilton decomposable (see \cite[Corollary 1.13]{kuhn2014}).

The two remaining cases $r = 2,3$, however, present a real challenge. In the case $r=2$, Jackson~\cite{jackson} conjectured in 1981 that every regular bipartite tournament has a Hamilton decomposition; this was proved approximately by Liebenau and Pehova in \cite{liebenau-published} and recently Granet \cite{bertille} completely resolved the conjecture for sufficiently large $n$.
In the remaining case $r=3$, K{\"u}hn and Osthus conjectured~\cite[Conjecture 1.14]{kuhn2014} that every regular tripartite tournament has a Hamilton decomposition. However, Granet~\cite{bertille} found a simple counter-example to this conjecture: the regular tripartite tournament obtained from the $n$-blow-up of the cyclic triangle (denoted $\ora{C_3}(n)$) by reversing the edges of a single triangle. One can argue that no Hamilton cycle in this tournament can possibly contain any of the edges of this triangle (see \cite[Proposition 1.1]{bertille}), and thus any collection of edge-disjoint Hamilton cycles necessarily misses at least three edges. 

The main result of this paper is to show that every regular tripartite tournament can be packed with Hamilton cycles covering all but an arbitrarily small proportion of the edges, thus confirming an approximate version of the K{\"u}hn-Osthus conjecture and showing that the counterexample above poses only a mild obstruction. 

\begin{theorem}\label{conj:ori-almost}
    For every $\delta > 0$ and sufficiently large $n$, every regular tripartite tournament on $3n$ vertices contains $(1-\delta)n$ edge-disjoint Hamilton cycles.  
\end{theorem}
Since these tournaments are $n$-regular, \autoref{conj:ori-almost} implies that all but $3\delta n^2$ edges can be covered with edge-disjoint Hamilton cycles.

%For tripartite digraphs, we derive the following result from the main result of \cite{kelly-published}.
Our second result concerns the directed setting. We determine asymptotically the degree threshold for a regular tripartite digraph to have a full Hamilton decomposition. 
%past the same degree threshold are fully Hamilton-decomposable, improving substantially the result from \cite{Ng} for $r=3$. 
\begin{theorem}\label{thm:di}
    For every $\eps>0$ and sufficiently large $n$, every balanced tripartite digraph on $3n$ vertices which is $d$-regular for some $d\ge (1+\eps) n$ has a Hamilton decomposition. 
\end{theorem}
Note that the disjoint union of two balanced complete tripartite digraphs on $3n/2$ vertices each (for $n$ even) forms a disconnected $n$-regular $3n$-vertex tripartite digraph. This shows that the above is asymptotically tight in a strong sense, since, for $d= n$, there exist $d$-regular tripartite digraphs without a single Hamilton cycle.
% Quick proof sketch

\paragraph{Proof Outline.}To prove \autoref{conj:ori-almost}, we first examine the expansion properties of our directed or oriented graph $G$. From \cite{kelly-published}, it is known that regular digraphs belonging to the family of \emph{robust outexpanders} are decomposable into Hamilton cycles, so it suffices to consider the case when $G$ is not a robust outexpander. A result from \cite{LPY23} combined with a careful structural analysis shows that in this case $G$ is close in terms of edit-distance to a family of highly structured non-expanding regular tripartite tournaments, which we denote by $\c G_\beta$. The parameter $\beta$ varies from $0$ to $1/2$, with $\c G_0=\{\ora{C_3}(n)\}$ and other $\c G_\beta$ for $\beta>0$ displaying a richer structure. 

In the main part of the paper, we use the structure of $\c G_\beta$ to show how to pack Hamilton cycles into almost all of $G$. Our approach builds on methods and tools from \cite{ferber-published, LPY23, liebenau-published, kelly-published, bertille}, but these fall short in several regards and many
novel ideas are needed. As is common in this area, our general strategy involves finding a collection of edge-disjoint near-spanning linear forests and then closing them into Hamilton cycles. However, in our setting, this second step is particularly delicate and demands that the linear forests satisfy certain key properties. First, their edges must be uniformly distributed among the three (oriented) bipartite subgraphs that compose the tripartite tournament (see the notion of \emph{bidirectional balanced}-ness in \autoref{sec:ori:g_beta}). Second, the tournament may contain a small set of vertices incident with many edges not obeying the structure of $\c G_\beta$ (recall that $G$ is only close in edit-distance to the family); these vertices must be incorporated into each linear forest prior to completion. The main technical work consists in constructing linear forests with these properties and showing that these conditions are sufficient to complete them into Hamilton cycles.

% Paper organisation
\paragraph{Organization.}The rest of this paper is organised as follows. In \autoref{sec:prelim}, we give notation and some basic probabilistic tools. As the proof of \autoref{thm:di} is simpler, we start with it in \autoref{sec:directed}. We then proceed to prove \autoref{conj:ori-almost} in \autoref{sec:oriented}. In section \autoref{subset:struct}, we pin down the structure of non-expanding regular tripartite tournaments. \autoref{sec:ori:g_beta}, our longest section, contains the proof that such graphs admit an approximate Hamilton decomposition. Some concluding remarks and open problems are given in \autoref{sec:conclusion}.

% Intuitively, $G$ is expanding if, for most vertex subsets $S \subseteq V(G)$, the robust outneighbourhood of $S$, that is the set of vertices that have many inneighbours in $S$, is significantly larger than $S$. The blow-up of the cyclic triangle $\ora{C_3}(n)$ is not an expander because the outneighbourhood of any vertex class is also a vertex class, and so of the same size. 

\section{Preliminaries}\label{sec:prelim}

\subsection{Notation}\label{sec:prelim:notation}

We follow standard graph theoretic notation. Any digraph specific notation is clarified in this section.

\paragraph{Digraphs.}Given a directed graph $G=(V,E)$ with vertex set $V$ and edge set $E$, we write $G[A,B]$ to denote the subgraph of $G$ on vertex set $A \cup B$ and edge set $E_G(A,B)$, i.e. the set of all edges whose invertex is in $A$ and outvertex in $B$. We write $e_G(A,B)$ to denote $|E_G[A,B]|$.

Throughout the text, we apply set theoretic notation directly to digraphs to refer to the digraphs obtained by the corresponding operations applied to their edge sets. So, for instance, given digraphs $G$ and $H$, $G \cup H, G \cap H, G \setminus H$ refer to the digraphs on vertex set $V(G) \cup V(H)$ and edge set $E(G) \cup E(H), E(G) \cap E(H), E(G) \setminus E(H)$. Similarly, $G \subseteq H$ means $E(G) \subseteq E(H)$.  

We write $N^+_G(v)=\{w\in V(G):~vw\in E(G)\}$ for the \emph{outneighbourhood} of $v$ and $N^-_G(v)=\{w\in V(G):~wv\in E(G)\}$ for the \emph{inneighbourhood} of $v$. We also have the in and outdegrees $d^+_G(v)=|N^+_G(v)|$ and $d^-_G(v)=|N^-_G(v)|.$ We write $N^\pm$ to indicate that a particular property holds both for in and outneighbourhoods. We use $d^\pm$ similarly. We will often find it convenient to refer to neighbourhoods of vertices to particular sets $A\subseteq V(G)$ by $N^\pm_G(v,A)=N^\pm_G(v)\cap A$ and $d^\pm_G(v,A)=|N^\pm_G(v,A)|.$ This definition applies analogously to simple graphs by dropping the $\pm$. As is standard in the literature, we use $\delta^+, \Delta^+$ and $\delta^-, \Delta^+$ to denote minimum and maximum out and indegree, respectively, and $\delta^0=\min\{\delta^+,\delta^-\}, \Delta^0 = \max\{\Delta^+, \Delta^-\}$ for the minimum and maximum \emph{semidegree}.

We write $\ora{P_i}$ to denote the directed path on $i$ edges and $\ora{C_i}$ for the directed cycle on $i$ edges. We write $\ora{C_i}(n)$ for the \emph{$n$-blow-up} of $\ora{C_i}$, i.e. the digraph obtained by replacing each vertex of $\ora{C_i}$ with an independent set of size $n$. 

\paragraph{Linear forests.} A \emph{linear forest} is a digraph consisting of vertex-disjoint directed paths. Given a linear forest $\mathcal{F}$, $V(\c F)$ is the set of vertices that are incident with an edge in $\c F$. If $S$ is a vertex set, we write $S^+(\mathcal{F})$ for the set of vertices $v \in S$ such that $d^+_{\c F}(v) = 0$, and analogously we write $S^-(\c F)$ for those satisfying $d^-_{\c F}(v)= 0$. A vertex is a path \emph{startpoint} (resp. \emph{endpoint}) in $\c F$ if it is incident with an outedge in $\c F$ but not an inedge (an inedge but not an outedge). Note that if $S \subseteq V(\c F)$, each vertex in $S^+(\c F)$ is a path endpoint in $\c F$ and each vertex in $S^-(\c F)$ is a startpoint.

\paragraph{Regular tripartite tournaments.} Let $G$ be a regular tripartite tournament with par\-ti\-tion $(V_1, V_2, V_3)$. The \emph{clockwise} edges in $G$ are those in $E_G(V_1,V_2)\cup E_G(V_2,V_3)\cup E_G(V_3,V_1)$, whereas the \emph{counterclockwise} edges are those in $E_G(V_1,
V_3)\cup E_G(V_3, V_2)\cup E_G(V_2, V_1)$. 

We say that two tripartite digraphs $G$ and $H$ on the same vertex set $V$ are \emph{$\eps$-close} if their vertex partitions are the same and $|E(G)\triangle E(H)|\le \eps |V|^2$. We say that $G$ is $\eps$-close to a family of tripartite digraphs $\c G$ if there is some $H \in \c G$ that is $\eps$-close to $G$.

\paragraph{Asymptotic notation.} Our asymptotic notation is mostly standard. Given $n \geq 1$ and $a_1, \dots, a_t \in \mathbb{R}$, we write that a certain quantity is $O_{a_1, \dots,a_t}(f(n))$ if it is bounded above by $C f(n)$ for some $C = C(a_1, \dots, a_t)$. If $a_1, \dots, a_t$ are not specified, we only require that $C$ does not depend on $n$. We write that a quantity is $o(f(n))$ if it is bounded above by some $g(n)$ with $g(n)/f(n)\to 0$. We define $\Omega(\cdot)$ and $\omega(\cdot)$ analogously but for lower bounds.

We will often make assumptions of the form $a \ll b_1, \dots, b_k$ for $a, b_1, \dots, b_k > 0$. This means that there exists a positive function $f$ for which the relevant result holds provided that $a \leq f(b_1, \dots, b_k)$. 

Given $a, b ,c, d \in \mathbb{R}^+$, we will write expressions of the form $a = b \pm c$ to mean $a \in [b-c, b+c]$. Similarly, $a \pm b = c \pm d$ denotes an asymmetric relation that should be interpreted as $[a - b, a + b] \subseteq [c - d, c + d]$.

We will generally omit floors and ceilings when they are not critical to the argument.

\subsection{Probability}

We will use the following well-known concentration inequality for the binomial distribution. 

\begin{lemma}\label{lem:chernoff_bin}(Chernoff's inequality for the binomial distribution)
Let $X$ be a binomial random variable with parameters $n,p$. Then, for any $0<t<\E X$,
$$\P(|X-\E X|\ge t)\le 2e^{-t^2/(3\E X)}.$$
\end{lemma}

We will also need the analogous result for the hypergeometric distribution. Recall that a hypergeometric random variable with parameters $N,n$ and $m$ takes value $k$ with probability $\binom{m}{k}\binom{N-m}{n-k}/\binom{N}{n}$.

\begin{lemma}\label{lem:chernoff}(Chernoff's inequality for the hypergeometric distribution)
Let $X$ be a hy\-per\-ge\-o\-met\-ric random variable with parameters $N,n$ and $m$. Then, for any $t>0$,
\[\P\left(|X-\E X|\ge t\right)\le 2e^{-t^2/(3\E X)}.\]
\end{lemma}

\section{The directed case}\label{sec:directed}

In this section we prove \autoref{thm:di}, which says that $d$-regular balanced tripartite digraphs on $3n$ vertices with $d\ge (1+\eps)n$ are Hamilton-decomposable.

The proof of this theorem relies on several results about the family of digraphs known as \emph{robust outexpanders}, which were first studied by K\"uhn, Osthus, and Treglown \cite{kkt}. Roughly speaking, these are digraphs in which for each set of vertices $S$ (that is neither too small nor too large) there are many vertices with many inneighbours in $S$. This is formalized in the following. 

\begin{defn}[Robust outneighbourhood]\label{defn:robust_outneighbourhood}
    Given an $n$-vertex digraph $G$, a set $S\subseteq V(G)$, and $\nu\in[0,1]$, the \emph{$\nu$-robust outneighbourhood} of $S$ is defined as $$RN^+_{\nu,G}(S)\coloneqq \{v\in V(G):~d^-_G(v,S)\ge \nu n\}.$$ 
\end{defn}

This leads to the following definition. 

\begin{defn}[Robust outexpander]\label{defn:robust_outexpander}
    Let $0<\nu\le \tau<1/2$. An $n$-vertex digraph $G$ is a \emph{robust $(\nu,\tau)$-outexpander} if $|RN^+_{\nu,G}(S)|\ge |S|+\nu n$ for all $S\subseteq V(G)$ with $\tau n\le |S|\le (1-\tau)n$.
\end{defn}

The key property of robust expanders we will be using is the following celebrated result of K\"uhn and Osthus \cite{kelly-published} which says that they are Hamilton-decomposable.

 \begin{theorem}\label{thm:ko}(\cite[Theorem 1.2]{kelly-published}) 
For every $\alpha>0$ there exists $\tau$ such that for all $\nu>0$ there is $n_0 = n_0(\alpha, \nu,\tau)$ for which the following holds. Let $G$ be a $d$-regular robust $(\nu,\tau)$-outexpander on $n\ge n_0$ vertices with $d\ge \alpha n$ even. Then $G$ has a Hamilton decomposition.
\end{theorem}

We will prove \autoref{thm:di} by showing that the digraphs from the theorem are in fact robust outexpanders, and then appealing to \autoref{thm:ko}. Our proof relies on the following lemma, which says that regular digraphs that are not robust expanders exhibit a special structure. It is a simple corollary of Lemma 3.6 in \cite{LPY23}.  

\begin{lemma}\label{lem:structure_not_expander}
	Let $1/n\ll\nu\ll\tau\ll\alpha \leq 1$ and let $G$ be a $d$-regular $n$-vertex digraph with $d = \alpha n$.\footnote{Note that \cite[Lemma 3.6]{LPY23} has a typo where $\alpha \ll 1$ should have said $\alpha \leq 1$, which is what the proof gives.
    }

	If $G$ is not a robust $(\nu,\tau)$-outexpander, then there is a partition of $V(G) = V_{11} \cup V_{12} \cup V_{21} \cup V_{22}$ with the following properties. Let $V_{i*} \coloneqq V_{i1} \cup V_{i2}$ for each $i = 1,2$, and similarly let $V_{*j} \coloneqq V_{1j} \cup V_{2j}$ for each $j = 1,2$. Then
 \begin{enumerate}[label=\emph{(\roman*)}]
\item\label{prop:i} $|V_{1*}|, |V_{*1}|, |V_{2*}|, |V_{*2}| \geq d - \nu^{1/2} n$,
\item\label{prop:ii} $e_G(V_{1*}, V_{*2}) +e_G(V_{2*}, V_{*1}) \leq 4 \nu n^2$, and
\item\label{prop:iii} $||V_{12}| - |V_{21}|| \leq 4 \nu \alpha^{-1}n$.
 \end{enumerate}
\end{lemma}

\begin{proof} %[Proof of \autoref{lem:structure_not_expander}]
\cite[Lemma 3.6]{LPY23} immediately gives \emph{\ref{prop:ii}} as well as
\begin{equation}\label{eq:sizeofV*}|V_{i*}|, |V_{*i}| \geq \tau n\end{equation}
for each $i \in \{1,2\}$. For \emph{\ref{prop:iii}} note that
\begin{align*}
|d|V_{1*}| - d|V_{*1}||
&=|e_G(V_{1*}, V) - e_G(V, V_{*1})|\\
&\leq 
|e_G(V_{1*}, V) - e_G(V_{1*}, V_{*1})| + |e_G(V_{1*}, V_{*1}) - e_G(V, V_{*1})| \\
&= e_G(V_{1*}, V_{*2}) + e_G(V_{2*}, V_{*1})\\ 
&\leq 4\nu n^2,
\end{align*}
where the first equality holds because $G$ is $d$-regular and the last inequality holds by \emph{\ref{prop:ii}}.
Dividing by $d = \alpha n$ gives \emph{\ref{prop:iii}}.

For \emph{\ref{prop:i}}, note that all but at most $8\nu^{1/2} n$ vertices in $V_{1*}$ have at least $d - \nu^{1/2} n$ outneightbours in $V_{*1}$ (otherwise we have at least $(8 \nu^{1/2}n)(\nu^{1/2}n) = 8 \nu n^2$ edges from $V_{1*}$ to $V_{*2}$ contradicting \emph{\ref{prop:ii}}). It then follows from \eqref{eq:sizeofV*} that there are at least $|V_{1*}| - 8 \nu^{1/2} n \geq \tau n - 8 \nu^{1/2}n > 1$ vertices in $V_{1*}$ with $d - \nu^{1/2}n$ outneighbours in $V_{*1}$. Hence, \[|V_{*1}| \geq d - \nu^{1/2}n .\]
Running a symmetrical argument with respect to $V_{1*}, V_{*2}$ and $ V_{2*}$ gives \emph{\ref{prop:i}}.
\end{proof}

Now we turn to the main ingredient for the proof of \autoref{thm:di}.

\begin{lemma}\label{lem:di}
Let $1/n \ll \nu \ll \tau, \varepsilon$, and let $G$ be a balanced tripartite digraph on $3n$ vertices which is $d$-regular for some $d\ge (1+\eps)n$. Then $G$ is a robust $(\nu, \tau)$-outexpander.
\end{lemma}

\begin{proof}
Suppose, for a contradiction, that $G$ is not a robust $(\nu, \tau)$-outexpander. By decreasing $\tau$ and $\eps$ if necessary, we may assume that $\tau,\eps \ll 1$ (since the statement of the lemma becomes stronger for smaller $\tau$ and $\varepsilon$). Apply \autoref{lem:structure_not_expander} 
 with $\alpha = \frac{1+\eps}{3}$ to obtain a partition $V(G)=V_{11}\cup V_{12}\cup V_{21}\cup V_{22}$ such that 

\begin{enumerate}[label=(\roman*)]
    \item $|V_{1*}|,|V_{2*}|,|V_{*1}|,|V_{*2}|\ge d - \nu^{1/2}n \geq \tau n$, 
    \item $e(V_{1*},V_{*2})+e(V_{2*},V_{*1})\le 36\nu n^2$, and 
    \item $||V_{12}|-|V_{21}||\le 12\nu (\frac{1+\eps}{3})^{-1}n\le 36\nu n$ (and so $| |V_{1*}| - |V_{*1}| | \le 36 \nu n$),
\end{enumerate}
where recall that $V_{i*}=V_{i1}\cup V_{i2}$ and $V_{*i}=V_{1i}\cup V_{2i}$.

Let $(A,B,C)$ be the tripartition of $G$ and let $V_{ij}^X=V_{ij}\cap X$ for $X \in \{ A,B,C \}$ and $i,j\in\{1,2\}$. Define $V_{*i}^X = V_{*i} \cap X$ and $V_{i*}^X =V_{i*} \cap X$. Since the tripartition is balanced, we have $\sum_{i,j}|V_{ij}^A|=\sum_{i,j}|V_{ij}^B|=\sum_{i,j}|V_{ij}^C|=n$.

\begin{figure}[ht]
    \centering
    \includegraphics[scale=0.9]{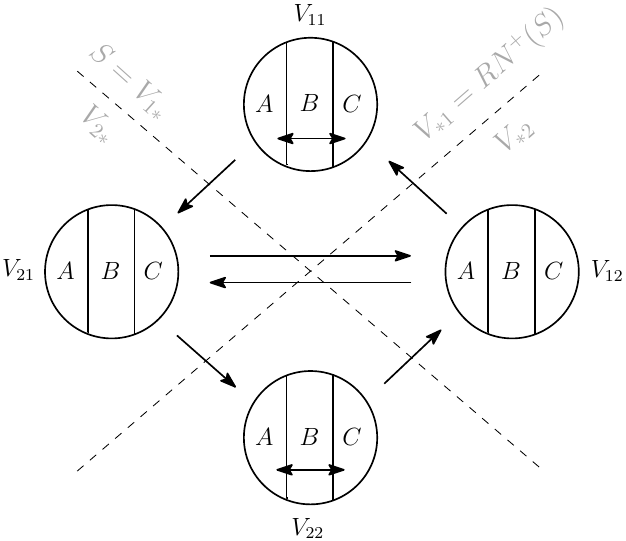}
    \caption{The structure of $G$ in the case when it is not a robust outexpander.}
    \label{fig:non_expander}
\end{figure}

We may assume that $|V_{*1}^A|\ge |V_{*1}^B| \ge |V_{*1}^C|$ by relabelling $A,B,C$ if necessary and we may assume that $|V_{11}|\le |V_{22}|$ by swapping the indices 1 and 2 if necessary. 
We now proceed to show that some of $V_{ij}^{A/B/C}$ are of size $o_\nu(n)$.

\begin{claim}$|V_{1*}^A|\le \sqrt{\nu}n$.
\end{claim}
\begin{proof}[Proof of claim.]
By summing $|V_{11}|\le |V_{22}|$ and $|V_{21}|\le |V_{12}|+36\nu n$ we have that $|V_{*1}|\le |V_{*2}|+36\nu n$. Since $|V_{*1}|+|V_{*2}|=3n$, this means that $|V_{*1}|\le 3n/2+18\nu n$. Recall that we insisted $|V_{*1}^A|$ is largest, so $|V_{*1}^B|+|V_{*1}^C|\le \frac{2}{3}|V_{*1}|\le n+12\nu n$.

From (ii), we know $e(V_{1*},V_{*2})\le 36\nu n^2$. On the other hand, we have
\begin{align*}
e(V_{1*},V_{*2}) 
\geq e(V_{1*}^A,V_{*2}^B \cup V_{*2}^C)
\geq |V_{1*}^A| \left( (1+ \varepsilon)n - |V_{*1}^B| - |V_{*1}^C| \right)
\geq |V_{1*}^A|(\varepsilon - 12\nu)n,
\end{align*}
where the second inequality follows because each vertex in $V_{1*}^A$ has at least $(1 + \varepsilon)n$ outneighbours, which must all lie in $B \cup C = V_{*1}^B \cup V_{*1}^C \cup V_{*2}^B \cup V_{*2}^C$. Combining the above two inequalities, we have  $|V_{1*}^A| \leq 36 \nu n^2 / (\varepsilon - 12\nu)n \leq \sqrt{\nu}n$ using $\nu\ll \eps$.
\end{proof}
%Now, $B$-sets can receive directed edges only from $A$- and $C$-sets, and we just showed that some $A$-sets are small. This means the inneighbourhood of some $B$-sets is confined (mostly) to $C$-sets. We proceed to use this to show that some $B$-sets must be small.
\begin{claim}
    $|V_{*1}^B|\le \sqrt{\nu}n$.
\end{claim}
\begin{proof}[Proof of claim.]
We have 
\[
e(V_{1*}^A\cup V_{1*}^C,V_{*1}^B)
\ge (1+\eps)n|V_{*1}^B| - e(V_{2*},V_{*1})
\ge (1+\eps)n|V_{*1}^B|-36\nu n^2,
\]
where the first inequality follows because each vertex in $V_{*1}^B$ has at least $(1+\eps)n$ inneighbours and they must lie in $A \cup C \subseteq V_{1*}^A\cup V_{1*}^C \cup V_{2*}$. The second inequality holds by (ii). On the other hand
\[
e(V_{1*}^A\cup V_{1*}^C,V_{*1}^B)
\leq (|V_{1*}^A| + |V_{1*}^C|)|V_{*1}^B|
\leq (\sqrt{\nu} + 1)n|V_{*1}^B|,
\]
where the last inequality follows by the previous claim. Combining the above gives that $|V_{*1}^B| \leq 36\nu n^2 / (\eps - \sqrt{\nu})n \leq \sqrt{\nu}n$ using $\nu\ll \eps$.
\end{proof}
%Now, vertices in $V_{1*}^C$ send edges mostly to $V_{*1}^A$ and $V_{*1}^B$. We already showed that the latter is small. We will use this to show that $V_{1*}^C$ is small.
\begin{claim}
    $|V_{1*}^C|\le \sqrt{\nu}n$.
\end{claim}
\begin{proof}[Proof of claim.]
Very similarly to the previous claim, we have
\[
e(V_{1*}^C,V_{*1}^A\cup V_{*1}^B) 
\ge (1+\eps)n|V_{1*}^C|- e(V_{1*}, V_{*2})
\ge (1+\eps)n|V_{1*}^C|-36\nu n^2
\]
and 
\[
e(V_{1*}^C,V_{*1}^A\cup V_{*1}^B) 
\leq |V_{1*}^C| (|V_{*1}^A| + |V_{*1}^B|)
\leq (\sqrt{\nu} + 1)n|V_{*1}^C|,
\]
using the previous claim. As before, combining these inequalities gives $|V_{1*}^C| \leq 36\nu n^2 / (\eps - \sqrt{\nu})n \leq \sqrt{\nu}n$ using $\nu\ll \eps$. 
\end{proof}
%By a similar argument considering the inneighbourhoods of vertices in $V_{*1}^C$, we can also conclude that $|V_{*1}^C|\le \nu^{1/8}n$.

Lastly, since $|V_{1*}| - |V_{1*}^B| =  |V_{1*}^A| + |V_{1*}^C| \le 2 \sqrt{\nu}n$, we have $|V_{1*}^B| \geq \tau n - 2\sqrt{\nu}n \geq  \tau n/2$ by (i). By (iii) we have $|V_{*1}| \leq |V_{1*}| + 36 \nu n \leq |V_{1*}^B| + 2\sqrt{\nu}n +36 \nu n \leq |B| + 40\sqrt{\nu}n = (1+ 40\sqrt{\nu})n$. Then
\begin{align*}
e(V_{1*}, V_{*2}) 
&\geq e(V_{1*}^B, V) - e(V_{1*}^B, V_{*1})
\geq |V_{1*}^B|\left( (1+ \eps)n - |V_{*1}| \right) \\
&\geq (\tau n/2) \left( (1+ \eps)n - (1+ 40\sqrt{\nu})n \right) 
=  \tau/2 \cdot\big(\eps - 40\sqrt{\nu}\big)n^2 > 36 \nu n^2 
\end{align*}
contradicting (ii), where we used $\nu \ll \tau, \eps$ at the end. This means $G$ must be a robust $(\nu, \tau)$-outexpander.
\end{proof}

Now \autoref{thm:di} is just a corollary.

\begin{proof}[Proof of \autoref{thm:di}]
Let $\tau$ and $n_0$ be as given by \autoref{thm:ko} with $\alpha:=\frac{1+\eps}{3}$. Choose a constant $\nu$ and take $n\ge n_0$ such that $1/n\ll \nu\ll \tau,\eps$ as given by \autoref{lem:di}. By \autoref{lem:di} our balanced tripartite digraph $G$ on $3n$ vertices is a robust $(\nu,\tau)$-outexpander and by \autoref{thm:ko} it has a Hamilton decomposition.
\end{proof}

\section{The oriented case}\label{sec:oriented}

In this section we prove \autoref{conj:ori-almost}. Similarly to \autoref{sec:directed}, we begin by pinning down the structure of regular tripartite tournaments that are not robust outexpanders (in \autoref{subset:struct}). As discussed in the introduction, we find that our graph $G$ is close to a family of structured non-expanders called $\c G_\beta$ (for some $0 \leq \beta \leq 1/2$). \autoref{sec:ori:g_beta}, which makes up the bulk of this section, contains the proof that tournaments close to $\c G_\beta$ are approximately Hamilton decomposable. We derive \autoref{conj:ori-almost} in \autoref{subsec:main_thm}.

\subsection{The structure of non-expanding regular tripartite tournaments}\label{subset:struct}

Consider the following construction of a family of regular tripartite tournaments which are not robust expanders.

\begin{defn}\label{def:Gbeta} For $\beta\in [0,1/2]$, $n\in \mathbb N$, and disjoint vertex sets $V_1=\ora{V_1}\cup \ola{V_1}, V_2$ and $V_3$ satisfying $|V_1| = |V_2| = |V_3| = n, |\ola{V_1}| = \beta n$ and $|\ora{V_1}| = (1- \beta)n$, let $\c G_\beta(\ora{V_1},\ola {V_1};V_2,V_3)$ be the family of tripartite tournaments on $(V_1,V_2,V_3)$ with edge set $(V_3 \times \ora{V_1}) \cup (\ora{V_1} \times V_2) \cup (V_2 \times \ola{V_1}) \cup (\ola{V_1} \times V_3)$ together with a $\beta n$-regular graph oriented from $V_3$ to $V_2$ and its complement (which has to be $(1-\beta)n$-regular) oriented from $V_2$ to $V_3$.  
\end{defn}
See \autoref{fig:G_beta} for a schematic of $\c G_\beta(\ora{V_1},\ola {V_1};V_2,V_3)$. We will often consider regular tripartite tournaments $G$ which are $\eps$-close to  $\c G_\beta(\ora{V_1},\ola {V_1};V_2,V_3)$ (recall from \autoref{sec:prelim:notation} that being $\eps$-close to a family of graphs means being $\eps$-close to one of its members). Any graph $G'$ from this family that is $\eps$-close to $G$ must, by definition (see \autoref{sec:prelim:notation}), share the same vertex tripartition, so the vertex partition is clear from context. For this reason, and also if we don't want to specify $\ola{V_1}$ and $\ora{V_1}$, we simply write $\c G_\beta$.

It is not difficult to see that elements of $\c G_\beta$ are not robust outexpanders. Indeed, we have that $RN^+_\nu (V_2 \cup \ola{V_1}) \subseteq V_3\cup \ola{V_1}$ and $RN^+_\nu(V_3\cup \ora{V_1}) \subseteq V_2\cup \ora{V_1}$ for all $\nu$.
%if $\nu\le 1/3$, and both are empty otherwise. 
Thus the sets $V_2 \cup \ola{V_1}$ and $V_3 \cup \ora{V_1}$ do not expand for any $\nu$, and have sizes $(1 + \beta)n$ and $(2 -\beta)n$ respectively. That is, for any $\tau < \frac{1 + \beta}{3}$ and any $\nu<\tau$, no element of $\c G_\beta$ is a robust $(\nu,\tau)$-outexpander. We show that this is the approximate structure of all regular tripartite tournaments which are non-expanders.

\begin{figure}[ht]
    \centering
    \includegraphics[scale=0.8]{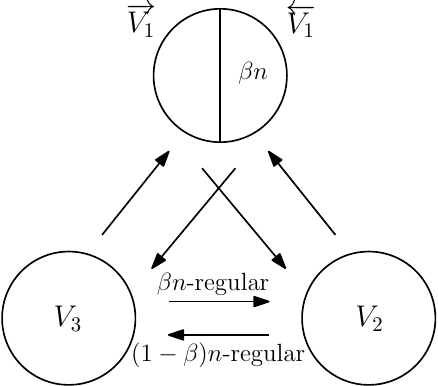}
    \caption{The non-expander complete tripartite oriented graphs $\c G_\beta$.}
    \label{fig:G_beta}
\end{figure}

To give a sketch the proof, we will consider a regular tripartite tournament $G$ with vertex classes $A, B$ and $C$, and which is not a robust outexpander. By applying \autoref{lem:structure_not_expander}, we obtain a partition $V(G)=V_{11} \cup V_{12} \cup V_{21} \cup V_{22}$ with the majority edge direction as prescribed by the lemma. By switching only a few vertices between these parts, and by relabelling the parts if necessary, we can obtain a new partition $V(G)=V'_{11} \cup V'_{12} \cup V'_{21} \cup V'_{22}$ such that $A = V'_{11} \cup V'_{22}$, $B = V'_{12}, C = V'_{21}$. \autoref{lem:almostreg} below, along with the properties of the partition in \autoref{lem:structure_not_expander}, will allow us to translate between being structurally close to $\c G_\beta$ within this reshuffling of few vertices, to being $\eps$-close in the sense defined in \autoref{sec:prelim:notation}. 

First we will need the following lemma.

\begin{lemma}\label{lem:regular_bipartite}
    Let $H$ be an undirected bipartite graph with vertex partition $V(H)=A\cup B$, with $|A|=|B|=n$. Suppose $d,t \in \mathbb{N}$ are such that $\sum_{a\in A}|d_H(a)-d| \le t$ and 
    $\sum_{b\in B}|d_H(b)-d| \le t$. Then $H$ can be made $d$-regular by adding or removing at most $9t$ edges.
\end{lemma}

\begin{proof}
    Since $ \sum_{a\in A}d(a)=\sum_{b\in B}d(b) = e(H) = dn \pm t$, we can add or remove at most $t$ edges (which may be chosen arbitrarily) to form a new graph $\hat{H}$ which has exactly $dn=e(\hat{H})$ edges. The sum of degrees of vertices in $\hat{H}$ clearly satisfy $\sum_{a\in A}|d_{\hat{H}}(a)-d|, \  \sum_{b\in B}|d_{\hat{H}}(b)-d| \le 2t$.

    If $\hat{H}$ is not regular, there exist either vertices $a,a'\in A$ for which $d(a)< d <d(a')$, or vertices $b,b'\in B$ for which $d(b)< d <d(b')$. In the former case, we may pick any vertex $x\in B$ for which $a'x \in E(\hat{H})$ and $ax\notin E(\hat{H})$, and replace the edge $a'x$ with $ax$ in the edge set $E(\hat{H})$. Similarly for the latter case. Such a step requires two edge modifications and reduces $\sum_{v\in V(H)}|d(v)-d|$. Therefore, iterating this as long as the graph is not regular, the process eventually terminates with a regular graph after at most $4t$ such steps (at most $2t$ times for vertices in $A$, as for vertices in $B$). This corresponds to at most $8t$ edge changes. Thus the total number of edge additions/deletions to achieve a regular graph from $H$ is at most $t+9t=9t$.    
\end{proof}

We use this lemma to prove the following.

\begin{lemma}\label{lem:almostreg}
    Let $G$ be a regular tripartite tournament with tripartition $V(G)=A\cup B \cup C$, where $|A|=|B|=|C|=n$. Suppose there is a partition $V(G)=V_{11}\cup V_{12}\cup V_{21}\cup V_{22}$ satisfying $e_G(V_{1*},V_{*2})+e_G(V_{2*},V_{*1})\le \eps_1 n^2$. Suppose further that $|A \triangle (V_{11}\cup V_{22})|, \ |B\triangle V_{12}|, \ |C\triangle V_{21}| \le \eps_2 n$. Then $G$ is $\eps_3$-close to $\c G_\beta$ for some $\beta$ where $\eps_3 \le 10\eps_1 + 90\eps_2$.
\end{lemma}

\begin{proof}
First note that we can modify the partition $V(G) = \bigcup_{i,j \in \{1,2\}} V_{ij}$ by moving at most $3\eps_2 n$ vertices to different parts to form the new partition $V(G) = \bigcup _{i,j \in \{1,2\}} V'_{ij}$ in which $A = (V'_{11}\cup V'_{22}), \ B = V'_{12}, \ C= V'_{21}$, and for which 

\begin{equation}\label{eq:few_bad_edges} e_G(V'_{1*},V'_{*2})+e_G(V'_{2*}, V'_{*1})\le \eps_1 n^2 + 3\eps_2n \cdot (3n). \end{equation}
This follows since each of the at most $3\eps_2 n$ moved vertices contributes at most $|V(G)|=3n$ extra edges to $e_G(V'_{1*},V'_{*2})+e_G(V'_{2*}, V'_{*1})$.

Note that, by switching the indices if necessary, we may assume that $|V'_{22}| \geq |V'_{11}|$. We will now proceed to change the directions of at most $\eps_3 n^2$ edges of $G$ to obtain a graph in the family $\c G_\beta(V'_{22}, V'_{11}; V'_{12}, V'_{21}) $, where $\beta := |V'_{11}|/n \leq 1/2$. Compared with \autoref{fig:G_beta}, we make the following identification of vertex sets.
$$V_2=V'_{12}=B, \ \ V_3=V'_{21}=C, \ \ora{V_1}=V'_{22}, \ \ola{V_1}=V'_{11}.$$
We first deal with edges in the wrong direction between $\ora{V_1}\cup \ola{V_1}$ and $V_2\cup V_3$. The number of such edges is
\[e(V'_{12},V'_{22})+e(V'_{11},V'_{12})+e(V'_{22},V'_{21})+e(V'_{21},V'_{11})\le e(V'_{1*},V'_{*2})+e(V'_{2*},V'_{*1}) \le (\eps_1 +9\eps_2)n^2,\]
where in the last inequality we used \eqref{eq:few_bad_edges}. Now we check that the edges from $V'_{21}$ to $V'_{12}$ (with orientations removed) form an almost $\beta n$-regular bipartite graph. Indeed
\begin{equation*}
    \begin{aligned}
\sum_{v\in V'_{21}}|d_G^+(v,V'_{12})-\beta n|&= \sum_{v\in V'_{21}}|n-d^+(v,V'_{22})-d^+(v,V'_{11})-\beta n| \\
&\le \sum_{v\in V'_{21}}(1-\beta )n-d^+(v,V'_{22})+d^+(v,V'_{11}) \\
&= \sum_{v\in V'_{21}}d^-(v, V'_{22})+d^+(v,V'_{11}) \\
&= e_G(V'_{22},V'_{21})+e_G(V'_{21},V'_{11}) 
\le (\eps_1+9\eps_2)n^2,
    \end{aligned}
\end{equation*}
where in the first inequality we used $d^+(v, V_{22}') \leq |V_{22}'| = (1- \beta)n$ and in the last equality we used \eqref{eq:few_bad_edges}.
Thus, applying \autoref{lem:regular_bipartite}, we have that the bipartite graph $G[V'_{21},V'_{12}]$ is within $9(\eps_1+ 9\eps_2)n^2$ edge modifications of being $\beta n$-regular. The same edge changes make the complement bipartite graph $G[V'_{12},V'_{21}]$ $(1-\beta n)$-regular. Along with the $\le (\eps_1 + 9\eps_2)n^2$ edge changes between $\ora{V_1}\cup \ola{V_1}$ and $V_2\cup V_3$ we obtain a graph in $\c G_\beta$, having made a total of at most $(10\eps_1 +90\eps_2)n^2$ edges modifications.
\end{proof}

We are now able to prove the main structural lemma for the oriented case.

\begin{lemma}\label{lem:stab}
    Let $1/n\ll \nu\ll \tau,\eps$. Suppose that $G$ is a regular tripartite tournament on $3n$ vertices which is not a $(\nu,\tau)$-outexpander. Then $G$ is $\eps$-close to $\c G_{\beta}$ for some $\beta \in [0,1/2]$. 
\end{lemma}
\begin{proof} Note that by decreasing $\tau$ and $\eps$ if necessary, we may assume that $\tau, \eps\ll 1$, since this only strengthens the statement of the lemma. Applying \autoref{lem:structure_not_expander} (with $d=\frac{1}{3} \cdot 3n$, that is $\alpha =1/3$) we obtain a partition $V(G)=V_{11}\cup V_{12}\cup V_{21}\cup V_{22}$ such that
\begin{enumerate}[label=(\roman*)]
    \item $|V_{1*}|,|V_{2*}|,|V_{*1}|,|V_{*2}| \geq d - 3\nu^{1/2} n$,
    \item $e(V_{1*},V_{*2})+e(V_{2*},V_{*1})\le 36\nu n^2$, and
    \item $\big||V_{12}|-|V_{21}|\big|\le 36\nu n$.
\end{enumerate}

We note the following structural facts about the partition.
\begin{claim}\label{clm:stab1}
    At least two of the three sets $V_{12}^A,  V_{12}^B,V_{12}^C$ have size at most $6\nu ^{1/2}n$. Moreover, the same holds for $V_{21}^A, V_{21}^B, V_{21}^C$.
\end{claim}
\begin{proof}[Proof of claim]
Suppose otherwise and for instance $|V_{12}^A|,|V_{12}^B| >6\nu ^{1/2}n$. Then $e_G(V_{1*},V_{*2})\ge e_G(V_{12},V_{12})\ge |V_{12}^A||V_{12}^B| > 36\nu n^2$, contradicting (ii). The second part follows by symmetry.
\end{proof}

\begin{claim}\label{clm:stab2}
    If $|V_{11}^X|,|V_{22}^Y|>6\nu ^{1/2}n$ for some $X,Y\in \{A,B,C\}$, then $X=Y$.
\end{claim}
\begin{proof}[Proof of claim]
If $X \neq Y$, then
$e_G(V_{1*},V_{*2})+e_G(V_{2*},V_{*1}) \ge e_G(V_{11}^X,V_{22}^Y)+e_G(V_{22}^Y,V_{11}^X) = |V_{11}^X||V_{22}^Y| > 36\nu n^2$
contradicting (ii).
\end{proof}

\begin{claim}\label{clm:stab3}
    $|V_{11}|+|V_{22}| \le (1+100 \nu ^{1/2})n$.
\end{claim}
\begin{proof}[Proof of claim]
    Assume without loss of generality that $|V_{11}^A|+|V_{22}^A| \ge |V_{11}^B|+|V_{22}^B| \ge |V_{11}^C|+|V_{22}^C|$, and that $|V_{11}^A|\ge |V_{22}^A|$. For a contradiction assume $|V_{11}|+|V_{22}| > (1+100\nu^{1/2})n$.
    Then $|V_{11}^A| \ge \frac{1}{6}(1+100\nu^{1/2})n > 6\nu^{1/2}n$, so by \autoref{clm:stab2} $|V_{22}^B|,|V_{22}^C| < 6\nu ^{1/2}n$.

    \paragraph{Case 1: $|V_{22}^A|>6\nu^{1/2}n$.}
    Here by \autoref{clm:stab2} we have $|V_{11}^B|,|V_{11}^C|<6\nu^{1/2}n$. Then 
    \begin{equation*}
        |V_{11}|+|V_{22}| = |V_{11}^A|+|V_{22}^A| + |V_{11}^B|+|V_{22}^B|+|V_{11}^C|+|V_{22}^C| \\
        \le |A|+4 \cdot 6\nu^{1/2}n
        = (1+24\nu^{1/2})n
    \end{equation*}
which contradicts the assumption that $|V_{11}|+|V_{22}|>(1+100\nu^{1/2})n$.

\paragraph{Case 2: $|V_{22}^A|\le 6\nu^{1/2}n$.}
In this case we have $|V_{22}|\le 3\cdot 6\nu^{1/2}n=18\nu^{1/2}n$. By (i) we have $|V_{2*}|=|V_{21}|+|V_{22}|\ge (1-3\nu^{1/2})n$ and $|V_{*2}|=|V_{12}|+|V_{22}|\ge (1-3\nu^{1/2})n$. Thus $|V_{12}|+|V_{21}|\ge 2(1-3\nu^{1/2})n-2|V_{22}|\ge (2-42\nu^{1/2})n$ and so $|V_{11}|+|V_{22}| = 3n-|V_{12}|-|V_{21}| \le (1+42\nu ^{1/2})n$. This contradicts the assumption that $|V_{11}|+|V_{22}|> (1+100\nu^{1/2})n$, which proves the claim.    
\end{proof}

We now use the claims to prove the theorem.
From \autoref{clm:stab1} we have $|V_{12}|, |V_{21}| \le (1+12 \nu ^{1/2})n$. Putting this together with \autoref{clm:stab3}, and using also that $\sum_{i,j\in \{1,2\}}|V_{ij}|=3n$, we have $(|V_{11}|+|V_{22}|), \ |V_{12}|, \ |V_{21}|=(1\pm 200\nu ^{1/2})n$.
We may assume without loss of generality that $|V_{12}^B| \ge |V_{12}^A|, \ |V_{12}^C|$. Thus by \autoref{clm:stab1}, we have $|V_{12}\setminus B|\le 12\nu^{1/2}n$.
Furthermore, since $|B|=n$ and $|V_{12}|=(1\pm 200 \nu^{1/2})n$, then $|V_{12}\triangle B| \le 212\nu^{1/2}n$.
Since $|V_{21}^B|$ cannot be large, we may also assume $|V_{21}^C|\ge |V_{21}^A|,|V_{21}^B|$. Thus also $|V_{21}\triangle C| \le 212 \nu^{1/2}$.
This then also forces $|(V_{11}\cup V_{22})\triangle A| \le 1000 \nu^{1/2}n$.
By \autoref{lem:almostreg}, $G$ is $\eps$-close to $\c G_\beta$ with $\eps \le 10 \times 36\nu^{1/2} + 90 \times 1000 \le 10^5\nu^{1/2}$.
\end{proof}

\subsection{Approximate decompositions of non-expanders}\label{sec:ori:g_beta}

In this section we prove that regular tripartite tournaments that are $\eps$-close to $\c G_\beta$ admit an approximate Hamilton decomposition. To be precise, we prove the following.

\begin{lemma}\label{thm:main_oriented}
    Let $1/n \ll \eps \ll \delta \leq 1$ and let $\beta \in [0,1/2]$.  Let $G$ be a $3n$-vertex regular tripartite tournament that is $\eps$-close to $\mathcal{G}_\beta$. Then $G$ contains at least $(1-\delta)n$ edge-disjoint Hamilton cycles. 
\end{lemma}

Our proof strategy takes inspiration from a technique of \cite{ferber-published} which was also used in \cite{liebenau-published}. As it is implemented in this paper, the technique constructs an approximate Hamilton decomposition of an $n$-vertex digraph $G$ by following these high level steps:

\begin{enumerate}[label=(\arabic*)]
\item\label{sketch:partition} Partition $G$ into spanning subgraphs $H_1, \dots, H_t$ (for some large constant $t$), each of which admits a vertex partition into an almost regular sparse graph $H_i[X_i]$ on $(1 - o(1))n$ vertices and an almost regular dense graph $H_i[W_i]$ on $o(n)$ vertices.   
\item\label{sketch:pathcovers} Show that each $H_i[X_i]$ can be approximately decomposed into nearly spanning linear forests, each satisfying key properties that ensure it can be closed to a Hamilton cycle.
\item\label{sketch:hamilton} Simultaneously close the linear forests from the previous step into edge-disjoint Hamilton cycles using $H_i[W_i]$. 
\end{enumerate}

While step \ref{sketch:partition} can mostly be accomplished with standard tools, steps \ref{sketch:pathcovers} and \ref{sketch:hamilton} present a real challenge as Hamilton cycles in the tripartite setting must satisfy certain parity conditions closely related to the following notion.

\begin{defn}[clockwise, counterclockwise, bidirectionally balanced] Let $G$ be a regular tripartite tournament with vertex partition $V(G) = V_1 \cup V_2 \cup V_3$. A subgraph $F\subseteq G$ is said to be \emph{clockwise-balanced} if \(e_{ F}(V_1,V_2)=e_{F}(V_2,V_3)=e_{ F}(V_3,V_1),\) and \emph{counterclockwise-balanced} if \( e_F(V_3, V_2) = e_{F}(V_2,V_1) = e_F(V_1, V_3).\) If both conditions hold, we say that $F$ is \emph{bidirectionally balanced}.
\end{defn}

Crucially, any Hamilton cycle in a regular tripartite tournament must be bidirectionally balanced (see \autoref{lem:factor_balanced} below). This implies that any almost spanning linear forests that is not close enough to being bidirectionally balanced cannot, in general, be extended into a Hamilton cycle. Accordingly, one of the key properties we enforce in step \ref{sketch:pathcovers} is that the linear forests are (exactly) bidirectionally balanced. 

In the special case when $G$ is itself a member of $\c G_\beta$, the combination of this balance condition with structural properties of $\c G_\beta$ suffices to carry out the entire decomposition strategy. In the general case, where $G$ is only $\eps$-close to some $G' \in \c G_\beta$, we can still do a lot using only the edges in $G'$ but  we cannot afford to ignore all edges of $G \setminus G'$, over which we have little control. Indeed, there may be a small set of \emph{exceptional vertices} in $G$ -- that is, vertices that are incident with many edges in $G \setminus G'$ -- whose incident edges belonging to $G\setminus G'$ must be visited by some of our Hamilton cycles. 

Our way of tackling this problem (which is similar to approaches from \cite{bertille}) is to carry out a  preliminary step before \ref{sketch:partition}--\ref{sketch:hamilton} that ensures no exceptional vertices are involved in subsequent steps:

\begin{enumerate}[label=(\arabic*)]
\setcounter{enumi}{-1}
\item\label{sketch:except} Construct $(1- o(1))n$ small linear forests $\c F_1, \dots, \c F_{\ell}$ such that for each exceptional vertex $v$, we have $d^\pm_{\c F_i}(v) = 1$ for all $i \in [\ell]$. 
\end{enumerate}

Then, in step \ref{sketch:pathcovers}, we are free to build bidirectionally balanced linear forests using only edges of $G'$. Finally, we combine each of these forests with a unique $\c F_i$. Provided that the $\c F_i$ also satisfy some weak balance conditions, the resulting forests can be completed into edge-disjoint Hamilton cycles.

Before proceeding to the proof, we note that it will be enough to prove \autoref{thm:main_oriented} for the following ranges of $\beta$: when we have $\eps \ll \beta$ and when $\beta = 0$. Then, \autoref{thm:main_oriented} easily follows in its full generality from this. Indeed, if $\eps \ll \beta$ does not hold, then necessarily $\beta \ll \delta$. In this case, it is straightforward to check that $G$ is $(\beta+\eps)$-close to $\c G_0$, so applying the result with $0, \eps + \beta$ in place of $\beta, \eps$ yields the result. Splitting into separate cases depending on the value of $\beta$ is essential, as several parts of the proof require different treatments in the two regimes.

The four steps outlined above correspond roughly to the subsections that follow. \autoref{sec:except1} shows how to cover the exceptional vertices as described in step~\ref{sketch:except}, first for the case $\eps \ll \beta$, and then for $\beta = 0$. In \autoref{sec:partition}, we establish the partition required for step~\ref{sketch:partition}. Step~\ref{sketch:pathcovers} is addressed in \autoref{sec:nearspanning}. \autoref{sec:hamilton} presents some auxiliary Hamiltonicity results used in step~\ref{sketch:hamilton}. Finally, the proof of \autoref{thm:main_oriented} is completed in \autoref{sec:comb}, where the strategy outlined above is executed by combining all previous ingredients.

\subsubsection{Covering the exceptional vertices}\label{sec:except1}

Given a regular tripartite tournament $G$ that is $\eps$-close to $G' \in \c G_\beta$, this section shows how to construct a family of $(1- o(1))n$ small linear forests that cover the set of exceptional vertices  in $G$; that is, vertices incident with many edges in $G \setminus G'$. We will make sure that each exceptional vertex has in/outdegree $1$ in all these forests, thus guaranteeing that the endpoints of these forests are non-exceptional and can be extended and ultimately closed to Hamilton cycles using edges of $G \cap G'$. This is easier to accomplish due to the rich structure of $\c G_\beta$. 

Let us start by defining what we mean by exceptional vertices. 

\begin{defn}[exceptional vertices]\label{defn:exceptionalvts}
    Let $\gamma > 0$. Let $G$ and $G'$ be digraphs on the same vertex set $V$ of size $n$. A vertex $v \in V$ is said to be \emph{$\gamma$-exceptional} in $G$ with respect to $G'$ if 
    \[d^+_{G \setminus G'}(v) + d^-_{G \setminus G'}(v) \geq \gamma n.\]
    We write $U^\gamma(G;G')$ to denote the set of $\gamma$-exceptional vertices in $G$ relative to  $G'$.
\end{defn}

We will often write $U^\gamma$ in place of $U^\gamma(G;G')$ when $G$ and $G'$ are clear from context.

% Let $G$ be a regular tripartite tournament that is $\eps$-close to $G' \in \c  G_\beta$. Recall from \autoref{sec:prelim:notation} that 
% \[U^\gamma(G;G') = \{u \in V(G): d^+_{G \setminus G'}(v) + d^-_{G \setminus G'}(v) \geq \gamma n\}\]
% is the set of \emph{$\gamma$-exceptional vertices} of $G$ with respect to $G'$. Exceptional vertices are problematic when building an approximate Hamilton decomposition because most of the edges of $G \setminus G'$ that they are incident with need to be used in the decomposition. \YP{Non-exceptional vertices can also be problematic by this definition. How about "[...] because most edges of $G\setminus G'$ incident to exceptional vertices need to be included in the decomposition but don't obey the structure of $\c G_\beta$"?} In this section, we overcome this issue by constructing a family of $(1 - o(1))n$ small linear forests that cover $U^\gamma(G;G')$, that is, each vertex of $U^\gamma(G;G')$ has in/outdegree $1$ in all these forests. \YP{Here I'd want to say something like "The endpoints of these forests are now non-exceptional and can be extended and ultimately closed to Hamilton cycles using edges in $G'$."} 

The following simple observation shows that the set of exceptional vertices is small. 

\begin{fact}\label{fact:exceptionalsize} Let $n \geq 1, \eps > 0$ and $\gamma \geq 18\eps^{1/2}$. Let $G, G'$ be $3n$-vertex tripartite tournaments on the same vertex set. If $G$ is $\eps$-close to $G'$, then
\[|U^\gamma(G;G')| \leq \eps^{1/2}n.\]
\end{fact}
\begin{proof} Each vertex in $U^\gamma(G; G')$ is incident with at least $\gamma n$ edges in $E(G) \setminus E(G')$. Thus,
\[\frac{|U^\gamma(G;G')|\cdot \gamma n}{2} \leq e(G \setminus G') \leq 9 \eps n^2,\]
which rearranges to the desired inequality. 
\end{proof}

In light of the discussion at the start of \autoref{sec:ori:g_beta}, we will build these forests while maintaining some control over the number of clockwise/counterclockwise edges that are used. The following fact will be useful for this.

\begin{fact}\label{lem:factor_balanced}
Any $1$-regular\footnote{A 1-regular digraph is also known as a \emph{1-factor}, or a \emph{cycle factor}.} balanced tripartite digraph $F$ is bidirectionally balanced.
\end{fact}
\begin{proof} Let $(V_1, V_2, V_3)$ be a tripartition of $V(F)$ with $|V_1| = |V_2| = |V_3| = n$. Since $F$ is a $1$-factor, $F[V_i, V_{i+1}], F[V_i, V_{i-1}]$ are matchings for each $i \in [3]$. Since each vertex in $V(F)$ has in/outdegree 1 in $F$, we have \(e_F(V_1, V_{2}) + e_F(V_1, V_3) = n\) and \(e_F(V_1, V_2) + e_F(V_3, V_2) = n\). This yields
\[e_F(V_1, V_3) = n - e_F(V_1, V_2) = e_F(V_3, V_2).\]
 By symmetry, we get equality of the number of counterclockiwse edges across the three bipartite graphs, and also for the clockwise edges. 
\end{proof}

The next lemma guarantees (through repeated applications) many cycle factors without short cycles in a regular digraph. Note that these are automatically bidirectionally balanced by \autoref{lem:factor_balanced}. They will serve as a `template' for building the small linear forests.

\begin{lemma}(\cite[Theorem 1.3]{lpy2})\label{lem:onefactor}
    For every $\alpha >0$ and sufficiently large $n$, every $d$-regular digraph $G$ on $n \geq n_0$ vertices with $d\geq \alpha n $ can be covered by at most $n/(d+1)$ vertex disjoint cycles, each of length at least $d/2$. Moreover, if $G$ is an oriented graph, then at most $n/(2d+1)$ cycles suffice.\footnote{Note that the statement of \cite[Theorem 1.3]{lpy2} does not include the fact that each cycle has length at least $d/2$, but it is mentioned in the text immediately afterwards and follows from the proof. 
    %\VP{This will be in the new arxiv version and the final published version, but is not in the current arxiv version.}
    }
\end{lemma}

Now we turn to the two main lemmas of this section, which guarantee an appropriate covering of the exceptional vertices in the cases $\eps \ll \beta$ and $\beta= 0$ respectively. We start with the first of these. 

\begin{lemma}\label{lem:except_vts_Gbeta}
Let $1/n \ll \eps \ll \gamma \ll \beta, \eta \leq 1/2$, $\ell \leq (1- \eta)n$. Let $G$ be a $3n$-vertex regular tripartite tournament with partition $V_1 \cup V_2 \cup V_3$ that is $\eps$-close to $G' \in \c G_\beta(\ora{V_1}, \ola{V_1}, V_2, V_3)$. Then $G$ contains $\ell$ edge-disjoint linear forests $\c F_1, \dots, \c F_{\ell}$ such that
\begin{enumerate}[label = \normalfont{(E\arabic*)}]
\item\label{prop:size} $e(\c F_i) \leq \eps^{1/3}n,$
    \item\label{prop:coveringUgamma} $d^\pm_{\c F_i}(v) = 1$ for each $v \in U^\gamma(G; G')$, 
    \item\label{prop:halfbalanced}  $e_{\c F_i}(V_3, V_1) = e_{\c F_i}(V_1, V_2)$ and $e_{\c F_i}(V_2, V_1) = e_{\c F_i}(V_1, V_3)$, and
    \item\label{prop:v2-v3} each path in $\c F_i$ has its startpoint in $V_3$ and its endpoint in $V_2$. 
\end{enumerate}
\end{lemma}

\begin{proof} We begin by applying \autoref{lem:onefactor} to $G$ $\ell' = \lfloor (1 - \eps^{1/3})n \rfloor$ times, each time removing a cycle factor, to obtain a collection of edge-disjoint cycle factors $\c J_1, \dots, \c J_{\ell'}$. After removing these cycle factors, the leftover graph is $d$-regular for some $d \geq \eps^{1/3}n$, so in each application of \autoref{lem:onefactor} we are guaranteed a $\c J_i$ consisting of cycles of length at least $\eps^{1/3}n/2$. 

Note that $e(G \setminus G') \leq 9\eps n^2$ since $G$ and $G'$ are $\eps$-close. We now discard from the collection $\{\c J_i\}_{i \in [\ell']}$ all the cycle factors which contain more than $\eps^{2/3} n$ edges of $G \setminus G'$. By edge-disjointness, the number of cycle factors we discard in this step is at most $e(G \setminus G')/(\eps^{2/3}n) \leq 9\eps^{1/3}n$. Note that $\ell \leq (1-\eta)n \leq (1 - 11\eps^{1/3})n$. So, by possibly removing some extra cycle factors and relabelling if necessary, we may thus assume that we have a collection of cycle factors $\c J_1, \dots, \c J_{\ell}$, each satisfying

\begin{enumerate}[label= (J)]
    \item\label{prop:Jprop} $\c J_i$ is a cycle factor containing at most $\eps^{2/3}n$ edges of $G \setminus G'$, and each of its cycles has length at least $\eps^{1/3}n/2$.
\end{enumerate}

We will obtain a family of edge-disjoint linear forests $\c F_1, \dots, \c F_{\ell}$ satisfying \ref{prop:size}--\ref{prop:v2-v3} by deleting some edges from each $\c J_i$. We edit each $\c J_i$ separately according to the following procedure. 

\begin{enumerate}[label = (\arabic*)]
    \item\label{step1} Delete all edges in $\c J_i[V_2, V_3]$ that are not incident with a vertex in $U^\gamma:=U^\gamma(G; G')$. 
    \item\label{step2} For any triple of vertices $u_2 \in V_2 \setminus U^\gamma, u_1 \in V_1 \setminus U^\gamma, u_3 \in V_3 \setminus U^\gamma$ such that $u_2 u_1, u_1 u_3 \in E(\c J_i)$, delete the edges $u_2 u_1$ and $u_1 u_3$ from $\c J_i$ (in other words, delete any $\ora{P_2}$ in $V_2 \times V_1 \times V_3$ not incident with $U^\gamma$).
    \item\label{step3} From the resulting oriented graph, delete all the edges inside of components that contain at most two edges and do not intersect $U^\gamma$. We call the final oriented graph $\c F_i$. 
\end{enumerate}

Observe that no edges incident with $U^\gamma$ have been deleted, and thus \ref{prop:coveringUgamma} is satisfied by construction. We will prove that $\c F_1, \dots, \c F_{\ell}$ are linear forests satisfying \ref{prop:size}, \ref{prop:halfbalanced}, and \ref{prop:v2-v3}. We start by observing that steps \ref{step1}-\ref{step3} yield some restrictions on paths in $\c F_i$. Below, we say that a subgraph of $\c F_i$ is \emph{atypical} if it either contains a vertex in $U^\gamma$ or an edge of $G \setminus G'$.

\begin{claim}\label{clm:4path} Suppose that $\c F_i$ contains a 4-vertex path $P = x_1x_2x_3x_4$ (not necessarily as a component). Then, $P$ is either atypical, or $\{x_1, x_4 \} \cap V_1 \neq \emptyset$. Consequently, any 6-vertex path is atypical.
\end{claim}
\begin{proof}
    We begin by proving the first part of the claim. Assume for a contradiction that $P$ does not contain a vertex in $U^\gamma$ or an edge of $G\setminus G'$, and that $x_1, x_4 \notin V_1$. 
    
    First, suppose that $x_1 \in V_2$. If $x_2 \in V_3$, then either $x_1$ or $x_2$ belongs to $U^\gamma$ as a result of step \ref{step1}, giving a contradiction. So we must have $x_2 \in V_1$ and $x_3\in V_3$, as $P$ has no edges in $G\setminus G'$ (here and throughout the proof, it is useful to refer back to \autoref{fig:G_beta} to see why certain edge configurations cannot occur). This contradicts step \ref{step2}.

    Now, suppose that $x_1 \in V_3$ instead. If $x_4 \in V_3$, then $x_3 \notin V_2$, or else $x_3 x_4 \in V_2 \times V_3$ and thus $\{x_3, x_4 \} \cap U^\gamma \neq \emptyset$ by step \ref{step2}. So in this case we have $x_3 \in V_1$ which, together with $x_1 \in V_3$, forces $x_2 \in V_2$. But then $x_2 x_3 x_4 \in V_2 \times V_1 \times V_3$ contains a vertex of $U^\gamma$ by step \ref{step2}, giving a contradiction. If instead $x_4 \in V_2$, by our assumption that no edge of $P$ is in $G\setminus G'$, we must have $x_2\in V_2$ and $x_3\in V_3$, contradicting step \ref{step1}. This proves the first part of the claim. 

    For the second part, take a 6-vertex path $x_1 \dots x_6$ and suppose, for contradiction, that it contains no vertex of $U^\gamma$ and no edge of $G\setminus G'$. Then by the first part of the claim applied to all its 4-vertex subpaths, we have that $x_1\in V_1$ or $x_4\in V_1$, $x_2\in V_1$ or $x_5\in V_1$, and $x_3\in V_1$ or $x_6\in V_1$. No two consecutive vertices can be in $V_1$ so, up to symmetry, we have $x_2,x_4,x_6 \in V_1$. The remaining 3 vertices must alternate between $V_2$ and $V_3$, as we are using only edges of $G'$, yielding a subpath in $V_2\times V_1\times V_3$, contradicting step \ref{step2}. 
\end{proof}

We start by observing that any (non-trivial) component in $\c F_i$ is either a cycle or a path whose startpoint is in $V_3$ and endpoint in $V_2$. Indeed, after steps \ref{step1} and \ref{step2}, only vertices in $V_3$ can possibly have lost their inedge in $\c J_i$ without also losing their outedge (an analogous observation explains why endpoints must be in $V_2$). In step \ref{step3}, on the other hand, entire path components may be deleted but without affecting the startpoint and endpoint of other path components.

\begin{claim}\label{clm:atypical} Let $C_1, \dots, C_m$ be an enumeration of the components of $\c F_i$. Then each $C_j$ is atypical. \end{claim} 
\begin{proof}
By \ref{prop:Jprop}, any $C_j$ with $|C_j| < \eps^{1/3}n/2$ induces a path. Thus, if $|C_j| \leq 3$, $C_j$ contains at most two edges and so must be atypical as a result of step \ref{step3}. If $|C_j| = 4$, $C_j$ is atypical by \autoref{clm:4path} combined with the fact that neither its startpoint nor endpoint lies in $V_1$ (we already argued that they have to lie in $V_3$ and $V_2$). Let us consider a 5-vertex (path) component $C_j = x_1 x_2 x_3 x_4 x_5$. Since each path component has startpoint in $V_3$ and endpoint in $V_2$, we have $x_1 \in V_3$ and $x_5 \in V_2$. Then applying \autoref{clm:4path} to $x_1 x_2 x_3 x_4$ and $x_2 x_3 x_4 x_5$ implies that either $C_j$ is atypical or $x_2, x_4 \in V_1$. If $x_3 \in V_3$, then $x_1 x_2 \in V_3 \times V_1$ and $x_2 x_3 \in V_1 \times V_3$, so one of these two edges must be in $G \setminus G'$, and thus $C_j$ is atypical. If $x_3 \in V_2$ instead, then either $x_3 x_4$ or $x_4x_5$ belongs to $G \setminus G'$, and so $C_j$ is atypical. This shows that any component on at most 5 vertices is atypical. Any component on at least 6 vertices is also atypical by the second part of \autoref{clm:4path}, so indeed every component is atypical. \end{proof}

We now proceed to argue that each component `decomposes' into small atypical subgraphs and thus obtain upper bounds on $m$ and on $e(\c F_i)$. Recalling that $m$ is the number of components, we have
\begin{equation}\label{eq:sizeofm} m \leq |U^\gamma| + e(\c J_i \cap (G \setminus G')) \leq 2\eps^{1/2}n,\end{equation}
where in the first inequality we used the fact that the components are all vertex-disjoint and \autoref{clm:atypical}, and in the second we used \autoref{fact:exceptionalsize} and \ref{prop:Jprop}.

The vertex set of each $C_j$ can be covered by vertex-disjoint 6-vertex paths and at most one path on at most 5 vertices. Thus
\[\sum_{i \in [m]} |C_i| \leq 6(|U^\gamma| + e( \c J_i  \cap (G \setminus G'))) + 5m \leq 22\eps^{1/2}n, \]
this time using \autoref{clm:4path} together with \eqref{eq:sizeofm}. Since each $C_i$ induces either a cycle or a path, we have
\begin{equation}\label{eq:sizeF'}
e(\c F'_i) \leq \sum_{i \in [m]} |C_i| \leq 22 \eps^{1/2}n,
\end{equation}
thus proving \ref{prop:size}. But each cycle of $\c J_i$ contains at least $\eps^{1/3}n/2$ edges by \ref{prop:Jprop}, and thus $\c F_i$ is a linear forest. We already showed that all paths in $\c F_i$ start in $V_3$ and end in $V_2$, giving \ref{prop:v2-v3}.

It remains to show that \ref{prop:halfbalanced} holds. Since $\c J_i$ is a (spanning) cycle factor in a balanced tripartite digraph, we have from \autoref{lem:factor_balanced} that
\[e_{\c J_i}(V_2, V_1) = e_{\c J_i}(V_1, V_3),\]
\[e_{\c J_i}(V_3, V_1) = e_{\c J_i}(V_1, V_2).\]
    In step \ref{step1}, only edges between $V_2$ and $V_3$ are deleted. In step \ref{step2}, the same number of edges are discarded from $\c J_i[V_2, V_1]$ as $\c J_i[V_1, V_3]$ (and no other edges). In step \ref{step3}, we only delete paths on at most two edges which, as we argued already, start in $V_3$ and end in $V_2$ (so their middle vertex, if present, lies in $V_1$). Thus, we necessarily remove the same number of edges from $\c J_i[V_3, V_1]$ as $\c J_i[V_1, V_2]$ (and no edges from $\c J_i[V_2, V_1]$ and $\c J_i[V_1, V_3]$). This shows that \ref{prop:halfbalanced} holds and concludes the proof. \end{proof}

Now we turn to the corresponding (though simpler) lemma for the case $\beta = 0$. 

    \begin{lemma}\label{lem:bad_forests} Let $1/n \ll \eps \ll \gamma, \delta \leq 1$, and let $G$ be a regular tripartite tournament on $3n$ vertices that is $\eps$-close to $\overrightarrow{C_3}(n)$. Then $G$ contains $\ell \geq (1-\delta)n$ edge-disjoint counterclockwise-balanced linear forests $\c F_1,\dots,\c F_\ell$ such that, for each $i\in [\ell]$,

\begin{enumerate}[label=\textnormal{(P\arabic*)}]
    \item $e(\c F_i) \leq \gamma n$, and\label{prop:small}
    \item for each $v \in U^\gamma(G; G')$ we have $d^\pm_{\mathcal{F}_i}(v) = 1$. \label{prop:internal}
\end{enumerate}
\end{lemma}
\begin{proof}
We begin the proof by applying \autoref{lem:onefactor} $\ell' = (1-\eps^{1/3})n$ times, each time removing a cycle factor, to obtain a collection of edge-disjoint cycle factors $\mathcal{J}'_1,..., \mathcal{J}'_{\ell'}$. Note that after removing $\ell'$ cycle factors, the resulting graph is $d$-regular for $d\ge \eps^{1/3} n$, so each $\mathcal{J}'_i$ consists of cycles of length at least $\eps^{1/3} n /2$.
    
We will obtain the required linear forests by removing a positive number of clockwise edges from each cycle, from a suitable subset of the cycle factors. Note that by \autoref{lem:factor_balanced} each $\c J_i'$ is counterclockwise-balanced and removing clockwise edges does not affect this property. We start by first discarding those $\mathcal{J}'_i$ having at least $\eps^{2/3} n$ counterclockwise edges. This will later help us in obtaining property \ref{prop:small}. Note that in doing so we discard at most $|E(\ola G)|/\eps^{2/3} n \leq \eps^{1/3} n$ of the cycle factors $\mathcal{J}'_i$, since $G$ is $\eps$-close to $\ora{C_3}(n)$.
Denote the remaining cycle factors by $\mathcal{J}_1,...,\mathcal{J}_{\ell}$, where $\ell \geq (1- 2\eps^{1/3})n\ge (1-\delta)n$.

  We now proceed, for each $i\in[\ell]$ in turn, to remove a positive number of clockwise edges from each cycle in $\mathcal{J}_i$ to obtain a linear forest satisfying properties \ref{prop:small} and \ref{prop:internal}. Let $C_1,...,C_s$ be the cycles of $\mathcal{J}_i$. We claim there are many clockwise edges which we can remove from each $C_j$ so that all vertices $v \in V(C_j)\cap U^{\gamma}$ remain internal to the resulting paths. Indeed, since every vertex in $U^\gamma$ is incident with at least $\gamma n$ counterclockwise edges, it follows that $|U^{\gamma}| \leq 2\eps n^2/\gamma n \le \eps^{1/3} n/8$, and thus the vertices in $U^{\gamma}$ are incident with at most $\eps^{1/3} n/4$ edges of $C_j$ which are ineligible for removal. So, since the total number of counterclockwise edges in $\mathcal{J}_i$ is at most $\eps^{2/3} n$, the number of clockwise edges which may be removed from $C_j$ is at least $\eps^{1/3} n/2 - \eps^{1/3} n/4 - \eps^{2/3} n>1$. That is, we can find at least one clockwise edge to remove from $C_j$ to turn it into a linear forest. This gives property \ref{prop:internal}.

We proceed to remove as many edges as possible from $\mathcal{J}_i$ as above, leaving only the coun\-ter\-clock\-wise edges and those required to keep vertices in $U^{\gamma}$ internal. The remaining linear forest has size bounded by $|\mathcal{F}_i| \leq \eps^{2/3} n + 2|U^\gamma| \le  \gamma n$ as required by property \ref{prop:small}.\end{proof}

The last result of this section allows us to `polish' the forests from the previous two lemmas so that some other useful properties are satisfied.

\begin{lemma}\label{lem:forest_cleaner}
Let $1/n \ll \eps \ll \eta$, let $\ell \leq (1- \eta)n$ and $\beta \in \{0\} \cup [3\eps^{1/4}n, 1/2]$. Let $G$ be a $3n$-vertex regular tripartite tournament that is $\eps$-close to $G' \in \c G_{\beta}(\ora{V_1}, \ola{V_1}; V_2, V_3)$. Let $\c F_1, \dots, \c F_{\ell}$ be edge-disjoint linear forests in $G$ with $e(\c F_i) \leq \eps n$ and $d^\pm_{\c F_i}(v) = 1$ for each $v \in U^\eps(G;G')$. Then $G$ contains edge-disjoint linear forests $\c F'_1, \dots, \c F'_{\ell}$ with $\c F_i \subseteq \c F'_i$ and a set $U^* \subseteq V(G)$ such that 
    \begin{enumerate}[label = \textnormal{(U\arabic*)}]
        \item\label{prop:uu1} $e(\c F'_i) \leq \eps^{1/2} n$,
        \item\label{prop:uu2} $\c F'_i \setminus \c F_i$ is a bidirectionally balanced subgraph of $G \cap G'$,
        \item\label{prop:uu3} letting $\c F' = \bigcup_{i \in [\ell]} \c F'_i$, we have $d^\pm_{\c F'}(v) = \ell$ for each $v \in U^*$ and $d^\pm_{\c F'}(v) \leq \eps^{1/4}n$ for each $v \in V(G) \setminus U^*$. 
    \end{enumerate}
\end{lemma}

\begin{proof} Let $\c F = \bigcup_{i \in [\ell]} \c F_i$ and define 
\[U^* = \{ v \in V(G): \max(d^+_\c F(v), d^-_{\c F}(v)) \geq \eps^{1/3}n\}.\]
Note that $U^\eps = U^\eps(G;G') \subseteq U^*$.
We have $e(\c F) \leq \sum_{i \in [\ell]} e(\c F_i) \leq  \eps n \ell \leq \eps n^2$, and thus
\begin{equation}\label{eq:sizeofU*} |U^*| \leq \frac{2\eps n^2}{\eps^{1/3}n} \leq 2\eps^{2/3}n\leq \eps^{1/2}n. \end{equation}

Let us extend $\c F_i$ to $\c F'_i$ one by one. Suppose we have already constructed $\c F'_1, \dots, \c F'_{m-1}$ satisfying \ref{prop:uu1}--\ref{prop:uu3} (for \ref{prop:uu3}, we assume that the condition is satisfied by $\bigcup_{i \in [m-1]}\c F'_i$ and that $d^\pm_{\c F'}(v) = m-1$), and let us see how to turn $\c F_m$ into $\c F'_m$ so that \ref{prop:uu1}--\ref{prop:uu3} are preserved. 

Let $Y$ be the set of vertices $v \in V(G)$ for which 
\[|\{i \in [m -1]: v \in V(\c F'_i) \}| \geq \eps^{1/4}n - 1.\]
Since \ref{prop:uu1} holds for $\c F'_1, \dots, \c F'_{m-1}$, the total number of edges in their union is at most $\eps^{1/2}n m \leq \eps^{1/2} n^2$. Each vertex of $Y$ is incident with at least $\eps^{1/4}n-1$ edges across the $\c F'_i$, which implies
\begin{equation}\label{eq:sizeofY}|Y| \leq \frac{4\eps^{1/2}n^2}{\eps^{1/4}n} \leq 4\eps^{1/2}n.\end{equation}

Let $U^* = \{v_1, \dots, v_k\}$. We now extend $\c F_m$ by constructing a collection of vertex-disjoint paths $P_1, \dots, P_k \subseteq G \cap G'$ such that $d^\diamond_{P_i}(v_i) = 1$ iff $d^\diamond_{\c F_m}(v_i) = 0$ for each $\diamond = \pm$. Moreover, each $P_i$ will have at most six edges, will be bidirectionally balanced, and will be vertex-disjoint from $V(\c F_m)$ (except possibly for intersecting at $v_i$) and from $Y$. If we can construct this collection of paths, then we are done with this step by taking $\c F'_{m} = \c F_m \cup \bigcup_{i \in [k]} P_i$. Indeed, while \ref{prop:uu2} is immediately true by construction, choosing each $P_i$ to avoid $Y$ ensures that \ref{prop:uu3} holds for $\c F'_m$ (note that each vertex gains at most one in/outedge in $\c F'_m$). For \ref{prop:uu1} we have 
\[e(\c F'_m) \leq e(\c F_m) + \sum_{i \in [k]} e(P_i) \leq \eps n + 6k \leq \eps n + 12\eps^{2/3} n\leq \eps^{1/2}n,\]
where in the third inequality we used \eqref{eq:sizeofU*}. 

So let us construct this family of paths. We again describe one step of an iterative procedure to achieve this. Suppose we have constructed $P_1, \dots, P_{t-1}$ satisfying the properties described above and consider $v_t$. Let \[\c U = \Bigg(\bigcup_{i \in [m-1]} \c F'_i \Bigg) \cup \Bigg( \bigcup_{i \in [\ell] \setminus [m-1] } \c F_i \Bigg)  \cup \Bigg( \bigcup_{i \in [t-1]} P_i \Bigg),\]
so that $\c U$ is the graph of edges that have already been used. Note that $\c U$ is the union of $\ell$ linear forests, so for each $v \in V(G) \setminus U^\eps$ we have
\begin{equation}\label{eq:availdegUgamma} d_{(G \cap G') \setminus \c U}^\pm(v) \geq n - \eps n -  \ell \geq \frac{\eta n}{2}.\end{equation}
In other words, each vertex outside of $U^\eps$ has at least $\eta n/2$ available in/outedges in $G \cap G'$. Though \eqref{eq:availdegUgamma} already provides a good lower bound on the number of available degree of most vertices, for each $v \in V(G) \setminus (Y \cup U^*)$ we also have the much stronger
\begin{equation}\label{eq:availdeg2}
    d_{(G \cap G') \setminus \c U}^\pm(v) \geq n - \eps n - \eps^{1/3}n - \eps^{1/4} n -1 \geq (1 - 2\eps^{1/4})n,
\end{equation}
which follows from the fact that $v$ is incident with at most $\eps n$ (in/out)egdes in $G\setminus G'$, at most edges $\eps^{1/4}n -1$ across the $\c F'_i$ by our choice of $Y$, at most $\eps^{1/3}n$ across the $\c F_i$ by our choice of $U^*$, and at most one across the $P_i$. 

The vertices we wish to avoid when constructing $P_t$ are those of $X: = V(\c F_m) \cup U^* \cup  Y \cup \bigcup_{i \in [t-1]} V(P_i)$. This set satisfies
\begin{equation}\label{eq:vtstoavoid}
    |X| \leq 2\eps n + 2\eps^{2/3}n + 4\eps^{1/2}n + 7k \leq 5\eps^{1/2}n,
\end{equation}
where in the first inequality we used $e(\c F_m) \leq \eps n$, \eqref{eq:sizeofU*}, \eqref{eq:sizeofY}, and the fact that each $P_i$ contains at most six edges, whereas for the second we again used \eqref{eq:sizeofU*}.

If $d^\pm_{\c F_m}(v_t) = 1$, there is nothing left to do and we let $P_t$ be the empty path. Suppose that $v_t$ is either missing an inedge or an outedge in $\c F_m$. By the assumptions of the lemma, this implies $v_t \notin U^\eps$. Then, as a result of \eqref{eq:availdegUgamma} and \eqref{eq:vtstoavoid}, we can pick some outneighbour $x_1 \notin X$ of $v_t$ in $(G \cap G') \setminus \c U$. First suppose that $v_tx_1$ is a clockwise edge. Since $v_tx_1 \in E(G')$, we must have $x_1 \in \ora{V_1} \cup V_2 \cup V_3$. But each vertex in $\ora{V_1} \cup V_2 \cup V_3$ has counterclockwise outneighbourhood of size at most $\beta n$ in $G'$, and so, provided $x_1 \in V_j$,
\[d_{(G \cap G')\setminus \c U}^+(x_1, V_{j+1}) \geq (1 - 2\eps^{1/4} - \beta)n,\]
by \eqref{eq:availdeg2}, where all indices are taken modulo 3 starting at 1. So we can pick an outneighbour $x_2 \in V_{j+1} \setminus X$ of $x_1$ in this graph and, by repeating the argument, we can pick an outneighbour $x_3 \in V_{j+2} \setminus X$ of $x_2$. Note that this yields a clockwise path $v_tx_1x_2x_3$. If $v_tx_1 \in G \cap G'$ is a counterclockwise edge, then $\beta \neq 0$ (if $G' \in \c G_0$, then $G'$ only contains clockwise edges). This implies $\beta \geq 3\eps^{1/4}n$ by the assumptions of the lemma. Then if $x_1 \in V_j$ we have
\[d^+_{(G \cap G') \setminus \c U}(x_1, V_{j-1}) \geq (1 - 2\eps^{1/4}n) - (1 - 3\eps^{1/4})n \geq \eps^{1/4}n, \]
where we used the fact that $x_1$ has at most $(1 - \beta)n$ outneighbours in $V_{j+1}$ in the graph $G'$, together with \eqref{eq:availdeg2}. So by \eqref{eq:vtstoavoid} we can pick an outneighbour $x_2 \in V_{j-1} \setminus X$ of $x_1$ and, by repeating the argument, we can pick an outneighbour $x_3 \in V_{j-2} \setminus X$ of $x_2$. This gives a counterclockwise path $v_tx_1x_2x_3$. 

By arguing in a similar way with respect to inedges, one can find $x'_1,x'_2,x'_3 \in V(G) \setminus X$ disjoint from $x_1,x_2,x_3$ such that $x'_3x'_2x'_1v_t$ is either a clockwise or a counterclockwise path. If $x_t$ is missing both an inedge and outedge in $\c F_m$, we let $P_t= x'_3x'_2x'_1v_tx_1x_2x_3$. If it is only missing an outedge, we let $P_t = v_tx_1x_2x_3$, and if it is only missing an inedge, we let $P_t = x'_3x'_2x'_1v_t$. In all of these cases, the path $P_t$ is bidirectionally balanced and on at most six edges. This finishes the construction of $P_t$. 

Hence, the procedure for constructing $P_1, \dots, P_k$ succeeds, yielding a linear forest $\c F'_m$ with the collection $\c F'_1, \dots, \c F'_m$ satisfying \ref{prop:uu1}--\ref{prop:uu3}. Repeating this process for $m = 1, \dots, \ell$ proves the lemma. \end{proof}

\subsubsection{The partition lemma}\label{sec:partition}

The following lemma says that any regular tripartite tournament that is $\eps$-close to $\c G_\beta$ can be decomposed into a small number of spanning subgraphs, each exhibiting a special structure that will allow us to approximately decompose it into Hamilton cycles at a later stage separately from the others. It is similar to \cite[Lemma 26]{ferber-published}.

\begin{lemma}\label{lem:partition}
    Let $1/n \ll K^{-1} \leq  \eta \ll 1$, and let $G$ be a regular tripartite tournament on $3n$ vertices with tripartition $V_1 \cup V_2 \cup V_3$. Then there are $K^3$ edge-disjoint spanning subgraphs $H_1,...,H_{K^3}$ of $G$ such that
    \begin{enumerate}[label=\normalfont{(P\arabic*)}]
        \item for each $\ell \in [K^3]$, there is a partition $V(G)=X_\ell\cup W_\ell$ such that, letting $W_\ell^i := W_\ell \cap V_i$ for $i \in [3]$, we have $|W_\ell^1|=|W_\ell^2|=|W_\ell^3|=n/K^2\pm 1$, \label{prop:partition}
        \item\label{prop:degreeWi} for each $\ell \in [K^3], v \in W_\ell$ and $k \in [3]$, we have
        \[d^{\pm}_{H_\ell}(v, W_\ell^k) = \left(\frac{d^{\pm}_G(v, V_k)}{n} \pm \frac{13}{K}\right)|W_\ell^k|,\]
        \item for each $\ell \in [K^3]$, there is some $r= \left(\frac{1 \pm 4\eta}{K^3}\right) n$ such that \(d^\pm_{H_\ell[X_\ell]}(v) = r \pm n^{4/7}\) for each $v \in X_\ell$, and\label{prop:regularUi}
        \item for each $\ell \in [K^3]$ and $v \in X_\ell$, 
        \(d^{\pm}_{H_\ell}(v, W_\ell) \geq \frac{\eta |W_\ell|}{30K}.\)
        \label{prop:degree_between}
    \end{enumerate}
\end{lemma}

\begin{proof} 
For each $i \in [K]$ and $k \in [3]$, we select uniformly at random (and independently of other choices of $i$ and $k$) a partition of $V_k$ into $K^2$ sets $S_{i,1}^k, \dots, S_{i,K^2}^k$ such that $|S_{i,1}^k| \geq \dots \geq |S_{i,K^2}^k| \geq |S_{i,1}^k| -1$. Thus, for each choice of $i\in[K]$ and $j \in [K^2]$, we have $| S_{i,j}^1 | = | S_{i,j}^2 | = | S_{i,j}^3 | = n/K^2 \pm 1$. We define $S_{i,j} = S_{i,j}^1\cup S_{i,j}^2 \cup S_{i,j}^3$ and note that $|S_{i,j}| = (3n)/K^2 \pm 3$. 

For each $i\in [K],j \in [K^2],$ let $Q_{i,j}$ be the oriented graph obtained from $G[S_{i,j}]$ by deleting any edge $uv \subseteq S_{i', j'}$ for some $i'\neq i$ and some $j' \in [K^2]$. We prove the following claim.

\begin{claim}\label{claim:Q}
    With high probability, for each $i \in [K], j \in [K^2], v \in S_{i,j}$ and $k \in [3]$ we have
    \[d_{Q_{i,j}}^\pm(v, S_{i,j}^k) = \left( \frac{d^\pm_G(v, V_k)}{n} \pm \frac{13}{K} \right) |S_{i,j}^k|.\]
\end{claim}

\begin{proof}

Consider $v \in S_{i,j}$ and fix some $i' \neq i$. Suppose that $v \in V_{k'}$ for some $k' \in [3]$. Denote by $X_{i, i'}(v)$ the random set of vertices $u \neq v$ such that there exists some $j'$ satisfying $u, v \in S_{i,j} \cap S_{i', j'}$. Observe that for each $k \in [3]$ distinct from $k'$, $X_{i, i'}(v) \cap V_k  = S^k_{i,j} \cap S^k_{i', j'}$. Conditioned on $S^k_{i,j} = S$ for some fixed $S \subseteq V_k$, $|X_{i,i'}(v) \cap V_k|$ has a hypergeometric distribution with parameters $(n, |S|, |S^k_{i',j'}|)$ and thus has expectation $\mu \coloneqq n^{-1}|S_{i,j}^k| |S_{i',j'}^k| \leq (2n)/K^4$ (note that assuming that $k$ is distinct from $k'$ implies that $v \notin S^k_{i',j'}, S^k_{i,j}$, and so these sets really are uniformly random subsets of $V_k$). By Chernoff's inequality for the hypergeometric distribution (\autoref{lem:chernoff}, henceforth referred to as Chernoff's inequality), 
$$
\mathbb{P}\left(\big||X_{i,i'}(v) \cap V_k | - \mu \big|> \frac{n}{K^4}\ \bigg|\ S^k_{i,j} = S\right) < e^{-\sqrt{n}}.
$$ 
By summing over all choices of $S$, we see that 
$$
\mathbb{P}\left(\left||X_{i,i'}(v) \cap V_k | - \mu \right|> \frac{n}{K^4} \right) < e^{-\sqrt{n}}.
$$  
In particular, with probability at least $1 - 6 n K^2 e^{-\sqrt{n}}$, we have $|X_{i,i'}(v) \cap V_k| \leq (3n)/K^4$ for each choice of $v \in V(G), i, i' \in [K]$ and $k \in[3]$ with $v \notin V_k$. If this holds, then for each $k \in [3], i \in [K], j \in [K^2], v \in S^k_{i,j}$, \begin{equation}\label{eq:loss_in_Qij}
d^{\pm}_{Q_{i,j}}(v) > d^\pm_{G[S_{i,j}]}(v) - \sum_{i' \neq i, k' \neq k} |X_{i, i'}(v) \cap V_{k'}| \geq d^\pm_{G[S_{i,j}]}(v) - \frac{6n}{K^3} \geq d^\pm_{G[S_{i,j}]}(v) - \frac{12 |S_{i,j}^k|}{K},
\end{equation}
where in the last inequality we used the fact that $|S_{i,j}^k| \geq n/(2K^2)$. 

Now, for each $i \in [K], j \in [K^2], v \in S_{i,j}$, by Chernoff's inequality, with probability at least $1 - 6nK^3 e^{-\sqrt{n}}$,
\[d_{G[S_{i,j}]}^\pm(v, S_{i,j}^k) = \left(\frac{d_G^\pm(v, V_k)}{n} \pm \frac{1}{K}\right)|S_{i,j}^k|.\]
With high probability all previous applications of Chernoff's inequality succeed. Thus, the previous equation together with \eqref{eq:loss_in_Qij} gives 
\[d^\pm_{Q_{i,j}}(v, S_{i,j}^k) \geq \left(\frac{d_G^\pm(v, V_k)}{n} - \frac{13}{K}\right)|S_{i,j}^k|.\]
It remains to prove a corresponding upper bound. Suppose seeking a contradiction that for some $v \in S_{i,j}$ and $ k \in [3]$ such that $v \notin V_k$ we have
\[d^+_{Q_{i,j}}(v, S_{i,j}^k) > \left(\frac{d_G^+(v, V_k)}{n} + \frac{13}{K}\right)|S_{i,j}^k|.\]
Then
\[d^-_{Q_{i,j}}(v, S_{i,j}^k) < |S_{i,j}^k| - \left( \frac{d_G^+(v, V_k)}{n} + \frac{13}{K}\right) |S_{i,j}^k| = \left(\frac{d_G^-(v, V_k)}{n} - \frac{13}{K}\right)|S_{i,j}^k|,\]
where we used the fact that $d_G^+(v, V_k) + d_G^-(v, V_k) = n$. Note that the same argument goes through replacing $+$ with $-$. So we get a contradiction in either case, thus proving the claim. 
 \end{proof}

To complete the proof, we will begin by arguing that most edges in $G$ intersect each $S_{i,j}$ in at most one vertex (i.e. most edges of $G$ are not inside the $S_{i,j}$). Then we will consider the subgraph consisting of all these edges and randomly partition it with a suitable probability distribution into $K^3$ subgraphs. We will then pair up each of these subgraphs with the $Q_{i,j}$ (playing the part of $H_i[W_i]$) and show that the resulting spanning subgraphs have the desired properties. 

Let $L$ be the spanning subgraph of $G$ defined by \[E(L) = \{uv \in E(G): \{u,v\} \not\subseteq S_{i,j} \mbox{ for all } i\in [K],j\in[K^2]\}.\] 

Now we have a claim.
\begin{claim}\label{claim:L} The following conditions hold with high probability for all $v, u \in V(G), i \in [K], j \in [K^2]$ and $\diamond= \pm$ simultaneously. 
\begin{enumerate}[label=(L\arabic*)]
    \item \label{prop:L-almost-reg} $d_L^\diamond(v) = (1-p)n \pm 3n^{1/2} \log n$ for some $p  \leq K^{-1}$ independent of $v$. 
    \item If $v,u \notin S_{i,j}$, $d_L^\diamond(v, S_{i,j}) = d_L^\diamond(u, S_{i,j}) \pm 16n^{1/2}\log n$ and $d_L^\diamond(v, S_{i,j}) \geq (1- 3K^{-1})\frac{n}{K^2}$. \label{prop:L-degree-to-Sij}
\end{enumerate}
\end{claim}
\begin{proof}

To prove the claim, fix some $k \in [3], v \in V(G) \setminus V_k$ as well as some $I \subseteq [K]$ and $T \subseteq V_{k}$ (later on we will plug in different sets for $I$ and $T$ to obtain \ref{prop:L-almost-reg} and \ref{prop:L-degree-to-Sij}). Let $j_1, \dots, j_K$ be indices such that $v \in S_{1, j_1} \cap \dots \cap S_{K, j_K}$ and consider the random subset $X \subseteq T$ defined by
$$X = T \cap \bigcup_{i \in I} S^{k}_{i, j_i}.$$  

We define 
\[p_{I} = 1 - \prod_{i \in I} \left(1 - \frac{|S^k_{i,j_i}|}{n} \right),\]
which quantifies the probability that any given $w \in T$ belongs to $\bigcup_{i \in I}S_{i,j_i}^{k}$ (recall that we assumed that $v \notin V_k$, so certainly $v \notin S^k_{i,j_i}$). Note that in general we have $|S^k_{i,j}|/n \leq 2K^{-2}$ . Thus, 
$$p_{I} \leq 1 - (1 - 2K^{-2})^{|I|} \leq 2K^{-2}|I| \leq 2K^{-1},$$ since $|I| \leq K$. So, if we set $\mu = \mathbb{E}(|X|) = p_{I} |T|$, we have $\mu \leq 2n/K$ since $T$ is a subset of $V_k$.

For each $i \in I$, let $U_i$ be a binomially distributed random subset of $V_k$, where each element is chosen independently with probability $|S_{i, j_i}^k|/n$, and let $\mathcal{E}$ be the event that $|U_i| = |S^k_{i,j_i}|$ for each $i \in I$. Then, if we define $Y = T \cap \bigcup_{i \in I} U_i$, it follows that $|Y|$ is binomially distributed with parameters $(|T|, p_{I})$. Thus $\mathbb{E}(|Y|) = \mu$ and, by Chernoff's inequality (\autoref{lem:chernoff_bin}), we have 
$
 \mathbb{P}\big(\big||Y| - \mu \big| > t\big) \leq 2e^{- \frac{t^2}{3\mu}}.
$
Moreover, conditioned on $\mathcal{E}$, $Y$ has the same distribution as $X$ (recall that $v \notin V_k$ and thus, even though $v \in S_{i,j_i}$, the set $S_{i,j_i}^k$ is still just a uniformly random subset of $V_k$ of size $n/K^2 \pm 1$). This implies that, for any $t \in \mathbb{N}$, 
$$\mathbb{P}\big(\big||X|-\mu\big| > t\big) = \mathbb{P}\big(\big||Y| - \mu \big| > t \big| \c E\big) \leq 
\frac{\mathbb{P}\big(\big||Y| - \mu \big| > t\big)}{\mathbb{P}(\mathcal{E})}
 \leq 
\frac{2e^{- \frac{t^2}{3\mu}}}{\mathbb{P}(\mathcal{E})}.
$$
However, $\mathbb{P}(|U_i| = |S^k_{i,j_i}|) = \mathbb{P}(|U_i| = \mathbb{E}(|U_i|)) \geq (n+1)^{-1}$ for each $i \in I$, since the random variable $|U_i|$ can only possibly take $n+1$ different values and in the binomial distribution the probability of an outcome is maximised at its expectation. Hence, $\mathbb{P}(\mathcal{E}) \geq (n+1)^{-|I|}$. It follows  that
$$\mathbb{P}\big(\big||X| - \mu \big| > n^{1/2} \log n\big) \leq 2e^{\frac{- n \log^2 n}{3\mu}} (n+1)^{|I|} \leq 2n^{-\frac{K \log n}{6}}(n+1)^K \leq n^{-\frac{K \log n}{8}},$$
where in the second inequality we used the fact that $\mu \leq 2n/K$. 
Thus, with probability at least $1 - n^{-K \log n/8}$,
\begin{equation}\label{eq:Tsize}
|X| = p_{I} |T| \pm n^{1/2} \log n.
\end{equation}
Recall that the choice of $I$ and $T$ was arbitrary. We will now proceed to use (\ref{eq:Tsize}) with two different choices of 
$I$ and $T$ to prove \ref{prop:L-almost-reg} and \ref{prop:L-degree-to-Sij}. 

First we prove \ref{prop:L-almost-reg}. Observe that if we set $I = [K]$ and $T = N_G^+(v, V_k)$ (resp. $N_G^-(v, V_k)$), then $T \cap \bigcup_{i \in I} S^k_{i, j_i}$ is precisely the set of outneighbours (inneighbours) $u \in V_k$ of $v$ such that $vu \notin E(L)$ ($uv \notin E(L)$). Let $p = 1 - (1 - K^{-2})^K \leq K^{-1}$ and note that certainly we have $p_{[K]} = p \pm n^{-1/2}$, but $p$ is independent of our choice of $v$ (this was not the case for $p_{[K]}$ since it depends on $j_1, \dots, j_K$). Using (\ref{eq:Tsize}), we see that with probability at least $1 - 4n^{-K \log n/ 8}$, assuming $v \in V_{k'}$,
\begin{equation}\label{eq:degreeinL}
    \begin{split}
        d^\pm_L(v) &= d^\pm_L(v, V_{k'+1}) + d^\pm_L(v, V_{k'+2}) \\ &= \sum_{j \in [2]} \left(d^\pm_G(v, V_{k'+j}) - \Big| N_G^\pm(v, V_{k'+j}) \cap \bigcup_{i \in [K]} S^{k'+j}_{i, j_i} \Big|\right) \\
        &= (1-p_{[K]})n \pm 2n^{1/2} \log n\\
        &= (1 - p)n \pm 3n^{1/2} \log n,
    \end{split}
\end{equation} 
for each $\diamond = \pm$. With probability at least $1 - 12n^{-K\log n/8 + 1}$, equation (\ref{eq:degreeinL}) holds simultaneously for every choice of $v \in V(G), \diamond = \pm$. This proves \ref{prop:L-almost-reg}.

We now go for \ref{prop:L-degree-to-Sij}. Fix some $i' \in [K], j' \in [K^2], \diamond = \pm, k \in [3], v \in V(G) \setminus V_k$ and let $T = N^\diamond_G(v, S^k_{i',j'})$. Let $I = [K] \sm \{i'\}$. Note that $T$ is a random subset of $N^\diamond_G(v, V_k)$ which, by Chernoff's inequality, satisfies, with probability at least $1 - \exp\left\{- \frac{K^2 \log^2 n}{8}\right\} = 1 - n^{-K^2 \log n/8}$,
\begin{equation}\label{eq:degreetoSij2} |T| = |N^\diamond_G(v, S_{i', j'}^k)| = \frac{d^\diamond_G(v, V_k)}{K^2} \pm n^{1/2}\log n.\end{equation} Conditioned on any choice of $S_{i', j'}^k$, we know that with probability at least $1 - n^{-K\log n/ 8}$ equation \eqref{eq:Tsize} holds for $X = T \cap \bigcup_{i \in I} S^k_{i, j_i}$ (here we used the fact that, for each $ i \neq i'$, $S^k_{i, j_i}$ is independent of $S^k_{i', j'}$). Since $S_{i',j'}^k$ is selected uniformly at random, equation \eqref{eq:Tsize} also holds at least with the same probability without conditioning on a choice of $S_{i',j'}^k$. Hence, with probability at least $1 - n^{-K^2 \log n/8} - n^{-K \log n/8}$,   
\begin{equation}\label{eq:degreetoSij} \Big|N^\diamond_G(v, S^k_{i',j'}) \cap \bigcup_{i \in I} S^k_{i,j_i}\Big| =  \frac{p_{I} d^\diamond_G(v, V_k)}{K^2} \pm 3n^{1/2} \log n.\end{equation}

With probability at least $1 - 18nK^3(2n^{-K \log n/8}) \geq 1 - n^{-K\log n/16}$, equations (\ref{eq:degreetoSij2}) and (\ref{eq:degreetoSij}) hold simultaneously for every choice of $k \in [3], v \in V \sm V_k, i' \in [K], j' \in [K^2]$ and $\diamond = \pm$. Thus, with the same probability, for any choice of $k' \in [3], v \in V_{k'}, i' \in [K], j' \in [K^2]$ such that $v \notin S_{i',j'}$, we have

\begin{equation}\label{eq:final_eq}
    \begin{split}
        d^\diamond_L(v, S_{i',j'}) &= d^\diamond_L(v, S^{k'+1}_{i',j'}) + d^\diamond_L(v, S^{k'+2}_{i',j'})  \\ 
        &= \sum_{j \in [2]} \big(d^\diamond_G(v, S^{k' + j}_{i',j'}) - | N^\diamond_G(v, S^{k'+j}_{i',j'}) \cap \bigcup_{i \in [K] \sm i'} S^{k'+j}_{i,j_i} | \big) 
        \\ &= \sum_{j \in [2]} \left(\frac{d^\diamond_G(v, V_{k'+j})}{K^2} \pm n^{1/2}\log n - \left(\frac{p_{I}d^\diamond_G(v, V_{k'+j})}{K^2} \pm 3n^{1/2} \log n\right)\right) \\
        &= \frac{(1 - p_{I}) n}{K^2} \pm 8n^{1/2}\log n,
    \end{split}
\end{equation}
where in the third equality we used (\ref{eq:degreetoSij2}) and (\ref{eq:degreetoSij}). Equation (\ref{eq:final_eq}) immediately implies the first part of \ref{prop:L-degree-to-Sij}, whereas the second part follows from the fact that $1 - p_{I} \geq 1 - 2K^{-1}$.  \end{proof}

Fix a collection $\{S_{i,j}\}$ so that each $Q_{i,j}$ satisfies the statement of \autoref{claim:Q} and $L$ satisfies \ref{prop:L-almost-reg} and \ref{prop:L-degree-to-Sij} in \autoref{claim:L}. Relabel $\{S_{i,j}\}$ as $W_1, \dots, W_{K^3}$ and $\{Q_{i,j}\}$ as $Q_1, \dots, Q_{K^3}$ accordingly (thus, $V(Q_\ell) = W_\ell$ for every $\ell\in [K^3]$). Note that each sets $W_i$ partitions into three sets $W_i^1 \cup W_i^2 \cup W_i^3$ such that if $W_i = S_{i',j'}$, then $W_i^k = S_{i',j'}^k$. For each $u \in V$, define the index set $I_u := \{\ell : u \in W_\ell\}$. Note that $|I_u| = K$ by the construction of $\{S_{i,j}\}$.  

Now, we randomly partition the edges of $L$ into $2K^3$ subgraphs $ D_1, \dots, D_{K^3}, E_1, \dots, E_{K^3}$ as follows. For each $e = (u,v) \in E(L)$, we add it to a unique subgraph, where the probability that $e$ falls within a given $D_\ell$ or $E_\ell$  for $\ell \in [K^3]$ is as follows.
\begin{itemize}
    \item Add $e$ to $E_\ell$ with probability $\frac{\eta}{2K}$ if $\ell \in I_u \cup I_v$. 
    \item Add $e$ to $D_\ell$ with probability $\frac{1 - \eta}{K^3 - 2K}$ if $\ell \notin I_u \cup I_v$.
\end{itemize}
Note that $I_u \cap I_v = \emptyset$ since $(u,v) \in E(L)$ and thus $|I_u \cup I_v| = 2K$. Hence, the probability that $e$ is added to some subgraph indeed adds up to 1.

Each $E_\ell$ behaves like a binomial random subgraph of $L[W_\ell, V \sm W_\ell]$, where each edge is taken with probability $\eta/2K$. By Chernoff's inequality, with probability at least $1 - 6nK^3e^{-\sqrt{n}}$, for each $\ell \in [K^3], v \in V(G) \setminus W_\ell$ and $\diamond = \pm$, we have
\begin{equation}\label{prop:degree_inbetween}
    d_{E_\ell}^\diamond(v) \geq \frac{\eta}{2K} \cdot d^\diamond_{L}(v, W_\ell) - \frac{\eta|W_\ell|}{8K} \geq \frac{\eta}{2K} \left(\frac{1}{3} - \frac{2}{K}\right)|W_\ell| - \frac{\eta| W_\ell|}{8K } \geq \frac{\eta|W_\ell|}{30K},
\end{equation}
where in the second inequality we used \ref{prop:L-degree-to-Sij}.

Each $D_\ell$ behaves like a binomial random subgraph of $M_\ell=  L[V(G) \sm W_\ell]$ where each edge is taken with probability $(1 - \eta)/(K^3 - 2K)$. For each $\ell \in [K^3]$ and $\diamond = \pm$, let $ u \in V(G) \sm W_\ell$ be arbitrary and let $d = d^\diamond_L(u, W_\ell)$. By \ref{prop:L-degree-to-Sij}, for each $v \in V(G) \sm W_\ell$ we have
$$d^\diamond_L(v, W_\ell) = d \pm 16n^{1/2}\log n.$$
Hence, together with \ref{prop:L-almost-reg}, we have
$$d^\diamond_{M_\ell}(v)
= d^\diamond_L(v) - d^\diamond_L(v, W_\ell)
= (1-p)n - d \pm 20n^{1/2}\log n.$$

% which also implies \corr
% $$d^\diamond_{M_\ell}(v) \geq (1 - o(1)) n \pm 20 \sqrt{n \log n},$$
% since $p = o(1)$ and $d \leq |W_\ell| \leq 2n/K^2$. 

By Chernoff's inequality, with probability at least $1 - 6nK^3\exp\{-n^{1/10}\}$, for each $\ell \in [K^3], v \in V(G) \setminus W_\ell$ and $\diamond = \pm,$
\begin{equation}\label{eq:Ui}
\begin{split}
    d^\diamond_{D_\ell}(v) &= \frac{1- \eta}{K^3 - 2K}d^\diamond_{M_\ell}(v) \pm \frac{1}{2}\left(\frac{n}{K^3}\right)^{4/7} \\
    &= \frac{(1-\eta)((1-p)n - d)}{K^3 - 2K} \pm \left(\frac{n}{K^3}\right)^{4/7} \\
    &= \frac{(1\pm4\eta)n}{K^3}\pm \left(\frac{n}{K^3}\right)^{4/7},
\end{split}
\end{equation}
where in the last equation we used the fact that $p \leq K^{-1} \leq \eta$ and $d \leq |W_\ell| \leq 2n/K^2$. 

We fix a choice of all the $D_\ell$ and $E_\ell$ such that all the above properties holding with high probability are simultaneously satisfied. We set $H_\ell = E_\ell \cup D_\ell \cup Q_\ell$ and argue that \ref{prop:partition}--\ref{prop:degree_between} hold. By construction, \ref{prop:partition} follows by considering all the partitions $(W_\ell, V(G) \sm W_\ell)$ for each $\ell \in [K^3]$. Secondly, \ref{prop:degreeWi} is implied directly by \autoref{claim:Q}. Thirdly, \ref{prop:regularUi} follows from (\ref{eq:Ui}) with $r$ being the first additive term in the second line. Finally, \ref{prop:degree_between} is a consequence of \eqref{prop:degree_inbetween}. This concludes the proof.  
\end{proof}

\subsubsection{Almost spanning balanced forests}\label{sec:nearspanning}

In this section we will show that near-spanning almost regular subgraphs of a regular tripartite tournament can be approximately decomposed into near-spanning linear forests $\c F_1, \dots, \c F_\ell$. In addition, we will show that this can be done while ensuring that each $\c F_i$ is bidirectionally balanced and avoids a set of forbidden vertices $S_i$. At a later stage, the first of these properties will be used to control the distribution of the forests' endpoints among the sets of the tripartition, whereas the second will allow us to combine these near-spanning forests with the small forests we constructed in \autoref{sec:except1}. Guaranteeing these extra properties is, in fact, the real novelty in our proof, as the decomposition into near-spanning forests is already guaranteed by \autoref{lem: path cover lemma} from \cite{ferber-published} below.

\begin{lemma}(\cite[Lemma 18]{ferber-published})
	\label{lem: path cover lemma}
Let $m,r \in {\mathbb N}$ with $ r \geq m^{49/50}$ and $m$ sufficiently large. Let $H$ be an $m$-vertex oriented graph with
$$r-r^{3/5}\leq \delta^0(H)\leq \Delta^0(H)\leq r+r^{3/5}.$$
Then there is a collection of $r - m^{24/25}\log m$ edge-disjoint linear forests in $H$, each having at least $m - m/\log ^4m$ edges. 
\end{lemma}

The next lemma is the main result of this section. For its proof, we will show that it is possible to edit the linear forests given by \autoref{lem: path cover lemma} so that our desired properties are satisfied.

\begin{lemma}\label{lem:balanced_covers2}
Let $1/n \ll \eps \ll \eps', \gamma, \eta \leq 1$, let $\ell \leq (1- \eta)\gamma n$ and $\beta \in \{0\} \cup [\eps',1/2]$. Let $G$ be a $3n$-vertex regular tripartite tournament that is $\eps$-close to $G' \in \c G_\beta(\ora{V_1}, \ola{V_1}; V_2, V_3)$. Let $V' \subseteq V(G)$ satisfy $|V' \cap V_1| = |V'\cap V_2| = |V' \cap V_3| \geq (1-\eps'/2)n$ and let $H$ be a subgraph of $G$ on vertex set $V'$ satisfying
\begin{equation}\label{eq:minmaxdeg}\gamma n - n^{4/7} \leq \delta^0(H) \leq \Delta^0(H) \leq \gamma n + n^{4/7}\end{equation} for each $v \in V'$. 
Let $S_1, \dots, S_{\ell} \subseteq V(G)$ satisfy $|S_i| \leq \eps n$ and $U^\eps(G; G') \subseteq U^* \subseteq S_i$ for some $U^*$, and suppose that each $v \in V' \setminus U^*$ is contained in at most $\eps n$ distinct $S_i$. Let $R$ be a digraph on vertex set $V(G)$ satisfying $d^\pm_R(v) \leq \eps n$ for each $v \in V' \setminus U^*$.

Then $H$ contains edge-disjoint bidirectionally balanced linear forests $\c F_1, \dots, \c F_\ell$ on vertex set $V'$ satisfying
\begin{enumerate}[label = (F\arabic*)]
    \item\label{prop:f1} $e(\c F_i) \geq |V'| - 5\eps^{1/8}n,$
    \item\label{prop:f2} $\c F_i \subseteq (H \cap G') \setminus R$,
    \item\label{prop:f3} $V(\c F_i) \cap S_i = \emptyset,$ and
    \item\label{prop:f4} letting $\c F =\bigcup_{i \in [\ell]} \c F_i$, we have $d^\pm_{\c F}(v) \geq \ell - 2\eps^{1/16}n$ for each $v \in V' \setminus U^*$.
\end{enumerate}
\end{lemma}

\begin{proof}
    Let $n' = |V'| \leq 3n$. We start by applying \autoref{lem: path cover lemma} to $H$ with $n'$ and $\gamma n$ playing the roles of $m$ and $r$. This yields a collection of $\ell' \geq \gamma n - (3n)^{24/25} \log (3n)$ edge-disjoint linear forests $\c F'_1, \dots, \c F'_{\ell'}$, each satisfying $e(\c F'_{i}) \geq (1 - \frac{1}{\log^3 n'})n'$. 
    
    % Also, letting $\c F = \bigcup_{i \in [\ell']} \c F_i$, we have $\delta_0(\c F') \geq \gamma n - n/\log^3 n $. 
    
    Now, we discard all the forests $\c F'_i$ which contain at least $\eps^{1/2}n$ edges belonging to $(H \setminus G') \cup R$. Note that any such forest either contains at least $\eps^{1/2}n/2$ edges of $H \setminus G'$ or it contains at least $\eps^{1/2}n/2$ edges of $R$. But
    \[e(H \setminus G') \leq 3\eps n^{2}\]
    since $G$ is $\eps$-close to $G'$, so there are at most $6\eps^{1/2}n$ forests of the first kind. Any forest of the second kind contains at least $\eps^{1/2}n/2 - 2\eps n \geq \eps^{1/2} n/4$ edges of $R$ incident with $V' \setminus U^*$ (note that $|U^*| \leq \eps n$ since $U^* \subseteq S_i$) and 
    \[|\{uv \in E(R): \{u, v\} \cap (V' \setminus U^*) \neq \emptyset\}| \leq 3\eps n^2,\]
    from our assumptions on $R$. Thus, there are at most $12\eps^{1/2}n$ forests of the second kind, implying that we discard at most $6\eps^{1/2}n + 12\eps^{1/2}n = 18\eps^{1/2}n$ linear forests in total. We have $\ell \leq \ell' - 18\eps^{1/2}n$, and so after discarding some extra forests and relabelling if necessary, we may assume that each of $\c F'_1, \dots, \c F'_{\ell}$ contains at most $\eps^{1/2}n$ edges of $(H \setminus G') \cup R$. 
    
    % Let us now update $\c F = \bigcup_{i \in [\ell]} \c F_i$, and observe that $\delta_0(\c F) \geq (1 - 4\eps^{1/2})\gamma n$.  

    Now, we discard from each $\c F'_i$
    \begin{enumerate}[label = (\roman*)]
        \item all edges that intersect $S_i$, and
        \item all edges in $(H \setminus G') \cup R$,
    \end{enumerate}
    thus deleting at most $2|S_i| + \eps^{1/2}n \leq 2\eps^{1/2}n$ edges from $\c F'_i$. The updated forests thus satisfy

    \begin{enumerate}[label = (A\arabic*)]
        \item\label{prop:a1} $e(\c F'_i) \geq n' - 3\eps^{1/2}n$,
        \item\label{prop:a2} $\c F'_i \subseteq (H \cap G') \setminus R$, and
        \item\label{prop:a3} $V(\c F'_i) \cap S_i = \emptyset$.
    \end{enumerate}

    We will now edit these linear forests to ensure that the missing properties are satisfied. We will achieve this over two steps: first we will ensure that in the union of the forests each $v \in V' \setminus U^*$ has high semidegree, and then we will remove some edges to make each forest bidirectionally balanced.

    \paragraph{Step 1: Boosting semidegrees.} Let $\c F' = \bigcup_{i \in[\ell]} \c F'_i$. For $\diamond = \pm$, let
    \[X^\diamond = \{v \in V' \setminus U^*: d^\diamond_{\c F'}(v) \leq \ell - \eps^{1/4} n\}. \]

    Each vertex in $X^\diamond$ contributes $+1$ to the tally of missing edges in at least $\eps^{1/4}n$ distinct forests, and thus
    \begin{equation}\label{eq:sizeofXdiamond}
        |X^\diamond| \leq \frac{3\eps^{1/2}n \cdot \ell}{\eps^{1/4}n} \leq 3\eps^{1/4}n,
    \end{equation}
    where in the first inequality we used \ref{prop:a1}.

    We will add some edges so that $X^+ \setminus S_i$ is covered with inedges and $X^- \setminus S_i$ with outedges in $\c F'_i$. To accomplish this, we process the $\c F'_i$ one after the other, turning each into a new forest $\c F''_i$. This new forest is constructed by adding some edges to $\c F'_i$, and then discarding some pre-existing edges in order to ensure that it remains a linear forest. The edges that are discarded all become part of a leftover graph $L$. To make sure that the resulting $\c F'' = \bigcup_{i\in[\ell]}\c F''_i$ has high semidegree of vertices in $V' \setminus U^*$, it will be enough to guarantee that the maximum semidegree of $L$ does not become too large.

    Suppose you have already constructed linear forests $\c F''_1, \dots, \c F''_{\ell'-1}$, and that the following are satisfied for each $i \in [\ell'-1]$. 

    \begin{enumerate}[label = (B\arabic*)]
        \item\label{prop:l1} $ \c F'_i \setminus \c F''_i \subseteq L$ and $\c F''_i \setminus \c F'_i \subseteq (H \cap G') \setminus R$,
        \item\label{prop:l2} $e(\c F''_{i} \triangle \c F'_{i}) \leq 30\eps^{1/4}n$,
        \item\label{prop:l3} $d_{\c F''_i}^+(v) = 1$ if $v \in X^+ \setminus S_i$, $d_{\c F''_i}^-(v) = 1$ if $v \in X^- \setminus S_i$, and $d^\pm_{\c F''_{i}}(v) = 0$ if $v \in S_i$, 
        \item\label{prop:l4} $e(L) \leq (\ell'-1) \cdot 24 \eps^{1/4}n$,
        \item\label{prop:l5} $\Delta^0(L) \leq \eps^{1/8} n$. 
    \end{enumerate}

    Let us show how to construct $\c F''_{\ell'}$ so that these properties are maintained (with $\ell'-1$ replaced with $\ell'$ in \ref{prop:l4}). Recall from the assumptions of the lemma that $U^\eps(G;G') \subseteq S_{\ell'}$. Thus, we have
    \[d^\pm_{(H\cap G') \setminus R }(v) = (\gamma \pm 3\eps)n \]
    for each $v \in V' \setminus S_{\ell'}$. 

    Let $D$ be the graph on vertex set $V'$ whose edges are precisely those in $(H \cap G') \setminus R$ that are not contained in any $\c F''_i$ ($i =1, \dots, \ell'-1$), in any $\c F'_i$ ($i = \ell' ,\dots, \ell$), or in $L$. From \ref{prop:l5} and the fact that each $\c F''_i$ and $\c F'_i$ is a linear forest, we must have
    \[d^\pm_{D}(v) \geq (\gamma - 3\eps)n - \ell -\eps^{1/8}n  \geq \frac{\eta \gamma n}{2}\]
    for each $v \in V' \setminus S_{\ell'}$. In other words, each such $v$ has $\eta \gamma n /2$ in/outedges unused by any of the forests or by $L$. 

    Define
    \[Y = \{v \in V': \max\{d^+_L(v), d^-_L(v)\} \geq \eps^{1/8}n - 1\},\]
    and note that 
    \[ \frac{|Y| \eps^{1/8}n}{4} \leq e(L) \leq 24\eps^{1/4}n^2,\]
    where we used \ref{prop:l4}. This rearranges to $|Y| \leq 96\eps^{1/8} n$. 
    
    In order to satisfy \ref{prop:l3} and \ref{prop:l5}, we need to make sure that we do not discard edges incident with $Y$, (pre-existing) outedges incident with $X^+$, or inedges incident with $X^-$. So let
    \[ T = Y \cup X^+ \cup X^- \cup  N^+_{\c F'_{\ell'}}(Y \cup X^+) \cup N^-_{\c F'_{\ell'}}(Y \cup X^-),\]
    and note that we have $|T|\leq 300\eps^{1/8}n$ since each vertex has in/outdegree at most $1$ in $\c F'_{\ell'}$. 

   For $\diamond = \pm$, let
   \[X^\diamond_{\ell'} = \{v \in X^\diamond : d^\diamond_{\c F'_{\ell'}}(v) = 0\}.\]
   
   Now we greedily construct a matching \[M^+ \subseteq D[X^+_{\ell'} \setminus S_{\ell'}, V' \setminus (S_{\ell'} \cup T)]\]
    covering $X^+_{\ell'} \setminus S_{\ell'}$ with outedges. This is possible since each of $X_{\ell'}^+, S_{\ell'}, T$ has size bounded above by $\eta \gamma n /20$, whereas each vertex in $X^+_{\ell'} \setminus S_{\ell'}$ has at least $\eta \gamma n /2$ available outedges in $D$. We now add $M^+$ to $\c F'_{\ell'}$ and discard from $\c F'_{\ell'}$ all pre-existing edges that are incident with $V(M^+) \setminus X^+_{\ell'}$. In doing so, we discard (and add to $L$) at most $2|V(M^+) \setminus X^+_{\ell'}| \leq 2 |X^+| \leq 6\eps^{1/4}n$ edges. Since $V(M^+) \setminus X_{\ell'}^+ \subseteq V' \setminus T$, no in/outedges incident with $Y$ are discarded, as well as no outedges incident with $X^+$ or inedges incident with $X^-$. After this modification, each vertex in $V(M^+) \setminus X^+_{\ell'}$ has no outedges in $\c F'_{\ell'}$, so indeed it remains a linear forest. 

    Now we greedily construct another matching
    \[M^- \subseteq D[V' \setminus (S_{\ell'} \cup T \cup V(M^+)), X^-_{\ell'} \setminus S_{\ell'}] \]
    covering $X^-_{\ell'} \setminus S_i$ with inedges. This is again possible since each of $X^-_{\ell'}, S_{\ell'}, T, V(M^+)$ has size at most $\eta \gamma n/20$ and each $v \in X^-_{\ell'} \setminus S_{\ell'}$ has at least $\eta \gamma n /2$ available inedges in $D$. We add $M^-$ to $\c F'_{\ell'}$ and discard from $\c F'_{\ell'}$ all pre-existing edges incident with $V(M^-) \setminus X^-_{\ell'}$. Note that we do not discard the edges of $M^+$ in this step, nor any in/outedges incident with $Y$, outedges incident with $X^+$, or inedges incident with $X^-$. Now each vertex of $V(M^-) \setminus X^-_{\ell'}$ has indegree $0$ in the forest, so it remains a linear forest. All the edges we discarded are added to $L$, and we call the resulting forest $\c F''_{\ell'}$. 

    We already argued that $\c F''_{\ell'}$ is a linear forest. \ref{prop:l1} is immediately true by construction. \ref{prop:l2} and \ref{prop:l4} follow from the fact that the total number of edges that were discarded and added to $L$ is at most
    \[2|V(M^+ \cup M^-)| \leq 24\eps^{1/4}n, \]
    where we used \eqref{eq:sizeofXdiamond}, whereas the total number of edges that were added is $|X^+_{\ell'}| +|X^-_{\ell'}| \leq 6\eps^{1/4}n$. Property \ref{prop:l3} is also true by construction: we kept the outedges incident with $X^+ \setminus X^+_{\ell'}$ and added an outedge at each element of $X^+_{\ell'}$ (and similarly for $X^-$). The last part of \ref{prop:l3} follows from \ref{prop:a3} and the fact that we only added edges avoiding $S_{\ell'}$. Finally, \ref{prop:l5} holds since no edge incident with $Y$ was added to $L$ and all other in/outdegrees in $L$ increased by $+1$ at most. 

    The procedure just described succeeds and yields a collection of linear forests $\{\c F''_{i}\}_{i \in [\ell]}$. Letting $\c F'' = \bigcup_{i\in[\ell]}\c F''_i$, we claim that this collection satisfies, for each $i \in [\ell]$,
    \begin{enumerate}[label = (C\arabic*)]
        \item\label{prop:c1} $e(\c F''_i) \geq n' - \eps^{1/8}n$,
        \item\label{prop:c2} $\c F''_i \subseteq H \cap G' \setminus R$,
        \item\label{prop:c3} $V(\c F''_i) \cap S_i = \emptyset$, and
        \item\label{prop:c4} $d^\pm_{\c F''}(v) \geq \ell - 2\eps^{1/8}n$ for each $v \in V' \setminus U^*$.
    \end{enumerate}
    \ref{prop:c1}--\ref{prop:c3} follow directly from \ref{prop:a1}--\ref{prop:a3} together with \ref{prop:l1}--\ref{prop:l3}. For \ref{prop:c4}, first note that if $v \in X^+$, then $v$ is incident with an outedge in every $\c F''_i$ such that $v \notin S_i$. But $v$ belongs to at most $\eps n$ $S_i$ since $v \in V' \setminus U^*$, and so we have
    \[d^+_{\c F''}(v) \geq \ell - \eps n.\]
    Analogously, if $v \in X^-$ we have
    \[d^-_{\c F''}(v) \geq \ell - \eps n. \]
    On the other hand, if $v \in V' \setminus (X^+ \cup U^*)$, then $v$ has retained all the outedges it had in $\c F'$ other than those that went to $L$. Thus, 
    \[d^+_{\c F''}(v) \geq \ell -\eps^{1/4}n - \eps^{1/8}n \geq \ell - 2\eps^{1/8}n,\] where we used \ref{prop:l5}. Similarly, each $v \in V' \setminus (X^- \cup U^*)$ satisfies
    \[d^-_{\c F''}(v) \geq \ell - \eps^{1/4}n - \eps^{1/8}n \geq \ell - 2\eps^{1/8}n. \]
    This proves \ref{prop:c4}.

\paragraph{Step 2: Balancing the forests.} It remains to make the forests bidirectionally balanced. Let us write $V'_i = V_i \cap V'$ for each $i \in [3]$, so that $|V'_i| = n'/3$. First let us observe that for each $i \in [\ell]$ we have
\[n'/3 - \eps^{1/8}n \leq e_{\c F''_i}(V_1', V'_2 \cup V'_3) \leq n'/3,\]
where the first inequality follows from \ref{prop:c1} by the fact that at most $\eps^{1/8}n$ vertices in $V'_1$ have outdegree 0 in $\c F''_i$, and the second inequality follows from the fact that $\c F''_i$ is a linear forest. Similarly, 
\[n'/3 - \eps^{1/8}n \leq e_{\c F''_i}(V_1' \cup V_3', V'_2 ) \leq n'/3.\]
By subtracting the second expression from the first, we deduce that $|e_{\c F''_i}(V'_1, V'_3) - e_{\c F''_i}(V'_3, V'_2)| \leq \eps^{1/8}n$. By symmetry, for each $j \in [3]$ we have
\begin{equation}\label{eq:almostbalanced}
    |e_{\c F''_i}(V'_j, V'_{j+1}) - e_{\c F''_i}(V'_{j+1}, V'_{j+2})| \leq \eps^{1/8}n.
\end{equation}
The same holds in the other direction, i.e. for each $j \in [3]$ we have
\begin{equation}
    |e_{\c F''_i}(V'_{j+2}, V'_{j+1}) - e_{\c F''_i}(V'_{j+1}, V'_{j})| \leq \eps^{1/8}n, 
\end{equation}
In other words, each $\c F''_i$ is already `almost' bidirectionally balanced up to an error of $\eps^{1/8}n$. 

Let 
\[\ora{\c F''_i} = \c F''_i[V'_1, V'_2] \cup \c F''_i[V'_2, V'_3] \cup \c F''_i[V'_3, V'_1],\]
\[\ola{\c F''_i} = \c F''_i[V'_1, V'_3] \cup \c F''_i[V'_3, V'_2] \cup \c F''_i[V'_2, V'_1].\]
We now have a claim. 

\begin{claim}\label{claim:lotsofedges}
    $e(\ora{\c F''_i}) \geq \eps'n/4$ and $e(\ola{\c F''_i}) \in \{0\} \cup [\eps' n/4, n]$
\end{claim}
\begin{proof} Recall from the statement of the lemma that $\beta \in \{0\} \cup [\eps', 1/2]$. First, suppose that $\beta = 0$. Then, each edge in $G'$ is a clockwise edge. In this case, the claim follows immediately from \ref{prop:c1} and \ref{prop:c2}. 

Otherwise, suppose that $\beta \in [\eps',1/2]$. This implies that $|\ora{V_1}|, |\ola{V_1}| \geq \eps'n$ in $G'$. Further recall from the assumptions of the lemma that $|V' \cap V_1| \geq (1 - \eps'/2)n$, so that $|V' \cap \ora{V_1}|, |V' \cap \ola{V_1}| \geq \eps' n /2$. By \ref{prop:c1}, at least $\eps' n/2 - \eps^{1/8}n \geq \eps'n/4$ vertices in $V' \cap \ola{V_1}$ have outdegree $1$ in $\c F''_i$. But $\c F''_i \subseteq G'$ and thus these edges have to be counterclockwise edges. A symmetrical argument shows that there are at least $\eps' n/4$ clockwise edges in $\c F''_i$, as required.
\end{proof}

We now enter the balancing stage, in which we remove some edges from each $\ora{\c F''_i}$ and $\ola{\c F''_i}$ (and add them to a new leftover graph $\tilde L$) so that the new forest $\c F_i$ given by their union is bidirectionally balanced. Suppose that we have already constructed edge-disjoint bidirectionally balanced linear forests $\c F_1, \dots, \c F_{\ell' -1}$ for some $\ell' \leq \ell$, and that
\begin{enumerate}[label=(D\arabic*)]
    \item\label{prop:d1} $\c F_i \subseteq \c F''_i$ for each $i \in [\ell'-1]$ and $\tilde{L} = \bigcup_{i \in [\ell'-1]}\c F''_i \setminus \c F_i$,
    \item\label{prop:d2} $e(\c F''_i \setminus \c F_i) \leq 4\eps^{1/8}n$ for each $i \in [\ell'-1]$,
    \item\label{prop:d3} $\Delta^0(\tilde L)\le \eps^{1/16}n$.
\end{enumerate}

Let us show how to balance $\ola{\c F''_{\ell'}}$ and $\ora{\c F''_{\ell'}}$ while ensuring that these properties are maintained (with $\ell'-1$ replaced with $\ell'$ in \ref{prop:d2}). 

First note that if $e(\ola{\c F''_{\ell'}}) =0$, $\ola{\c F''_{\ell'}}$ is already counterclockwise-balanced and it does not need to be edited. So, by \autoref{claim:lotsofedges} it suffices to consider the case where the `direction' we wish to balance (i.e. $\ora{\c F''_{\ell'}}$ or $\ola{\c F''_{\ell'}}$) has at least $\eps' n/4$ edges. The proof is identical in either case, so let us describe the procedure for $\ora{\c F''_{\ell'}}$ alone. 

Define
\[\tilde{Y} = \{v \in V': \max\{d^+_{\tilde{L}}(v), d^-_{\tilde{L}}(v)\} \geq \eps^{1/16}n - 1\},\]
so that
\[\frac{|\tilde{Y}|\cdot  \eps^{1/16}n}{4} \leq e(\tilde{L}) \leq 4\eps^{1/8}n^2,\]
where the second inequality follows from \ref{prop:d2}. This rearranges to $|\tilde{Y}| \leq 16\eps^{1/16}n$.

However, by \eqref{eq:almostbalanced} and the fact that $e(\ora{\c F''_{\ell'}}) \geq \eps'n/4$, we must have
\[e_{\c F''_{\ell'}}(V'_{j}, V'_{j+1}) \geq \frac{\eps' n}{24}\]
for each $j \in [3]$. At least $\eps' n/24 - 32\eps^{1/16}n \geq \eps^{1/8}n$ of these edges avoid $\tilde{Y}$. Thus, supposing $\c F''_{\ell'}[V'_{j}, V'_{j+1}]$ is the sparsest of the three bipartite graphs, by \eqref{eq:almostbalanced} it suffices to remove at most $\eps^{1/8}n$ edges from each of $\c F''_{\ell'}[V'_{j+1}, V'_{j+2}] $ and $\c F''_{\ell'}[V'_{j+2}, V'_{j}]$ to make $\ora{\c F''_{\ell'}}$ clockwise-balanced. These edges can be chosen arbitrarily as long as they avoid $\tilde{Y}$, which ensures that $ \Delta^0 (\tilde{L}) \leq \eps^{1/16}n$ after adding them to $\tilde{L}$. 

We now perform the symmetrical procedure with respect to $\ola{\c F''_{\ell'}}$ (the argument and inequali-ties are all the same), and let $\c F_{\ell'}$ be the union of $\ora{\c F''_{\ell'}}$ and $\ola{\c F''_{\ell'}}$ after the edits. Properties \ref{prop:d1} and \ref{prop:d2} clearly hold for $\c F_{\ell'}$, and \ref{prop:d3} is preserved since we only remove edges avoiding $\tilde{Y}$, as already discussed. The forest $\c F_{\ell'}$ is bidirectionally balanced by construction. 

So, the procedure succeeds and yields bidirectionally balanced linear forests $\c F_1, \dots, \c F_{\ell}$. Note that \ref{prop:f1} follows from \ref{prop:c1} combined with \ref{prop:d1} and \ref{prop:d2}. \ref{prop:f2} follows from \ref{prop:c2} and \ref{prop:d1}, \ref{prop:f3} from \ref{prop:c3} and \ref{prop:d1}. \ref{prop:f4} follows from \ref{prop:c4} combined with \ref{prop:d1} and \ref{prop:d3}.  \end{proof}

\subsubsection{Hamiltonicity results}\label{sec:hamilton}

In this section we will prove some Hamiltonicity lemmas for digraphs that are structurally similar to those in the family $\c G_\beta$. Similarly to previous sections, we will prove two different results corresponding to the cases $\eps \ll \beta$ and $\beta = 0$. These results will later be used to close near-spanning linear forests into Hamilton cycles. 

We will be using the following classical theorem. 

\begin{theorem}\label{thm:gh}(\cite{ghouilahouri})
   Let $G$ be an $n$-vertex digraph with $\delta^0(G) \ge n/2$. Then $G$ contains a Hamilton cycle.
\end{theorem}

The next well-known result is a simple consequence of Hall's marriage theorem. 

\begin{theorem}\label{thm:hall}
Let $G = G[A,B]$ be an undirected bipartite graph with $|A| = |B| = n$ and $\delta(G) \geq n/2$. Then $G$ contains a perfect matching.
\end{theorem}

Our first Hamiltonicity lemma corresponds to the case $\eps \ll \beta$. 

\begin{lemma}\label{lem:hamilton_Gbeta}
    Let $1/n \ll \eps \ll 1$. Let $G$ be a $4$-partite digraph with partition $V(G) = \ora{V_1} \cup \ola{V_1} \cup V_2 \cup V_3$ where $|V_2| = |V_3| = n$ and $|\ora{V_1}|, |\ola{V_1}|\leq (1-8\eps)n$. Suppose that
    \begin{enumerate}[label = (D\arabic*)]
    \item\label{prop:ora} $d^+(v, V_2), d^-(v, V_3) \geq (1- \eps)n$ for each $v \in \ora{V_1}$, 
    \item\label{prop:ola} $d^+(v, V_3), d^-(v, V_2) \geq (1- \eps)n$ for each $v \in \ola{V_1}$, 
    \item\label{prop:v2} $d^+(v, \ola{V_1} \cup V_3), d^-(v, \ora{V_1} \cup V_3) \geq (1 - \eps)n$ for each $v \in V_2$, and
    \item\label{prop:v3} $d^+(v, \ora{V_1} \cup V_2), d^-(v, \ola{V_1} \cup V_2) \geq (1 - \eps)n$ for each $v \in V_3$.
    \end{enumerate}
    Let $M = M[V_3, V_2]$ be a matching in $G$ on at most $\eps n$ edges. Then $G$ contains a Hamilton cycle $C$ with $M \subseteq C$. 
\end{lemma}
\begin{proof} Let us write $\ora{G}$ to denote the digraph on vertex set $ \ora{V_1} \cup V_2 \cup V_3$ and edge set
\[E(\ora{G}) = E(G[V_3, \ora{V}_1 \cup V_2]) \cup E(G[\ora{V_1}, V_2]).\]
Similarly, $\ola{G}$ is the digraph on vertex set $\ola{V_1} \cup V_2 \cup V_3$ and edge set
\[E(\ola{G}) = E(G[V_2, \ola{V}_1 \cup V_3]) \cup E(G[\ola{V_1}, V_3]).\]
We will first find a spanning forest of $\ora{P_2}$s and edges in $\ora{G}$, all having their startpoint in $V_3$ and endpoint in $V_2$. Then we will show that this forest can be completed to a Hamilton cycle using edges in $\ola{G}$. 

Let us start by matching the vertices of $V_2$ to those in $V_3$ in such a way that the resulting directed perfect matching, say $\ola{M}$, does not form any cycles together with $M$ ($\ola{M}$ does not need to be a subgraph of $G$, so this is easy to achieve). For each $v \in V_2$, we write $\ola{M}(v)$ to denote its unique outneighbour in $\ola{M}$, and similarly if $u \in V_3$ we write $\ola{M}^{-1}(u)$ to denote its unique inneighbour in $\ola{M}$. 

Now we `contract' $\ora{G}$ along the matching $\ola{M}$: we consider the digraph $G_1$ on vertex set $\ora{V_1} \cup V_3$ and edge set
\[E(G_1) = \left\{(v,u) \in E(\ora{G}) : v, u \in \ora{V_1} \cup V_3\right\} \cup \left\{(v, \ola{M}(u)): (v,u) \in E(\ora{G}), u \in V_2\right\}. \]
Observe that each $u \in V_2$ has no outedges in $\ora{G}$, and so each edge of $\ora{G}$ corresponds to a unique edge of $G_1$. It is also easy to verify that $G_1$ is a digraph without multiple edges satisfying 
\begin{enumerate}[label = (\roman*)]
    \item $d_{G_1}^+(v) = d^+_{\ora{G}}(v)$ for each $v \in \ora{V_1} \cup V_3$, 
    \item $d_{G_1}^-(v) = d^-_{\ora{G}}(v)$ for each $v \in \ora{V_1}$, and
    \item $d_{G_1}^-(v) = d^-_{\ora{G}}(\ola{M}^{-1}(v))$ for each $v \in V_3$.
\end{enumerate}

$G_1$ may have loops but, if so, at most one at each vertex: we remove them at the expense of a $-1$ in the degrees. It immediately follows from the above together with \ref{prop:ora}--\ref{prop:v3} that 
\begin{equation}\label{eq:minsemi}
    \delta_0(G_1) \geq (1-\eps)n -1.
\end{equation}

The matching $M$ also corresponds to a subgraph $M_1$ of $G_1$. It is possible that $M_1$ is not a matching anymore; however, since we chose $\ola{M}$ not forming any cycles with $M$, it is easily verified that $M_1$ is a linear forest. Moreover, we still have $e(M_1) \leq \eps n$.

Let $P_1, \dots, P_m$ be an enumeration of the (path) components in $M_1$. Let $x_i$ and $y_i$ (for $i \in [m]$) be the startpoint and endpoint (respectively) of $P_i$. We now construct one more `contracted' graph $G_1'$ as follows: we first delete all the vertices in each $P_i$ aside from $y_i$, and then modify the inedges at each $y_i$ so that
\[N^-_{G_1'}(y_i) = N^-_{G_1}(x_i) \setminus \bigcup_{j \in [m]} V(P_j).\]
In other words, $y_i$ (mostly) inherits $x_i$'s inneighbourhoood in $G_1$. Note that every other vertex lost at most
\[\sum_{j \in [m]}|V(P_j)| \leq 2 e(M_1) \leq 2\eps n\] in/outneighbours relative to $G_1$. So we have
\[\delta_0(G_1') \geq (1 - 4\eps)n = \frac{(2-8\eps)n}{2} \geq \frac{|V(G'_1)|}{2},\]
where in the first inequality we used \eqref{eq:minsemi} and in the third inequality we used $|\ola{V_1}|\leq (1-8\eps)n$ and $|V_2| = n$. So, $G_1'$ contains a Hamilton cycle $C_1'$ by \autoref{thm:gh}. 

For some choice of indices $i_1, \dots, i_m \in [m]$, this Hamilton cycle can be written as
\[C_1' = y_{i_1} P'_1 y_{i_2} P'_2 y_{i_3} \dots y_{i_m} P'_m y_{i_1},\]
where each $P'_j$ is a path in $G_1'$. De-contracting $C_1'$ yields the Hamilton cycle
\[C_1 = y_{i_1} P'_1 x_{i_2} P_{i_2}y_{i_2} P'_2 x_{i_3} P_{i_3}y_{i_3} \dots y_{i_m} P'_m x_{i_1} P_{i_1} y_{i_1}\]
in $G_1$. Note that $M' \subseteq C_1$. 

Now, we further de-contract $C_1$: we replace each edge $uv \in E(C_1)$ satisfying $v \in V_3$ with the edge $(u, \ora{M}^{-1}(v)) \in E(\ora{G})$, and call the resulting digraph $\c F$ (so that $\c F \subseteq \ora{G}$). It is easy to see that $\c F$ is a linear forest with $d^+_\c F(v) = 1$ iff $v \in \ora{V_1} \cup V_3$ and $d^-_{\c F}(v) = 1$ iff $v \in \ora{V_1} \cup V_2$. Moreover, $M \subseteq \c F$. So, $\c F$ is a forest of paths spanning $\ora{V_1} \cup V_2 \cup V_3$ and incorporating $M$, with each path having at most two edges and its startpoint in $V_3$ and endpoint in $V_2$. 

We will now repeat a similar (though even simpler) procedure in $\ola{G}$ to close $\c F$ into a Hamilton cycle. Observe that $\c F$ is a linear forest consisting of exactly $|V_2| = |V_3| = n$ paths of length at most two, so let $Q_1, \dots, Q_n$ be an enumeration of these paths. Let $w_i \in V_3$ be the startpoint of $Q_i$, and let $z_i \in V_2$ be its endpoint. Then
\[\ora{M} = \{(w_i, z_i) : i \in [n]\}\]
is a perfect matching from $V_3$ to $V_2$. Just like before, we write $\ora{M}(v)$ for $v \in V_3$ to denote its unique outneighbour in $\ora{M}$, and write $\ora{M}^{-1}(u)$ for $u \in V_2$ to denote $u$'s unique inneigbhour. 

Now we consider $\ola{G}$ instead and contract along $\ora{M}$, thus obtaining a digraph $G_2$ on vertex set $\ola{V_1} \cup V_2$ and edge set
\[E(G_2) = \left\{(v,u) \in E(\ola{G}) : v, u \in \ola{V_1} \cup V_2\right\} \cup \left\{(v, \ora{M}(u)): (v,u) \in E(\ola{G}), u \in V_3\right\}. \]
Just like before, after removing at most one loop at each vertex if necessary, it can be easily checked using \ref{prop:ola}, \ref{prop:v2}, and \ref{prop:v3} that $G_2$ is a simple digraph without multiple edges satisfying
\[\delta_0(G_2) \geq (1 - \eps)n - 1 \geq \frac{(2 - 8\eps)n}{2} \geq \frac{|V(G_2)|}{2},\]
where the third inequality uses $|\ola{V_1}| \leq (1- 8\eps)n$ and $|V_3| = n$. So $G_2$ contains a Hamilton cycle $C_2$ by \autoref{thm:gh}. 

Recall that $V_2 = \{z_1, \dots, z_n\}$ and $V_3 = \{w_1, \dots, w_n\}$ where $(w_i,z_i) \in E(\ora{M})$ for each $i \in [n]$. Then for some choice of indices $j_1, \dots, j_n$, the cycle $C_2$ can be written as
\[C_2 = z_{i_1} R_1 z_{i_2} R_2 z_{i_3} \dots z_{i_n} R_n z_{i_1},\]
where each $R_i$ is either an empty path or consists of a single vertex of $\ola{V_1}$ (note that $\ola{V_1}$ is an independent set in $G_2$ and thus no two consecutive vertices of $C_2$ lie in it). Moreover, since $C_2$ is Hamiltonian in $G_2$, each vertex of $\ola{V_1}$ is contained in some $R_i$. Thus, de-contracting $C_2$ yields the cycle
\[C = z_{i_1} R_1 w_{i_2} Q_{i_2} z_{i_2} R_2 w_{i_3} Q_{i_3} z_{i_3} \dots z_{i_n} R_n w_{i_1} Q_{i_1} z_{i_1} \]
in $G$. Since $C$ incorporates the paths $Q_1, \dots Q_n$, it also incorporates $M$. Every vertex of $\ola{V_1}$ is covered by some $R_i$, and every vertex of $\ora{V_1}$ is covered by some $Q_i$. Thus, $C$ is a Hamilton cycle, as desired. \end{proof}

Now we turn to the lemma corresponding to the case $\beta = 0$.

\begin{lemma}\label{lem:hamilton_cyclictriangle}
    Let $1/n \ll \eps \ll 1$. Let $G$ be a tripartite digraph with partition $V_1 \cup V_2 \cup V_3$ where $(1-\eps)n \leq |V_1| \leq n$ and $|V_2| = |V_3| = n$. Suppose that $d^+(v, V_{i+1}) \geq (1-\eps)n$ and $d^-(v, V_{i-1}) \geq (1-\eps)n$ for each $i \in [3]$ and $ v \in V_i$. Let $M = M[V_3, V_2]$ be a matching in $G$ on precisely $n - |V_1|$ edges. Then $G$ contains a Hamilton cycle $C$ such that $M \subseteq C$. 
\end{lemma}

\begin{proof}
Let $n' = |V_1|$ and observe that $|V_2 \setminus V(M)| = |V_3 \setminus V(M)|= n'$. The graphs $G[V_3 \setminus V(M), V_1]$ and $G[V_1, V_2 \setminus V(M)]$, seen as undirected graphs, have minimum degree at least $(1-\eps)n - (n - |V_1|)= |V_1| - \eps n  \geq \frac{n'}{2}$. Thus, by \autoref{thm:hall} we can find two perfect matchings $M_1 = M_1[V_3 \setminus V(M), V_1]$ and $M_2 = M_2[V_1, V_2 \setminus V(M)]$. Letting $\c F = M \cup M_1 \cup M_2$, we have that $\c F$ is a spanning forest of $\ora{P}_2$s and edges, each having its startpoint in $V_3$ and endpoint in $V_2$.   

Let $P_1, \dots, P_n$ be an enumeration of the paths in $\c F$, and for each $i \in [n]$ let $x_i\in V_3$ and $y_i \in V_2$ be $P_i$'s startpoint and endpoint (respectively). Let $\tilde{M} = \tilde{M}[V_3, V_2]$ be the auxiliary perfect matching whose edges are precisely the pairs $(x_i, y_i)$ for $i \in [n]$. Given $v \in V_3$, we write $\tilde{M}(v)$ to denote its unique outneighbour in $\tilde{M}$, and given $v \in V_2$, we write $\tilde{M}^{-1}(v)$ to denote its unique inneighbour.

We now consider the `contracted graph' $\tilde{G}$ on vertex set $V_2$ and edge set
\[E(\tilde{G}) = \{(u, \tilde{M}(v)): (u,v) \in G[V_2, V_3]\}.\]
It is easy to see that this is a digraph without multiple edges and at most one loop at each vertex. Moreover, each edge in $G[V_2, V_3]$ corresponds to a unique edge of $\tilde{G}$, and we have
\[d_{\tilde{G}}^+(v) = d^+_{G[V_2, V_3]}(v),\]
\[d_{\tilde{G}}^-(v) = d^-_{G[V_2, V_3]}(\tilde{M}^{-1}(v)),\]
for each $v \in V_2$. After removing loops if necessary, we thus have
\[\delta_0(\tilde{G}) \geq (1 - \eps)n -1 \geq \frac{n}{2},\]
and so by \autoref{thm:gh} $\tilde{G}$ contains a Hamilton cycle $\tilde{C}$. Writing this Hamilton cycle as
\[\tilde{C} = y_{i_1} y_{i_2}y_{i_3}\dots y_{i_n} y_{i_1}\]
for some choice of indices $i_1, \dots, i_n \in [n]$, it follows that
\[C = y_{i_1} x_{i_2} P_{i_2} y_{i_2} x_{i_3} P_{i_3} y_{i_3} \dots y_{i_n} x_{i_1} P_{i_1} y_{i_1} \]
is a cycle in $G$ incorporating $M$ since $M \subseteq \c F$. Each vertex of $V_1$ is contained in some $P_i$, and so this is a Hamilton cycle.  
\end{proof}

\subsubsection{Putting everything together}\label{sec:comb}

In this section we combine all the tools from previous sections to prove that regular tripartite tournaments that are $\eps$-close to $\c G_\beta$ are approximately Hamilton decomposable. 

As previously discussed, in the first part of the proof we will construct some linear forests with the aim of closing them into Hamilton cycles. The following fact says that whenever these linear forests are bidirectionally balanced, their startpoints and endpoints are evenly divided among the sets of the partition. Recall from \autoref{sec:prelim:notation} that if $\c F$ is a linear forest and $X$ is a set, then $X^+(\c F)$ is the set of vertices $v \in X$ such that $d^+_{\c F}(v) =0$ ($X^-(\c F)$ is defined similarly).

\begin{fact}\label{obs:endpoints}
Let $G$ be a tripartite oriented graph on vertex classes $V_1 \cup V_2 \cup V_3$, each of size $n$, and let $\c F$ be a bidirectionally balanced linear forest in $G$. Then \[|V_1^+(\c F)| = |V_2^+(\c F)| = |V_3^+(\c F)| = |V^-_1(\c F)| = |V^-_2(\c F)| = | V^-_3(\c F)|.\]
If in addition $G \in \c G_\beta(\ora{V_1}, \ola{V_1}; V_2, V_3)$, then $|\ora{V_1}^+(\c F)| = |\ora{V_1}^-(\c F)|$ and $|\ola{V_1}^+(\c F)| = |\ola{V_1}^-(\c F)|$. 
\end{fact}

\begin{proof}
Let $a= e_{\c F}(V_1, V_2) = e_{\c F}(V_2, V_3) = e_{\c F}(V_3, V_1)$ and $b = e_{\c F}(V_3, V_2) = e_{\c F}(V_2, V_1) = e_{\c F}(V_1, V_3)$. Note that $| V_1 \sm V_1^-(\c F)| = e_{\c F}(V_2, V_1) + e_{\c F}(V_3, V_1) = a + b$, so that $|V_1^-(\c F)| = m - a - b$. By symmetry, the same holds for any $i \in [3]$ and also with $V^-_i(\c F)$ replaced by $V^+_i(\c F)$. 

For the second part, note that since $G \in \c G_{\beta}$ we have
\[|\ora{V_1}^-(\c F)| = |\ora{V_1}| - e_{\c F}(V_3, V_1) = |\ora{V_1}| - e_{\c F}(V_1, V_2) = |\ora{V_1}^+(\c F)|. \]
A symmetrical argument shows that the analogous statement holds for $\ola{V_1}$. 
\end{proof}

We are now ready to prove \autoref{thm:main_oriented}. As discussed at the start of \autoref{sec:ori:g_beta}, we first prove this lemma under the additional assumption that either $\eps \ll \beta$ or $\beta = 0$. That is, we prove the following statement.

\begin{lemma}\label{lem:main_ori_specific}
        Let $1/n \ll \eps \ll \eps', \delta \leq 1$ and let $\beta \in \{0\} \cup [\eps', 1/2]$. Let $G$ be a $3n$-vertex regular tripartite tournament that is $\eps$-close to $\cal G_{\beta}$. Then $G$ contains at least $(1- \delta)n$ edge-disjoint Hamilton cycles.
\end{lemma}

Before proving this lemma, let us show how to derive the general statement from it. 

\begin{proof}[Proof of \autoref{thm:main_oriented} from \autoref{lem:main_ori_specific}]
Introduce a new constant $\eps'$ with the new hierarchy being
\[1/n \ll \eps \ll \eps' \ll \delta \leq 1.\]
Let $G$ be a regular tripartite tournament on $3n$ vertices that is $\eps$-close to $G' \in \c G_\beta$. 

If $\beta \geq \eps'$, then $G$ contains at least $(1-\delta)n$ edge-disjoint Hamilton cycles by \autoref{lem:main_ori_specific} (applied with each constant playing its own role). 

Now let us consider the case $\beta \leq \eps'$. Observe from \autoref{fig:G_beta} that changing the orientation of $3\beta n^2$ edges is enough to turn $G'$ into the $n$-blow-up of the cyclic triangle, $\ora{C_3}(n)$. Thus, $G'$ is $\beta$-close to $\c G_0$, and, in turn, $G$ is $(\eps + \beta)$-close to $\c G_0$. But $\eps + \beta \leq 2\eps'$, and thus we can apply \autoref{lem:main_ori_specific} with $1/n, 2\eps', \delta, \delta, 0$ playing the roles of $1/n, \eps, \eps', \delta, \beta$. This yields $(1-\delta)n$ edge-disjoint Hamilton cycles, as desired. 
\end{proof}
\begin{proof}[Proof of \autoref{lem:main_ori_specific}]
    We introduce constants $ \gamma, \eta > 0$ and $K \geq 1$ according to the hierarchy 
    \[\frac{1}{n} \ll \eps \ll \gamma \ll \frac{1}{K} \leq \eta \ll \eps', \delta \leq 1. \]
    The tournament $G$ is tripartite with parts $V_1, V_2,$ and $V_3$, each of size $n$. By relabeling the parts if necessary, we may assume that $G$ is $\eps$-close to some $G' \in \c G_\beta(\ora{V_1}, \ola{V_1}; V_2, V_3)$ with $V_1 = \ora{V_1} \cup \ola{V_1}$ as shown in \autoref{fig:G_beta}.

    Let $\ell = \lceil (1- \delta)n \rceil$. Let $U^\gamma = U^\gamma(G; G')$ be the set of exceptional vertices of $G$ with respect to $G'$. We will now set aside a collection of $\ell$ small forests covering the vertices in $U^\gamma$. We require the forests to have slightly different properties depending on whether $\beta = 0$ or $\beta \geq \eps'$. Throughout the proof, we will be handling both of these cases `in parallel' since the proof is the same aside from some small (though crucial) differences. 
    
    If $\beta = 0$, then $G'$ is just $\ora{C_3}(n)$, so we can apply \autoref{lem:bad_forests} to $G$ with $n, \eps, \gamma, \delta$ playing their own roles. 
    
    If $\beta \geq \eps'$ instead, we apply \autoref{lem:except_vts_Gbeta} with $n, \eps, \gamma, \beta, \ell$ playing their own roles and $\delta$ in place of $\eta$ (note that in this regime we have $\gamma \ll \eps' \leq \beta$). 
    
    Regardless of the value of $\beta$ (and thus regardless of which lemma we applied), we obtain edge-disjoint linear forests $\c F_1, \dots, \c F_{\ell}$ in $G$ satisfying, for $i\in[\ell]$,

    \begin{enumerate}[label = (U\arabic*)]
        \item\label{prop:u1} $e(\c F_i) \leq \gamma n$, 
        \item\label{prop:u2} $d^\pm_{\c F_i}(v) = 1$ for each $v \in U^\gamma$.
    \end{enumerate} 
    In addition to the two properties above, we also get
    \begin{enumerate}[label = (U\arabic*)]
    \setcounter{enumi}{2}
        \item\label{prop:u3} if $\beta = 0$, each $\c F_i$, $i\in[\ell]$, is counterclockwise-balanced,  
    \end{enumerate}
    whereas \ref{prop:halfbalanced} and \ref{prop:v2-v3} in \autoref{lem:except_vts_Gbeta} yield
    \begin{enumerate}[label = (U\arabic*$'$)]
    \setcounter{enumi}{2}
        \item\label{prop:u3'} if $\beta \geq \eps'$, we have $e_{\c F_i}(V_3, V_1) = e_{\c F_i}(V_1, V_2)$ and $e_{\c F_i}(V_2, V_1) = e_{\c F_i}(V_1, V_3)$ for every $i\in[\ell]$, 
        \item\label{prop:u4'} if $\beta \geq \eps'$, each path in $\c F_i$, $i\in[\ell]$, has its startpoint in $V_3$ and endpoint in $V_2$. 
    \end{enumerate}

    Now we apply \autoref{lem:forest_cleaner} with each $\c F_i$ playing its own role and $n, \gamma, \delta/2$ playing the role of $n, \eps, \eta$. This yields edge-disjoint linear forests $\c F'_1, \dots, \c F'_{\ell}$ with $\c F_i \subseteq \c F'_i$ and a set $U^* \supseteq U^\gamma$ satisfying
    \begin{enumerate}[label = (F\arabic*)]
        \item\label{main_ori:prop:f1} $e(\c F'_i) \leq \gamma^{1/2}n$, 
        \item\label{main_ori:prop:f2} $\c F'_i \setminus \c F_i$ is a bidirectionally balanced subgrah of $G \cap G'$, 
        \item\label{main_ori:prop:f3} letting $\c F' = \bigcup_{i \in [\ell]} \c F'_i$, we have $d^\pm_{\c F'}(v) = \ell$ for each $v \in U^*$ and $d^\pm_{\c F'}(v) \leq \gamma^{1/4}n $ for each $v \in V(G) \setminus U^*$. 
    \end{enumerate}

    Now we apply \autoref{lem:partition} to $G$ with $n, K,$ and $\eta$ playing their own roles (without removing the edges of each $\c F'_i$ from $G$). This yields $K^3$ edge-disjoint spanning subgraphs $H_1, \dots, H_{K^3}$ satisfying \ref{prop:partition}--\ref{prop:degree_between}. Let us partition the collection $\{\c F'_i\}_{i \in [\ell]}$ arbitrarily into sets $\mathbf{F}_1, \dots, \mathbf{F}_{K^3}$ of size $\frac{\ell}{K^3} \pm 1$. In the remainder of the proof, we will restrict our attention to a single $H_i$ and $\mathbf{F}_i = \{\c F'_{i_1}, \dots, \c F'_{i_m}\}$ for $i \in [K^3]$, where $m = \frac{\ell}{K^3} \pm 1$. We will show that the oriented graph $(H_i \setminus \c F') \cup \bigcup_{j \in [m]} \c F'_{i_j}$ contains $m$ edge-disjoint Hamilton cycles, each incorporating a unique $\c F'_{i_j}$. This is enough to finish the proof, as it yields $\sum_{i \in [K^3]} |\mathbf{F}_i| = \ell$ edge-disjoint Hamilton cycles in total. 

    So, let us fix some $i \in [K^3]$. For ease of notation, we henceforth omit the subscript $i$ and will be referring to $H_i$ and $\mathbf{F}_i = \{\c F_{i_1}, \dots, \c F_{i_m}\}$ as $H$ and $\mathbf{F} = \{\c F_1, \dots, \c F_m\}$. Similarly, we will refer to the sets $W_i$ and $X_i$ guaranteed by \autoref{lem:partition} as $W$ and $X$. Note, however, that $\c F'$ will still be referring to the oriented graph $\bigcup_{i \in [\ell]}\c F'_i$. 

    Let us write $W_1, W_2, W_3$ and $X_1, X_2, X_3$ in place of $W \cap V_1, W \cap V_2, W \cap V_3$ and $X \cap V_1, X \cap V_2, X \cap V_3$. Our application of \autoref{lem:partition} guarantees that
    \begin{enumerate}[label = (P\arabic*)]
        \item\label{prop:p1} $V(G)$ partitions into $ W$ and $X$ with $|W_1| =|W_2| = |W_3| = n/K^2 \pm 1$, 
        \item\label{prop:p2} for each $v \in W$ and $k \in [3]$, we have
        \[d^\pm_H(v, W_k) = \left( \frac{d^\pm_G(v, V_k)}{n} \pm \frac{13}{K} \right) |W_k|,\]
        \item\label{prop:p3} there is some $r = \left(\frac{1 \pm 4\eta}{K^3} \right) n$ such that for each $v \in X$, $d^\pm_{H[X]}(v) = r \pm n^{4/7}$, and
        \item\label{prop:p4} for each $v \in X$, $d^\pm_{H}(u, W) \geq \frac{\eta |W|}{30K}$.
    \end{enumerate}

    % Observe from \ref{prop:u1} that $V(\c F_i) \leq 2\gamma n$. For each $i \in [m]$ and $k \in [3]$, let $Z^k_i$ be a random subset of $X_k \setminus V(\c F_i)$ of size $|V(\c F) \setminus X_k|$, chosen according to the uniform distribution and independently of other $Z^{k'}_{i'}$. Note that each vertex of $X_k$ belongs to $Z^k_i$ with probability at most $\frac{|V(\c F) \setminus X_k|}{|X_k \setminus V(\c F_i)|} \leq 4\gamma$. By Chernoff's inequality for the hypergeometric distribution (\autoref{lem:chernoff}), with probability at least $1 - 3n e^{-\sqrt{n}}$, 
    % \begin{equation}\label{eq:containedinZi}
    %     |\{i \in [m]: v \in Z_i^k\}| \leq 8\gamma m,
    % \end{equation}
    % simultaneously for each choice of $k \in [3]$ and $ v \in X_k$ -- we fix a choice of each $Z_i^k$ satisfying this inequality. For $i \in [m]$ and $k \in [3]$, let $Y_i^k = X_k \setminus (V(\c F) \cup Z_i^k)$ and , and note that \begin{equation}\label{eq:sizeofYi}|Y_i^1| = |Y_i^2| = |Y_i^3| = \frac{|U|}{3} - |V(\c F)| \geq (1 - 2K^{-2} -2\gamma)n.\end{equation}
    % Let us write $Z_i = Z_i^1 \cup Z_i^2 \cup Z_i^3$ and $Y_i = Y_i^1 \cup Y_i^2 \cup Y_i^3$. 

    Our next step is to apply \autoref{lem:balanced_covers2} with $G, G',X,H[X], \c F', U^*$ playing the roles of $G, G', V', H, R, U^*$. We take the family of sets  $\{V(\c F'_i)\}_{i \in [m]}$ to play the role of the family $\{S_i\}$. The numerical parameters are $n, m, 2\gamma^{1/4}, \eps', r/n, \delta/4, \beta$ in place of $n, \ell, \eps, \eps', \gamma, \eta, \beta$. 
    
    Let us quickly verify that the assumptions of the lemma are met with these choices. First note that
    \begin{equation}\label{eq:boundonm}m \leq \frac{\ell}{K^3} + 1 \leq \left(1 -\frac{\delta}{2}\right)\frac{n}{K^3}\leq \left(1-\frac{\delta}{4}\right)\left(\frac{1 - 4\eta}{K^3}\right)n \leq \left(1- \frac{\delta}{4}\right)r.\end{equation}
    We have \(|X_1| = |X_2| = |X_3| \geq (1 - 2K^{-2})n \geq (1 - \eps'/2)n\) by \ref{prop:p1}. To verify \eqref{eq:minmaxdeg}, note that in/outdegrees in $H[X]$ are equal to $r \pm n^{4/7}$ by \ref{prop:p3} and $r/n$ is playing the role of $\gamma$. Finally, $|V(\c F'_i)| \leq 2e(\c F'_i) \leq 2\gamma^{1/2}n$ by \ref{main_ori:prop:f1}, and each  $v \in X \setminus U^*$ appears in at most $2\gamma^{1/4}n$ forests $\c F'_i$ by \ref{main_ori:prop:f3}.

    The lemma thus yields edge-disjoint linear forests $\c L_1, \dots, \c L_m$ on vertex set $X$ such that
    \begin{enumerate}[label = (L\arabic*)]
    \setcounter{enumi}{-1}
        \item\label{main_ori:prop:l1} $\c L_i$ is bidirectionally balanced, 
        \item\label{main_ori:prop:l2} $e(\c L_i) \geq |X| - 10 \gamma^{1/32}n$, 
        \item\label{main_ori:prop:l3} $\c L_i \subseteq (H[X]\cap G') \setminus \c F'$, 
        \item\label{main_ori:prop:l4} $V(\c L_i) \cap V(\c F'_i) = \emptyset$, 
        \item\label{main_ori:prop:l5} letting $\c L = \bigcup_{i \in [m]} \c L_i$, we have $d^\pm_{\c L}(v) \geq m - 4 \gamma^{1/64}n$ for each $v \in X \setminus U^*$.
    \end{enumerate}

    Observe that \ref{main_ori:prop:l4} implies that $\c G_i = \c F'_i \cup \c L_i$ is a linear forest. Moreover, \ref{main_ori:prop:l3} implies that $\c G_1, \dots, \c G_m$ are pairwise edge-disjoint. We will now consider the $\c G_i$ one at a time, and show that each can be completed to a Hamilton cycle $\c C_i$. Once we find $\c C_i$, we update $H$ by removing the edges of $\c C_i$ from it and continue onto $\c G_{i+1}$. To be precise, suppose that for some $m' \leq m$ we have already constructed Hamilton cycles $\c C_1, \dots, \c C_{m'-1}$ with $\c G_i \subseteq \c C_i$ for each $i \in [m'-1]$. Let $H'$ be obtained from $H$ by removing each edge that is contained in some $\c C_i$, $i \in [m'-1]$, in some $\c G_i$, $m'\le i \le m$, or in $\c F'$. We will show how to extend $\c G_{m'}$ to a Hamilton cycle only using edges of $H'$. 

  First, let us quantify the number of available edges from each vertex of $X \setminus U^*$ to $W$ in $H'$. First note that, for each $i \in [m]$, the linear forest $\c G_i \setminus \c F'_i = \c L_i$ is entirely contained in $H[X]$, and therefore $\c G_i[X, W] = \c F'_i [X, W]$. So, any edge $uv \in H \setminus H'$ with $u \in X, v \in W$ either belongs to $\c F'$ or it belongs to $\c C_i \setminus \c G_i$ for some $i \in [m'-1]$. The same is true of any edge $uv \in H \setminus H'$ with $u \in W, v \in X$. Thus, for each $v  \in X \setminus U^*$ and $\diamond = \pm$, we have
    \begin{equation}\label{eq:availdegreeinX}
        \begin{split}
            d^\diamond_{H'}(v, W) &\geq d^\diamond_H(v, W) - d^\diamond_{\c F'}(v) - |\{ i \in [m'-1]: d^\diamond_{\c G_i}(v) = 0 \}| \\
            &\geq \frac{\eta |W|}{30K} - \gamma^{1/4}n - 4\gamma^{1/64}n \geq \frac{n}{K^5},
        \end{split}
    \end{equation}
    where in the second inequality we used \ref{prop:p4}, \ref{main_ori:prop:f3}, and \ref{main_ori:prop:l5}. 
    
    Any edge $uv \in H[W] \setminus H'$ is also either contained in $\c F'$ or in $\c C_i \setminus \c G_i$ for some $i \in [m'-1]$, again using the fact that $\c L_i \subseteq H[X]$. So, for each $v \in W \setminus U^*, \diamond = \pm,$ and $k \in [3]$, we have
    \begin{equation}\label{eq:availdegreeinW}
    \begin{split}
        d_{H'}^\diamond(v, W_k) &\geq d_H^\diamond(v, W_k) - d_{\c F'}^\diamond(v) -  m \\ &\geq \left( \frac{d^\diamond_G(v, V_k)}{n} -\frac{13}{K}\right)|W_k| - \gamma^{1/4}n - \frac{2n}{K^3} \\ &\geq \left( \frac{d^\diamond_G(v, V_k)}{n} -\frac{18}{K}\right)|W_k|,
        \end{split}
    \end{equation}
    where in the second inequality we used \ref{prop:f3}, \ref{prop:p2}, \ref{prop:p3} and \eqref{eq:boundonm}, and in the third inequality we used \ref{prop:p1}. Inequalities \eqref{eq:availdegreeinX} and \eqref{eq:availdegreeinW} are the (only) lower bounds on degrees in $H'$ that we will be using to extend $\c G_{m'}$ into $\c C_{m'}$ in $H'$.

    We now extend $\c G_{m'}$ into a linear forest with additional properties needed for a completion to a Hamilton cycle.

    \begin{claim}\label{clm:extendtoG'}
        There is a linear forest $\c G'_{m'}$ which extends $\c G_{m'}$ using edges of $H'$ such that 
        \begin{enumerate}[label = (G\arabic*)]
            \item\label{prop:g1} $d_{\c G'_{m'}}^\pm(v) = 1$ for all $v \in X$, 
            \item\label{prop:g2} $\c G'_{m'}$ has all its startpoints in $W_3$ and endpoints in $W_2$, and 
            \item\label{prop:g3} $|V(\c G'_{m'}) \cap W| \leq 22\gamma^{1/32}n$.
        \end{enumerate}
        Furthermore,
        \begin{enumerate}[label = (G\arabic*)]
        \setcounter{enumi}{3}
            \item\label{prop:g4} if $\beta = 0$, then $\c G'_{m'}$ is counterclockwise-balanced,
        \end{enumerate}
        \begin{enumerate}[label = (G\arabic*$'$)]
        \setcounter{enumi}{3}
            \item\label{prop:g4'} if $\beta \geq \eps'$, we have $e_{\c G'_{m'}}(V_2, V_1) = e_{\c G'_{m'}}(V_1, V_3).$
        \end{enumerate}
    \end{claim}

    \begin{proof} Recall that $\c G_{m'} = \c F_{m'} \cup \c L_{m'}$. Observe that $\c G_{m'}$ satisfies \ref{prop:g4} and \ref{prop:g4'}: If $\beta = 0$, then $\c G_{m'}$ is counterclockwise-balanced by \ref{prop:u3}, \ref{main_ori:prop:f2}, and \ref{main_ori:prop:l1}. If $\beta \geq \eps'$, then $\c G_{m'}$ satisfies $e_{\c G_{m'}}(V_2, V_1) = e_{\c G_{m'}}(V_1, V_3)$ by \ref{prop:u3'}, \ref{main_ori:prop:f2}, and \ref{main_ori:prop:l1}. So,we want to extend $\c G_{m'}$ while preserving properties \ref{prop:g4} and \ref{prop:g4'}.

    Let $Y^+$ be the set of endpoints of paths in $\c G_{m'}$ (i.e. vertices with outdegree $0$ and indegree $1$) which do not lie in $W_2$. Similarly, let $Y^-$ be the set of startpoints of paths in $\c G_{m'}$ which are not in $W_3$. Finally, let 
    $Y^0 = \{v\in X: d_{\c G_{m'}}^\pm(v) = 0\}.$
    Note that $Y^+, Y^-,$ and $Y^0$ are pairwise disjoint. We will extend $\c G_{m'}$ by appending paths at each vertex of $Y^+ \cup Y^- \cup Y^0$. 
    
    Let $y_1, \dots, y_t$ be an enumeration of the elements of $Y^+\cup Y^- \cup Y^0$. Since $\c L_{m'}$ is a subgraph of $H[X]$, we have $\c G_{m'}[W] = \c F'_{m'}[W]$. Recall from \autoref{sec:prelim:notation} that $X^\diamond(\c G_{m'})$, $\diamond = \pm,$ is the set of elements $v \in X$ satisfying $d^\diamond_{\c G_{m'}}(v) = 0$. For each $\diamond = \pm$, we have $Y^\diamond \cap X \subseteq X^+(\c G_{m'})$ and $Y^\diamond \cap W \subseteq V(\c G_{m'}) \cap W \subseteq V(\c F'_{m'})$. Also, $Y^0 \subseteq X^+(\c G'_{m'})$. Therefore,
    \begin{equation}\label{eq:sizeoft}
    \begin{split}t &= |Y^+ \cup Y^- \cup Y^0| \\&\leq |V(\c F'_{m'})| +  |X^+(\c G_{m'})| + |X^-(\c G_{m'})| \\ &\leq 2\gamma^{1/2}n + 20\gamma^{1/32}n \leq 21\gamma^{1/32}n,\end{split}\end{equation}
    where in the second inequality we used \ref{prop:f1} and \ref{main_ori:prop:l2}. 

    Now we construct a family of vertex-disjoint paths $P_1, \dots, P_t$ (one for each of $y_1, \dots, y_t$) satisfying the following properties. 
    \begin{enumerate}[label = (E\arabic*)]
        \item\label{main_ori:prop:e1} $e(P_i) \leq 10$;
        \item\label{main_ori:prop:e2} $P_i \subseteq H' \cap G'$ and $V(P_i) \setminus \{y_i\} \subseteq W \setminus V(\c G_{m'})$;
        \item\label{main_ori:prop:e3} $y_i$ is $P_i$'s startpoint if $y_i \in Y^+$, it is its endpoint if $y_i \in Y^-$, and it is internal to it if $y_i \in Y^0$;
        \item\label{main_ori:prop:e4} $P_i$'s endpoint lies in $W_2$ if $y_i \in Y^+$, its startpoint lies in $W_3$ if $y_i \in Y^-$, and both of these are true if $y_i \in Y^0$;
        \item\label{main_ori:prop:e5} $P_i$ is counterclock\-wise-balanced if $y_i \in \ora{V_1} \cup V_2 \cup V_3$, and $P_i$ consists entirely of counterclockwise edges if $y_i \in \ola{V_1}$. 
    \end{enumerate}

    We construct these paths iteratively. Suppose that we have already constructed $P_1, \dots, P_{s-1}$ satisfying \ref{main_ori:prop:e1}--\ref{main_ori:prop:e5}, and consider $y_s$. 

    Note that the number of vertices of $W$ that belong to $V(\c G_{m'})$ or $V(P_i)$ for some $i \in [s-1]$ is at most
    \begin{equation}\label{eq:used_vts}
        |V(\c F'_{m'})| + \sum_{i \in [s-1]}|V(P_i)|
        \stackrel{\mathrm{\ref{main_ori:prop:f1}}}{\leq} 2\gamma^{1/2}n + 11t
        \stackrel{\mathrm{\eqref{eq:sizeoft}}}{\leq} 22\gamma^{1/32} n.
    \end{equation}
%    where in the second inequality we used \ref{main_ori:prop:f1} and in the third we used \eqref{eq:sizeoft}. 
    
    First suppose that $y_s \in Y^+$. Since $y_s$ does not have both positive in and outdegree in $\c G_{m'}$, we have $y_s \notin U^\gamma$ by \ref{prop:u2}. So, if $y_s \in X$ we have 
    \[d^+_{H' \cap G'}(v, W ) \geq \frac{n}{K^5} - \gamma n > 22\gamma^{1/32}n\]
    by \eqref{eq:availdegreeinX}, and we can thus pick an unused outneighbour $w_1 \in W$ of $y_s$ in $H' \cap G'$. If $y_s \in W_k$ for some $k \in [3]$ instead, we have
    \begin{equation}\begin{split}d^+_{H' \cap G'}(y_s, W) &\geq d^+_{H' \cap G'}(y_s, W_{k+1}) + d^+_{H' \cap G'}(y_s, W_{k+2}) \\ &\geq\left(\frac{d^+_{G}(y_s, V_{k+1}) + d^+_{G}(y_s, V_{k+2})}{n} - \frac{36}{K} \right)\frac{|W|}{3} -2\gamma n\\  &= \left(\frac{1}{3} - \frac{12}{K}\right)|W| - 2\gamma n > 22\gamma^{1/32}n,\end{split}\end{equation}
    where in the second inequality we used $y_s \notin U^\gamma$ and \eqref{eq:availdegreeinW}, and in the third we used the fact that $y_s$ has outdegree $n$ in $G$. So, again in this case we can pick an unused outneighbour $w_1 \in W$ of $y_s$ in $H' \cap G'$. 

    If $y_s w_1$ is a counterclockwise edge, then it follows from $y_s w_1 \in G'$ that $G' \neq \ora{C_3}(n)$. This forces $\beta \geq \eps'$ and $w_1 \in \ola{V_1} \cup V_2 \cup V_3$. But then $w_1$ has at least $\beta n \geq \eps' n$ counterclockwise outneighbours in $G'$, and at least $(\eps' - \gamma)n$ counterclockwise outneighbours in $G \cap G'$ since $w_1 \notin U^\gamma \subseteq V(\c G_{m'})$. So, by \eqref{eq:availdegreeinW} and \eqref{eq:used_vts}, $w_1$ has at least $(\eps'- \gamma - 18/K)\frac{|W|}{3} - 22\gamma^{1/32}n > 0$ unused counterclockwise outneighbours in $W$ in the graph $H' \cap G'$. We pick $w_2$ to be one such vertex and, by repeating the argument, we pick an unused $w_3 \in W$ such that $w_2w_3 \in H' \cap G'$ is a counterclockwise edge. Observe that this yields a counterclockwise path $y_sw_1w_2w_3$. 
    
    If $y_s w_1$ is a clockwise edge instead, then it follows from similar arguments that $w_1 \in \ora{V_1} \cup V_2 \cup V_3$ and that $w_1$ has at least $(1 -\beta - \gamma)n \geq (1/2 - \gamma)n$ clockwise outneighbours in $G \cap G'$. We again apply \eqref{eq:availdegreeinW} and \eqref{eq:used_vts}, this time to construct a clockwise path $y_sw_1w_2w_3$, which in this case works regardless of the value of $\beta$.

    So far, we have constructed a path $y_sw_1w_2w_3$ that is either entirely clockwise or counterclockwise. Now we further extend this path so that \ref{main_ori:prop:e1}--\ref{main_ori:prop:e5} are satisfied. If $y_s \in \ora{V_1} \cup V_2 \cup V_3$, then $w_3 \in \ora{V_1} \cup V_2 \cup V_3$ since the path $y_sw_1w_2w_3$ is contained in $G'$. If $w_3 \in W_2$, we let $P_s = y_sw_1w_2w_3$ and finish this step. If not, again by combining $w_3 \notin U^\gamma$, \eqref{eq:availdegreeinW}, and \eqref{eq:used_vts}, we can pick an unused clockwise outneighbour $w_4 \in W$ of $w_3$ in $H' \cap G'$ and, if $w_4 \notin W_2$, an unused clockwise outneighbour $w_5$ of $w_4$. Observe that either $w_4$ or $w_5$ has to lie in $W_2$. Depending on which, we either let $P_s = y_sw_1w_2w_3w_4$ or $y_sw_1w_2w_3w_4w_5$. As promised, this path is counterclockwise-balanced since $y_sw_1w_2w_3$ is bidirectionally balanced and $w_3w_4w_5$ is entirely clockwise. 

    If $y_s \in \ola{V_1}$ instead (which, as we saw, forces $\beta \geq \eps'$), then $y_sw_1w_2w_3$ must be a counterclockwise path since it is a subgraph of $G'$. Using similar arguments to the above, we can pick unused $w_4 \in W_3$ and $w_4 \in W_2$ such that $w_3w_4, w_4w_5 \in H' \cap G'$. Letting $P_s = y_sw_1w_2w_3w_4w_5$ (which is a counterclockwise path) completes this step. 

    We have now described in detail how to construct $P_s$ if $y_s \in Y^+$. If $y_s \in Y^-$, repeating the argument in a symmetrical fashion yields a path $P_s = w'_1 \dots w'_{j} y_s$ where $w'_1 \in W_3$, $j \leq 5$, and $P_s$ satisfies \ref{main_ori:prop:e5}. If $y_s \in Y^0$, we again use this argument twice over to construct a path $P_s = w'_1 \dots w'_{j'} y_s w_1 \dots w_j$ where $w'_1 \in W_3, w_j \in W_2,$ $j, j' \leq 5$ and $P_s$ satsifies \ref{main_ori:prop:e5}. Note that \ref{main_ori:prop:e1}--\ref{main_ori:prop:e5} are clearly satisfied, and so this completes the construction of $P_s$. 

    The above procedure yields paths $P_1, \dots, P_t$ satisfying \ref{main_ori:prop:e1}--\ref{main_ori:prop:e5}. We claim that $\c G'_{m'} = \c G_{m'} \cup \bigcup_{i \in [t]} P_i$ satisfies \ref{prop:g1}--\ref{prop:g4} and \ref{prop:g4'}. Note that \ref{prop:g1} and \ref{prop:g2} follow from the construction of $\c G'_{m'}$ and properties \ref{main_ori:prop:e3} and \ref{main_ori:prop:e4}. For \ref{prop:g3}, we have
    \[|V(\c G'_{m'}) \cap W| \leq |V(\c F'_{m'})| + \sum_{i \in [t]}|V(P_i)| \leq 2\gamma^{1/2}n + 11t \leq 22\gamma^{1/32}n,\]
    where we used \ref{prop:f1} and \eqref{eq:sizeoft} combined with the fact that $V(\c G_{m'}) \cap W = V(\c F'_{m'})\cap W$. For \ref{prop:g4}, observe that if $\beta = 0$ we have $\ola{V_1} = \emptyset$. Hence,  \ref{main_ori:prop:e5} combined with the fact that in this case $\c G_{m'}$ is counterclockwise-balanced (which was argued at the start of the proof of \autoref{clm:extendtoG'}) yields that $\c G'_{m'}$ is counterclockwise-balanced. 

    It remains to argue for \ref{prop:g4'}. Assume that $\beta \geq \eps'$. As noted earlier, in this case we have $e_{\c G_{m'}}(V_2, V_1) = e_{\c G_{m'}}(V_1, V_3)$, and so we want to show that adding $P_1, \dots, P_t$ to $\c G_{m'}$ maintains this property. We begin by observing that each $P_i$ with $y_i \in \ora{V_1}\cup V_2 \cup V_3$ is counterclockwise-balanced by \ref{main_ori:prop:e5} and thus has no effect on this property. So it suffices to consider $y_i \in \ola{V_1}$. For each of these, $P_i$ is an entirely counterclockwise path by \ref{main_ori:prop:e5}. It is easy to see using \ref{main_ori:prop:e3}--\ref{main_ori:prop:e5} that
    \begin{enumerate}[label = (\roman*)]
        \item if $y_i \in Y^+$, $e_{P_i}(V_2, V_1) = e_{P_i}(V_1,V_3) -1$, 
        \item if $y_i \in Y^-$, $e_{P_i}(V_2, V_1) = e_{P_i}(V_1, V_3)+1$, and 
        \item if $y_i \in Y^0$, $e_{P_i}(V_2, V_1) = e_{P_i}(V_1, V_3)$.        
    \end{enumerate}

    It follows that
    \begin{equation}\label{eq:imbalance}
    \begin{split}
        e_{\c G'_{m'}}(V_2, V_1) - e_{\c G'_{m'}}(V_1, V_3) &=  \sum_{y_i \in \ola{V_1}} (e_{P_i}(V_2, V_1) - e_{P_i}(V_1, V_3)) \\
        &= |\ola{V_1} \cap Y^-| - |\ola{V_1} \cap Y^+|
    \end{split}
    \end{equation}
    We now prove that 
    \[|\ola{V_1} \cap Y^+| = |\ola{V_1} \cap Y^-|.\]
    This, together with \eqref{eq:imbalance}, implies \ref{prop:g4'}, completing the proof of  \autoref{clm:extendtoG'}.   
    
    By definition, $\ola{V_1} \cap Y^+$ (resp. $\ola{V_1} \cap Y^-$) is precisely the set of endpoints (startpoints) of paths in $\c G_{m'}$ which lie in $\ola{V_1}$, i.e. vertices $v \in \ola{V_1}$ which satisfy $d_{\c G_{m'}}^+(v)=0$ and $d^-_{\c G_{m'}}(v) = 1$ (resp. $d_{\c G_{m'}}^+(v)=1$ and $ d^-_{\c G_{m'}}(v) = 0$). However, each path in $\c F_{m'}$ starts in $V_3$ and ends in $V_2$ by \ref{prop:u4'}, and so $d^+_{\c G_{m'}}(v) = 0$ and $d^-_{\c G_{m'}}(v) = 1$ iff $d^+_{\c G_{m'} \setminus \c F_{m'}}(v) = 0$ and $d^-_{\c G_{m'} \setminus \c F_{m'}}(v) = 1$ for each $v \in \ola{V_1}$. Similarly, $d^+_{\c G_{m'}}(v) = 1$ and $d^-_{\c G_{m'}}(v) = 0$ iff $d^+_{\c G_{m'} \setminus \c F_{m'}}(v) = 1$ and $d^-_{\c G_{m'} \setminus \c F_{m'}}(v) = 0$ for each $v \in \ola{V_1}$. Then, recalling that $\ola{V_1}^\diamond(\c G_{m'} \setminus \c F_{m'})$ for $\diamond = \pm$ is precisely the set of vertices $v\in \ola{V_1}$ satisfying $d^\diamond_{\c G_{m'} \setminus \c F_{m'}}(v) = 0 $, we have 
        \begin{equation}\label{eq:imbalance:2}
    \begin{split}
    \ola{V_1} \cap Y^+    
     &= \ola{V_1}^+(\c G_{m'} \setminus \c F_{m'}) \setminus \ola{V_1}^- (\c G_{m'} \setminus \c F_{m'}),\\
    \ola{V_1} \cap Y^- &= \ola{V_1}^-(\c G_{m'} \setminus \c F_{m'}) \setminus \ola{V_1}^+ (\c G_{m'} \setminus \c F_{m'}).
      \end{split}
    \end{equation}
    However $\c G_{m'} \setminus \c F_{m'} = (\c F'_{m'} \setminus \c F_{m'} ) \cup \c L_{m'}$ is a bidirectionally balanced subgraph of $G'$ by \ref{main_ori:prop:f2}, \ref{main_ori:prop:l1}, and \ref{main_ori:prop:l3}. So, by the second part of \autoref{obs:endpoints}, 
    \[| \ola{V_1}^+ (\c G_{m'} \setminus \c F_{m'})| = |\ola{V_1}^-(\c G_{m'} \setminus \c F_{m'})|,\]
    which combined with \eqref{eq:imbalance:2} implies 
    $|\ola{V_1} \cap Y^+| = |\ola{V_1} \cap Y^-|$.\end{proof} 

    To finish the proof of \autoref{lem:main_ori_specific}, let $S_1, \dots, S_p$ be an enumeration of the (path) components of $\c G'_{m'}$. Let us write $a_i$ and $b_i$ to denote the startpoint and endpoint (respectively) of $S_i$. We have $a_i \in W_3$ and $b_i \in W_2$ by \ref{prop:g2}, so $p \leq 22\gamma^{1/32}n \leq \gamma^{1/33}n$ by \ref{prop:g3}. Define \(M = \{(a_i, b_i) : i \in [p]\},\) which is a matching from $W_3$ to $W_2$. 
    
    Let $W^* \subseteq W$ be obtained by removing from $W$ any vertex satisfying $d^\pm_{\c G'_{m'}}(v) = 1$. Each path in $\c G'_{m'}$ starts in $W_3$ and ends in $W_2$ by \ref{prop:g2}, so the elements of $W_3 \setminus W^*$ are precisely the vertices of $W_3$ with an inedge in $\c G'_{m'}$, and the elements of $W_2 \setminus W^*$ are those with an outedge. Thus, 
    \begin{equation}\label{eq:sizeW*} |W_3 \setminus W^*| = e_{\c G'_{m'}}(V_2, V_3) + e_{\c G'_{m'}}(V_1, V_3)= e_{\c G'_{m'}}(V_2, V_3) + e_{\c G'_{m'}}(V_2, V_1) = |W_2 \setminus W^*|,\end{equation}
    which follows from \ref{prop:g4} if $\beta = 0$ and from \ref{prop:g4'} if $\beta \geq \eps'$. 

    We will now find a Hamilton cycle $\c C'_{m'}$ in $H'[W^*]$ incorporating $M$. We distinguish two cases.

    \paragraph{Case 1: $\beta = 0$.} It follows from \ref{prop:g2} that $W_1 \setminus W^*$ is precisely the set of vertices in $W_1$ with an inneighbour in $\c G'_{m'}$. Therefore,
    \begin{equation}\label{eq:sizeofW1}
    |W_1 \setminus W^*|
     = e_{\c G'_{m'}}(V_3, V_1) + e_{\c G'_{m'}}(V_2, V_1)
     \stackrel{\mathrm{\ref{prop:g4} }}{=} 
     e_{\c G'_{m'}}(V_3, V_1) + e_{\c G'_{m'}}(V_3, V_2)
     = 
    |W_3 \setminus W^*| + p.
    \end{equation}
    So, we have $|W^* \cap W_2| = |W^* \cap W_3| = |W^* \cap W_1| + p$. 
    
    Note that the vertices of $W^*$ do not belong to $U^\gamma$ because of \ref{prop:u2}. Thus, for each $k \in [3]$, each $v \in W^* \cap W_k$ has at least $(1-\gamma)n$ outneighbours in $V_{k+1}$ in $G$ and at least $(1-\gamma)n$ inneighbours in $V_{k-1}$. This implies
    \begin{align*}
      d_{H'[W^*]}^+(v, W^* \cap W_{k+1})
      &\stackrel{\mathrm{\ref{prop:g3}}}{\geq}
      d^+_{H'}(v, W_{k+1}) - \gamma^{1/33}n\\ &\stackrel{\mathrm{\eqref{eq:availdegreeinW}}}{\geq} \left(1- \gamma - \frac{19}{K}\right) |W_{k+1}|\\ 
      &\geq \left(1 -\frac{20}{K} \right) |W_{k+1}|.
    \end{align*}
      Reasoning symmetrically, we get 
    \[d_{H'[W^*]}^-(v, W^* \cap W_{k-1}) \geq \left(1 - \frac{20}{K} \right)|W_{k-1}|.\]
    Thus, we can apply \autoref{lem:hamilton_cyclictriangle} to $H'[W^*]$ with $M$ playing its own role and $|W^* \cap V_2|$ and $ 20/K$ playing the roles of $n$ and $\eps$. Its assumptions on the relative sizes of the sets in the tripartition are met since $M$ is a matching on $p$ edges and $|W_1 \cap W^*| = |W_2 \cap W^*| - p$ by \eqref{eq:sizeofW1}. The application of the lemma yields a Hamilton cycle $\c C'_{m'}$ in $H'[W^*]$ incorporating $M$, as desired. 

    \paragraph{Case 2: $\beta \geq \eps'$.} We are going to show that we can apply \autoref{lem:hamilton_Gbeta} to the digraph $H'[W^*]$ and matching $M$. Let us start by arguing that neither of $W^* \cap \ora{V_1}$ and $W^* \cap \ola{V_1}$ is too large. Let $v^* \in W^* \cap W_3$ be chosen arbitrarily, and note that $v^* \notin U^\gamma$ by \ref{prop:u2} and the definition of $W^*$. So we have $d_{G'}^\pm(v^*, V_1) \geq \beta n$ and thus $d^\pm_{G}(v^*, V_1) \geq (\beta - \gamma)n$. In turn, this implies $d_{H'}^\pm(v^*, W_1) \geq (\beta - \gamma - 18/K)|W_1| \geq \beta |W_1|/2$ by \eqref{eq:availdegreeinW}. Note that any outedge from $v^*$ to $\ola{V_1}$ is in $G \setminus G'$, and similarly any inedge at $v^*$ coming from $\ora{V_1}$ is in $G \setminus G'$. Thus, again using the fact that $v^* \notin U^\gamma$, we have $d^-_{H'}(v^*, \ora{V_1}), d_{H'}^+(v^*, \ola{V_1}) \leq \gamma n$. Hence,  \[d_{H'}^+(v^*, W_1 \cap \ora{V_1}), d_{H'}^-(v^*, W_1 \cap \ola{V_1}) \geq \frac{\beta |W_1|}{2} - \gamma n \geq \frac{\beta |W_1|}{4}.\]
    In particular, $|W_1 \cap \ora{V_1}|, |W_1 \cap \ola{V_1}| \geq \frac{\beta |W_1|}{4}$. But $|W \setminus W^*| \leq |V(\c G'_{m'}) \cap W| \leq \gamma^{1/33}n$ by \ref{prop:g3}, and so \[|W^* \cap \ora{V_1}|, |W^* \cap \ola{V_1}| \geq \frac{\beta |W_1|}{4} - \gamma^{1/33}n \geq \frac{\beta |W_1|}{8}.\] Since $|W^* \cap \ora{V_1}| + |W^* \cap \ola{V_1}| \leq |W_1|$, we have 
    \[|W^* \cap \ora{V_1}|, |W^* \cap \ola{V_1}| \leq (1 - \beta/8) |W_1|.\]
    Note that \begin{equation}\label{eq:sizeofW2}|W^* \cap W_2| = | W^* \cap W_3| \geq |W_1| - \gamma^{1/33}n\end{equation} by \eqref{eq:sizeW*}, \ref{prop:g3}. and $|W_1|=|W_2|=|W_3|$ This yields 
    \begin{equation}\label{eq:nottoolarge}
      |W^* \cap \ora{V_1}|, |W^* \cap \ola{V_1}| \leq (1 - \beta/16)|W^*\cap W_2|,\end{equation}
    which shows that these two sets are bounded away from $|W^* \cap W_2|$ in size. Letting $n' = |W^* \cap V_2| = |W^* \cap V_3|$, it follows from \eqref{eq:availdegreeinW}, \eqref{eq:sizeofW2}, \eqref{eq:nottoolarge}, and the fact that $W^*$ contains no vertices of $U^\gamma$ that
    \begin{enumerate}[label = (D\arabic*)]
        \item $d_{H'}^+(v, W^* \cap V_2), d_{H'}^-(v, W^* \cap V_3) \geq \left(1 - \frac{80}{K}\right)n'$ for each $v \in W^* \cap \ora{V_1},$
        \item $d_{H'}^+(v, W^* \cap V_3), d_{H'}^-(v, W^* \cap V_2) \geq \left(1 - \frac{80}{K}\right)n'$ for each $v \in W^* \cap \ola{V_1}$,
        \item $d_{H'}^+(v, W^* \cap ( \ola{V_1} \cup V_3)), d^-_{H'}(v, W^* \cap (\ora{V_1} \cup V_3)) \geq \left(1 - \frac{80}{K}\right)n'$ for each $v \in W^* \cap V_2$, and 
        \item $d_{H'}^+(v, W^* \cap ( \ora{V_1} \cup V_2)), d^-_{H'}(v, W^* \cap (\ola{V_1} \cup V_2)) \geq \left(1 - \frac{80}{K}\right)n'$ for each $v \in W^* \cap V_3$. 
    \end{enumerate}
    Therefore, we can apply \autoref{lem:hamilton_Gbeta} to $H'[W^*]$ with $M$ playing its own role and $n', 80/K$ playing the roles of $n, \eps$, where \eqref{eq:nottoolarge} provides the necessary bound on the sizes of $|W^* \cap \ora{V_1}|$ and $ |W^*\cap \ola{V_1}|$. The lemma application yields a Hamilton cycle $\c C'_{m'}$ in $H'[W^*]$ incorporating $M$. 

    \medskip 

    We argued that in either case we can find a Hamilton cycle $\c C'_{m'}$ in $H'[W^*]$ incorporating $M$. For some choice of distinct indices $i_1, \dots, i_p \in [p]$, we can write $\c C'_{m'}$  as
    \[\c C'_{m'}  = a_{i_1}b_{i_1} T_1 a_{i_2}b_{i_2} T_2 a_{i_3}b_{i_3} \dots a_{i_p}b_{i_p}T_p a_{i_1},\]
    where $T_1, \dots,T_p$ are vertex-disjoint paths in $H'[W^*]$ and each vertex of $W^* \setminus V(\c G'_{m'})$ is contained in one such $T_i$. Recalling that $a_i$ and $b_i$ are the extremities of the path $S_i$ in $\c G'_{m'}$, this yields the Hamilton cycle 
    \[\c C_{m'} = a_{i_1} S_{i_1} b_{i_1} T_1 a_{i_2} S_{i_2} b_{i_2} \dots a_{i_p} S_{i_p} b_{i_p} T_p a_{i_1}\]
    in $G$. Every edge of $\c C_{m'} \setminus \c G_{m'}$ belongs to $H'$, so that $\c C_{m'}$ is edge-disjoint from $\c C_1, \dots \c C_{m'-1}$, from $\c G_{m'}, \dots, \c G_{m}$, and from $\c F'$. This finishes the construction of $\c C_{m'}$.

    The procedure carried out successfully reaches completion and yields a collection of edge-disjoint Hamilton cycles $\c C_{i_1}, \dots, \c C_{i_m}$ spanning $V(G)$, each incorporating a unique linear forest in $\textbf{F}_i$. Moreover, each $\c C_{i_j} \setminus \c F'_{i_j}$ is entirely contained in $H_i \setminus \c F'$. Since the various $H_i$ are pairwise edge-disjoint (and so are the linear forests contained inside of the various $\textbf{F}_i$), summing over different $H_i$ gives $\sum_{i \in [K^3]} |\mathbf{F}_i| = \ell$ edge-disjoint Hamilton cycles in total. \end{proof}

\subsection{Proof of \autoref{conj:ori-almost}}\label{subsec:main_thm}

We can now derive our main theorem for the oriented case.

\begin{proof}[Proof of \autoref{conj:ori-almost}]
    Let $1/n\ll \delta$ be as given and let $G$ be a $(3n)$-vertex regular tripartite tournament. Introduce additional constants $\nu, \tau, \eps$ such that
    $$1/n\ll \nu\ll \tau,\eps \text{ and }\eps \ll \delta.$$
    We may further assume that $\delta\ll 1$ as this only strengthens the theorem.

    By \autoref{lem:stab} $G$ is either a robust $(\nu,\tau)$-outexpander or it is $\eps$-close to $\c G_\beta$ for some $\beta\in[0,1/2]$. If $G$ is an expander, by \autoref{thm:ko} it has a full decomposition into Hamilton cycles. If $G$ is $\eps$-close to $\c G_\beta$, by \autoref{thm:main_oriented} it contains at least $(1-\delta)n$ edge-disjoint Hamilton cycles.
\end{proof}
\section{Concluding remarks}\label{sec:conclusion}

There is one known example of a regular tripartite tournament which is not decomposable into Hamilton cycles. This example was given in \cite{bertille} and consists of $\ora{C_3}(n)$ with the edges of a single triangle reversed. Denote this tournament by ${\c T}_\triangle$. We can see that ${\c T}_\triangle$ does not have a Hamilton decomposition from \autoref{lem:factor_balanced} which implies that any Hamilton cycle is counterclockwise-balanced. The counterclockwise edges form a cyclic triangle $\ola{C_3}$ which cannot be partitioned into counterclockwise-balanced sets of edges extendable to Hamilton cycles. We conjecture this to be the only counterexample.

\begin{conj}
    For sufficiently large $n$, every regular tripartite tournament on $3n$ vertices, except for ${\c T}_\triangle$, has a Hamilton decomposition.
\end{conj}

In the proof of Kelly's conjecture  \cite{kelly-published}, that regular tournaments are decomposable into Hamilton cycles, the authors showed that $d$-regular $n$-vertex oriented graphs with $d\ge(3/8+o(1))n$ have a Hamilton decomposition. This prompts the following question.

\begin{question}
    What is the smallest constant $c \geq 0$ such that for each $c' > c$, every sufficiently large $c'n$-regular $n$-vertex oriented graph has a Hamilton decomposition?
\end{question}

The result in \cite{kelly-published} shows that $c\le 3/8$, which may be suspected to be tight from the fact that the Hamiltonicity threshold in \cite{keevash2009} is also $3/8$. However, the extremal example from \cite{keevash2009} is not regular, so it does not exclude values of $c$ lower than $3/8$. The true value could be as low as $1/4$ (i.e. the threshold for Hamiltonicity in regular oriented graphs \cite{LPY23}), but this may be hard to prove as the common technique of building Hamilton cycles on random vertex subsets would fail for any $c<3/8$. Indeed, if $U$ is a set of vertices chosen at random, the induced subgraph on $U$ will be almost regular and not exactly regular. Thus it will not even be Hamiltonian for $c<3/8$, preventing a lot of existing approaches from generating full Hamilton cycles. 

The same question in the bipartite setting was asked in \cite[Conjecture 4.2]{liebenau-published}. The general partite tournament version of this question would be the following.

\begin{question}
    Let $r\ge 2$. What is the smallest constant $c_r \geq 0$ such that for each $c > c_r$, every sufficiently large $cn$-regular $rn$-vertex balanced $r$-partite oriented graph has a Hamilton decomposition?
\end{question}

The question for $c_3$ is moot because any $cn$-regular balanced tripartite oriented graph on $3n$ vertices satisfies $c \leq 1$ but $\c T_{\Delta}$ is $n$-regular and not Hamilton decomposable. However, determining the threshold for an approximate decomposition in this setting would still be an interesting problem. For general $r$, we have $c_r\le 3r/8$ by \cite{kelly-published}.

\newcommand{\etalchar}[1]{$^{#1}$}

\end{document}